\documentclass[reqno]{amsart}

\usepackage{amsmath}
\usepackage{mathtools}
\usepackage{amssymb}
\usepackage{amsthm}
\usepackage{bbm}
\usepackage{dsfont}
\usepackage{mathrsfs}
\usepackage{commath}
\usepackage{youngtab}

\usepackage{graphicx}
\usepackage{pdfpages}

\DeclareMathOperator{\lcm}{lcm}

\newtheorem{theorem}{Theorem}[section]
\newtheorem{lemma}{Lemma}[section]
\newtheorem{proposition}{Proposition}[section]
\newtheorem{corollary}{Corollary}[section]

\theoremstyle{definition}
\newtheorem{definition}{Definition}[section]
\newtheorem{example}{Example}[section]

\theoremstyle{remark}
\newtheorem{remark}{Remark}[section]

\numberwithin{equation}{section}

\usepackage[margin=1.35in]{geometry}
\begin{document}

\title[Eigenvalue PP for Symmetric Group Permutation Representations]{The Eigenvalue Point Process for Symmetric Group Permutation Representations on $k$-tuples}

\author{Benjamin Tsou}
\address{Department of Mathematics, University of California, Berkeley, CA 94720-3840}
\email{benjamintsou@gmail.com}

\begin{abstract}
Equip the symmetric group $\mathfrak{S}_n$ with the Ewens distribution.  We study the eigenvalue point process of the permutation representation of $\mathfrak{S}_n$ on $k$-tuples of distinct integers chosen from the set $\{1,2,...,n\}$.  Taking $n \to \infty$, we find the limiting point process in the microscopic regime, i.e. when the eigenvalue point process is viewed at the scale of the mean eigenvalue spacing.  A formula for the limiting eigenvalue gap probability in an interval is also given.  In certain cases, a power series representation exists and a combinatorial procedure is given for computing the coefficients.   
 
\vspace{12pt}

\vspace{12pt}

\end{abstract}

\maketitle

\allowdisplaybreaks

\section{Introduction} \label{S:Intro}

Historically, much of the attention in random matrix theory has focused on the spectrum of continuous matrix ensembles such as random Hermitian matrices or compact Lie groups (e.g. $O(n)$ or $U(n)$) equipped with Haar measure.  Although less studied, there has been increasing interest in the asymptotic eigenvalue behavior of discrete matrix groups and in particular discrete (finite) subgroups of $U(n)$.  Perhaps the most natural family of such subgroups to consider are the $n \times n$ permutation matrices, which can be thought of as a (faithful) representation of the symmetric group $\mathfrak{S}_n$ in $U(n)$.  In \cite{diaconisshah}, Diaconis and Shahshahani note that the spectrum of uniformly distributed permutation matrices (thought of as a random probability measure) converges weakly to a point mass at the Haar measure on the unit circle.  Wieand \cite{wieand} has shown that eigenvalue fluctuations of uniformly distributed permutation matrices satisfy an asymptotically Gaussian law.  These results have been generalized via wreath products to more general families of finite groups.  In \cite{wieand2}, Wieand extends her limiting normality results to representations corresponding to the natural action of $(\mathbb{Z}_{K_1} \times... \times \mathbb{Z}_{K_M}) \wr \mathfrak{S}_n$ and $\mathbb{T}^M \wr \mathfrak{S}_n$ on $\mathbb{T}^M \times \{1,...,n\}$.  Evans \cite{evans} studied the eigenvalue distribution of the natural permutation representation of the $n$-fold wreath product $G \wr G... \wr G$ of a permutation group $G$ acting transitively on a set.

Rather than looking at representations arising from the natural permutation action of more complicated permutation groups, one can take a somewhat different perspective and instead consider higher dimensional representations of the symmetric group itself.  In particular, let $\rho_{n,k}$ denote the permutation representation of $\mathfrak{S}_n$ on $k$-tuples of distinct integers chosen from the set $[n] := \{1,...,n\}$ induced from the usual action of $\mathfrak{S}_n$ on $[n]$.  In \cite{tsou}, the author studied the fluctuation of eigenvalue statistics of $\rho_{n,k}$ (as well as a few related representations) in a given arc of the unit circle as $n \to \infty$ when $\mathfrak{S}_n$ is equipped with the Ewens measure and showed convergence to a compactly supported limiting distribution.  This corresponds to the macroscopic regime in which the arc being considered stays fixed and does not shrink as $n \to \infty$.  In this paper, we instead investigate the microscopic regime and characterize the limiting behavior of the eigenvalue point process of $\rho_{n,k}$ when viewed at the scale of the average eigenvalue spacing.  

To begin, we define the representation $\rho_{n,k}$ precisely.  Let $Q_{n,k}$ be the set of ordered $k$-tuples $(t_1,...,t_k)$ of distinct integers chosen from the set $[n]$.  The symmetric group $\mathfrak{S}_n$ acts naturally on this set by $\sigma(t_1,...,t_k) = (\sigma(t_1),...,\sigma(t_k))$.  We can form the $\frac{n!}{(n-k)!}$-dimensional $\mathbb{C}$-vector space $V_{n,k}$ with basis elements $e_{(t_1,...,t_k)}$.  Then the action of $\mathfrak{S}_n$ on $Q_{n,k}$ induces the permutation representation $\rho_{n,k}: \mathfrak{S}_n \rightarrow O(V_{n,k})$ where $O(V_{n,k})$ is the orthogonal group on $V_{n,k}$.

For $\theta > 0$, recall that the Ewens measure with parameter $\theta$ is given by 
\begin{equation} \label{ewensdef} \nu_{n, \theta}(\sigma) = \frac{\theta^{K(\sigma)}}{\theta(\theta+1)...(\theta+n-1)} 
\end{equation}  
where $K(\sigma)$ denotes the number of cycles of $\sigma$.  A random permutation chosen from this distribution is called $\theta$-biased.  Note that the uniform distribution corresponds to the case $\theta = 1$.  See \cite{ewens} for historical motivation of this measure from population genetics.

Endowing $\mathfrak{S}_n$ with the Ewens measure induces the corresponding eigenvalue point process of $\rho_{n,k}$.  When $k=1$, the eigenvalues of $\rho_{n,1}(\sigma)$ have a simple characterization.  Since all the eigenvalues are of the form $e^{2 \pi i \phi}$, it will be convenient to refer to each eigenvalue $e^{2 \pi i \phi}$ by its \textit{eigenangle} $\phi \in [0, 1)$.  For each $\sigma \in \mathfrak{S}_n$, let $C_j^{(n)}(\sigma)$ denote the number of cycles of $\sigma$ of length $j$.  For $\rho_{n,1}$, the set of eigenangles corresponding to a $j$-cycle of a permutation is $\left\{0, \frac{1}{j},...,\frac{j-1}{j} \right\}$.  Thus, the point process of eigenangles is given by 
\begin{equation} \Lambda_n = \sum_{j=1}^n C_j^{(n)} \sum_{i=0}^{j-1} \delta_{i/j} 
\end{equation}

It is not hard to see (as noted in \cite{diaconisshah}) that the empirical spectral measure $\Lambda_n/n$ converges weakly to the point mass at the Haar measure on the unit circle.  To obtain a scaled point process limit, we renormalize by a factor of $n$ and extend the process to the entire real line by noting that the angles are periodic with period 1.  Then we are led to consider the convergence of the point process \begin{equation} \mathfrak{E}_n^0 := \sum_{j=1}^n C_j^{(n)} \sum_{i \in \mathbb{Z}} \delta_{n i/j}
\end{equation}
as $n \to \infty$.  To clarify the sense of convergence that will be used, we briefly review some notions from the theory of point processes.  Daley and Vere-Jones \cite[Chapt. 11]{daley2} discuss various modes of convergence for point processes.  Here, we just need a few basic facts.

\begin{definition}
Let $\mathfrak{N}$ be the space of locally finite counting measures on $\mathbb{R}$ and let $\mathcal{N}$ be the smallest $\sigma$-algebra on $\mathfrak{N}$ with respect to which the mappings $N \mapsto N(A)$ are measurable for each $A \in \mathcal{B}(\mathbb{R})$.  Then a point process $\xi$ on $\mathbb{R}$ is a measurable map $\xi: \Omega \to \mathfrak{N}$ from a probability space $(\Omega, \mathcal{F}, \mathbb{P})$ to the measurable space $(\mathfrak{N}, \mathcal{N})$.  Every point process can be represented as $\displaystyle{\xi = \sum_{i=1}^N \delta_{X_i}}$ where $\delta$ denotes Dirac measure, $N$ is an integer-valued random variable, and $X_i$ are real-valued random variables.  The point process $\xi$ is called simple if the $X_i$ are a.s. distinct.
\end{definition}
The space of measures $\mathfrak{N}$ can be made into a completely separable metric space \cite[pp. 402-406]{daley1}.  Under the induced metric topology, $\mathcal{N}$ is the Borel $\sigma$-algebra $\mathcal{B}(\mathfrak{N})$.  Then weak convergence of point processes can be defined in the usual way: If $\mu_n$ and $\mu$ are probability measures on the measurable space $(\mathfrak{N}, \mathcal{B}(\mathfrak{N}))$ then $\mu_n \to \mu$ weakly iff $\displaystyle{\int f d\mu_n \to \int f d \mu}$ for all continuous and bounded functions $f$ on $\mathfrak{N}$.  

\begin{definition}
Given a point process $\xi$, a Borel set $A$ is a stochastic continuity set for the process if $\mathbb{P}\{\xi(\delta A) = 0 \} = 1$.  The sequence of point processes $\xi_n$ converges in the sense of convergence of fidi distributions if for every finite family $\{A_1,...,A_k\}$ of bounded continuity sets $A_i \in \mathcal{B}(\mathbb{R})$, the joint distributions of $\{\xi_n(A_1),...,\xi_n(A_k)\}$ converge weakly in $\mathcal{B}(\mathbb{R}^k)$ to the joint distribution of $\xi(A_1),...,\xi(A_k)$.  
\end{definition}
It is often easier to check the following conditions for weak convergence:
\begin{proposition}[c.f. {\cite[p. 137]{daley2}}] \label{fididist}
Weak convergence of point processes is equivalent to convergence of the fidi distributions.
\end{proposition}

\begin{proposition}[c.f. {\cite[p. 138]{daley2}}]
Let $\xi_n$ and $\xi$ be point processes.  Then $\xi_n \to \xi$ weakly iff we have the weak convergence of random variables $\displaystyle{\int_\mathbb{R} f d \xi_n \to \int_\mathbb{R} f d \xi}$ for all continuous and compactly supported $f \colon \mathbb{R} \to \mathbb{R}$. 
\end{proposition}

The Poisson-Dirichlet distribution $PD(\theta)$ plays a central role in the limit of $\mathfrak{E}_n^0$ and in general the (appropriately scaled) eigenvalue point process for $\rho_{n,k}$ as $n \to \infty$.  Perhaps the most intuitive way to describe the $PD(\theta)$ distribution is in terms of the $GEM(\theta)$ distribution.  Let $U_1, U_2,...$ be i.i.d. random variables with density $\theta (1-x)^{\theta-1}$ on $[0,1]$.  Define $V_j = U_j \prod\limits_{i=1}^{j-1}(1-U_i)$.  By the Borel-Cantelli Lemma, $\sum_j V_j =1$ with probability 1.  The distribution of $(V_1, V_2,...)$ is called the $GEM(\theta)$ distribution and there is a visually pleasing interpretation of this as a stick-breaking process.  Then the decreasing order statistics of $GEM(\theta)$ have the $PD(\theta)$ distribution.  For more information, see e.g. \cite[pp. 38-49]{billingsley}, \cite[Sect. 5.7]{a}, \cite{pitmanyor}, or \cite{feng}.    

Kingman \cite{kingman2} proved the following:

\begin{proposition}[Kingman] \label{PDconvergence}
Let the random variable $L_i^{(n)}$ denote the length of the $i^\textrm{th}$ longest cycle in a $\theta$-biased random permutation of size $n$.  Then \[n^{-1} (L_1^{(n)}, L_2^{(n)},...) \buildrel d \over \rightarrow (L_1,L_2,...) \] as $n \to \infty$ where the limiting random vector is distributed $PD(\theta)$.   
\end{proposition}

Now define \begin{equation}\label{e0} \mathfrak{E}_\infty^0 := \sum_{i=1}^\infty \sum_{q \in \mathbb{Z}} \delta_{q/L_i} \end{equation}
where $(L_1,L_2,...)$ is a random vector following the $PD(\theta)$ distribution.

\begin{remark}\label{GEMremark} We could replace the $PD(\theta)$-distributed process $(L_1,L_2,...)$ with a $GEM(\theta)$-distributed process $(V_1,V_2,...)$ in the definition of $\mathfrak{E}_\infty^0$ since this just corresponds to reordering the terms in the sum.
\end{remark}
Note that $\mathfrak{E}_\infty^0$ isn't quite a point process since it has an infinite Dirac mass at 0.  However, using Proposition \ref{PDconvergence} it is easy to see that $\mathfrak{E}_n^0 \to \mathfrak{E}_\infty^0$ weakly in the following sense: $\displaystyle{ \int f d\mathfrak{E}_n^0 \to \int f d\mathfrak{E}_\infty^0}$ weakly for all continuous, compactly supported $f$ such that $f(0) = 0$.  Note that if $f(0) > 0$, the distribution of $\int f d\mathfrak{E}_n^0$ converges weakly to the zero measure (i.e. mass escapes to $\infty$) and $\int f d\mathfrak{E}_\infty^0 = \infty$ almost surely. Najnudel and Nikeghbali \cite[Prop. 4.1]{najnudel} in fact show a.s. convergence when all the symmetric groups are put on the same probability space through a construction known as virtual permutations.

The point process $\mathfrak{E}_n^0$ can be thought of as the process $\Lambda_n$ ``zoomed in'' at 0 by a factor of $n$.  ``Zooming in'' on different ``windows'' of the point process $\Lambda_n$ can give very different limiting behaviors.  To be more precise, define for any real number $\alpha$ the shifted point process \begin{equation}\label{ea} \mathfrak{E}_n^\alpha := \sum_{j=1}^n C_j^{(n)} \sum_{i \in \mathbb{Z}} \delta_{n (i/j-\alpha)} \end{equation}
which corresponds to $\Lambda_n$ zoomed in at $\alpha$.
The number-theoretic properties of $\alpha$ play a very important role in determining the limiting process.  Recall that the irrationality measure $\mu(r)$ of a real number $r$ is given by 
\begin{equation} \label{defirrational}
\mu(r)= \inf \left\{ \lambda\colon \left\lvert r-\frac{p}{q}\right\rvert < \frac{1}{q^{\lambda}} \text{ has only finitely many integer solutions in } p  \text{ and } q \right\} 
\end{equation}

Then we can state the following convergence result, which we prove in Section \ref{S:k=1irrational}.

\begin{theorem}\label{T:1dmain}
Let $U_1,U_2,...$ be i.i.d. random variables distributed uniformly on the interval $[0,1]$ and independent of the $PD(\theta)$ process $(L_1,L_2,...)$.  Then define the point process \begin{equation}\label{estar} \mathfrak{E}_\infty^* := \sum_{i=1}^\infty \sum_{q \in \mathbb{Z}} \delta_{(U_i + q)/L_i} \end{equation} If $\alpha \in (0,1)$ is an irrational number with finite irrationality measure, then $\displaystyle{ \mathfrak{E}_n^\alpha \to \mathfrak{E}_\infty^*}$ weakly.
\end{theorem}

\begin{remark}
Najnudel and Nikeghbali \cite{najnudel} have studied the limiting eigenvalue distributions of ``randomized'' permutation matrices in which the 1's are replaced by i.i.d. random variables taking values in $\mathbb{C}^*$.  In the latter half of the paper, they study the eigenvalue point processes for these ensembles in the microscopic regime.  Among other results, they show that if the 1's are replaced by uniformly distributed entries on the complex unit circle, the limiting eigenangle process is also $\mathfrak{E}_\infty^*$.  Bahier \cite{bahier} has recently proved asymptotic normality results on the number of points of the limiting point processes $\mathfrak{E}_\infty^0$ and $\mathfrak{E}_\infty^*$ lying in an interval of size tending to infinity. It is easy to see that $\mathfrak{E}_\infty^*$ is a stationary simple point process with intensity 1, i.e. $\mathbb{E}[\mathfrak{E}_\infty^*]$ is Lebesgue measure on $\mathbb{R}$. 
\end{remark}

In Theorem \ref{T:k=1rational}, we show that at rational angles $\alpha=s/t$ where $s$ and $t$ are relatively prime, $\mathfrak{E}_n^\alpha$ converges weakly to some limit $\mathfrak{E}_{\infty}^t$.  The point process $\mathfrak{E}_{\infty}^t$ is defined essentially like $\mathfrak{E}_\infty^*$ except that the $U_i$ are discrete random variables uniform over the $t$ fractions $1/t,2/t,...,1$ instead of the interval $[0,1]$.  

Now let us consider the eigenvalue point process for general $k$.  For $\sigma \in \mathfrak{S}_n$, let $\sigma_k$ be the permutation corresponding to $\rho_{n,k}(\sigma)$ and $C_{j,k}^{(n)}(\sigma)$ be the number of $\sigma_k$-cycles of length $j$.  Note that $C_{j,k}^{(n)}(\sigma) > 0$ iff there exist $i_1,...,i_k$ (not necessarily distinct) such that $\prod_{l=1}^k C_{i_l}^{(n)}(\sigma) > 0$ and $\lcm(C_{i_1}^{(n)}(\sigma),...,C_{i_k}^{(n)}(\sigma)) = j$.  Then the eigenangle point process of $\rho_{n,k}$ is 
\begin{equation}
\Lambda_{n,k} = \sum_{j} C_{j,k}^{(n)} \sum_{i=0}^{j-1} \delta_{i/j}
\end{equation}
As for the $k = 1$ case, it is easy to see that $\Lambda_{n,k}/n^k$ converges weakly to the point mass at Haar measure on the unit circle.  
To obtain a scaled point process limit, we renormalize by a factor of $n^k$ and extend the process to the real line by periodicity.  Then we have the point process \begin{equation}
\mathfrak{E}_{n,k}^0 :=  \sum_{j} C_{j,k}^{(n)} \sum_{i \in \mathbb{Z}} \delta_{n^k i/j}
\end{equation}

As $n \to \infty$, it is not hard to see that the number-theoretic properties of the largest cycle lengths of $\sigma \in \mathfrak{S}_n$ play a dominant role in the behavior of $\mathfrak{E}_{n,k}^0$.  Thus, we introduce the following infinite array $X_{mi}$ of random variables giving the limiting prime factor distribution for a random integer in $[N]$ as $N \to \infty$. 
\begin{definition}

Let $p_m$ be the $m^\textrm{th}$ prime.  For positive integers $m$ and $i$, (by Kolmogorov's extension theorem) define $X_{m i}$ to be independent geometrically distributed random variables with parameter $1-1/p_m$, i.e. 
\begin{equation} \label{Xmi} \mathbb{P}\{X_{m i} = a\} = (1-1/p_m)(1/p_m)^{a} 
\end{equation} 
Let $\textbf{X}_l^{i_1,...,i_k}$ denote the $l \times k$ matrix consisting of entries $X_{m, i_j}$ for $1 \le m \le l$ and $1 \le j \le k$ as defined in \eqref{Xmi}.  We allow $l$ to be infinite.  
\end{definition}

 We will also need:

\begin{definition}
For $k > 0$, let $\mathbf{e}$ be an $\infty \times k$ matrix whose entries $\mathbf{e}_{mi}$ for $1 \le m < \infty$ and $1 \le i \le k$ are non-negative integers.  Then define the function 
\begin{equation}\label{gk} g_{k}(\mathbf{e}) := \prod_{m=1}^\infty \frac{p_m^{e_{m 1} + ...+ e_{m k}}}{p_m^{\max(e_{m1},...,e_{mk})}} 
\end{equation} 
We allow $g_k$ to take finite matrix arguments by filling in the remaining entries with $0$'s. 
\end{definition}
By the Borel-Cantelli lemma, $\prod\limits_{m=1}^\infty p_m^{\min(X_{m, i}, X_{m, j})}$ is a.s. finite.  Then $g_{k}(\textbf{X}_\infty^{i_1,...,i_k})$ is finite for all $k$ and sequences $(i_1,...,i_k)$ almost surely.  

Let the $PD(\theta)$ process $L_1,L_2,...$ be independent from all the $X_{mi}$ defined in \eqref{Xmi}.  Then we can define
\begin{equation} \mathfrak{E}_{\infty, k}^0 := k! \sum_{i_1 <...< i_k} g_{k}(\textbf{X}_\infty^{i_1,...,i_k}) \sum_{q \in \mathbb{Z}} \delta_{g_{k}(\textbf{X}_\infty^{i_1,...,i_k}) q /(L_{i_1}...L_{i_k}) } 
\end{equation}
The weak convergence $\mathfrak{E}_{n,k}^0 \to \mathfrak{E}_{\infty, k}^0$ will follow as a special case of Theorem \ref{T:rationalgeneral}.   

As for the $k = 1$ case, we define for each real number $\alpha$ the shifted point process \begin{equation} \label{kalpha} \mathfrak{E}_{n, k}^\alpha :=  \sum_{j} C_{j,k}^{(n)} \sum_{i \in \mathbb{Z}} \delta_{n^k (i/j - \alpha)} \end{equation} and consider $\mathfrak{E}_{n, k}^\alpha $ for rational and irrational $\alpha$ separately.  Theorem \ref{T:generalmain} below is the main result of the paper.  Section \ref{S:Factorization classes} discusses factorization classes and multi-dimensional discrepancy and well-distribution, which are used in the proof.  In Section \ref{S:main theorem}, we prove the theorem assuming the result in Theorem \ref{T:technical}.  Sections \ref{S:boundsmallcycle} and \ref{S:logdisc} are devoted to the proof of Theorem \ref{T:technical}, which addresses the primary technical difficulty: showing that only the largest cycles in $\sigma$ make non-negligible contributions to the point process in the limit $n \to \infty$.

\begin{theorem}\label{T:generalmain}

Let $U_{i_1,...,i_k}$ for $\displaystyle{i_1 < ...< i_k}$ be i.i.d. random variables distributed uniformly on $[0,1]$ and independent of the $PD(\theta)$ process $(L_1,L_2,...)$.  Let the variables $X_{mi}$ from \eqref{Xmi} be independent of both $U_{i_1,...,i_k}$ and $(L_1,L_2,...)$.  Then define the point process \begin{equation} \label{klimit} \mathfrak{E}_{\infty, k}^* := k!
\sum_{i_1 <...< i_k} g_{k}(\textbf{X}_\infty^{i_1,...,i_k}) \sum_{q \in \mathbb{Z}} \delta_{(q + U_{i_1,...,i_k}) g_{k}(\textbf{X}_\infty^{i_1,...,i_k})  /(L_{i_1}...L_{i_k}) } \end{equation}  If $\alpha \in (0,1)$ is an irrational number with finite irrationality measure, then $\mathfrak{E}_{n, k}^\alpha \to \mathfrak{E}_{\infty,k}^*$ weakly. 
\end{theorem}
The analog of Theorem \ref{T:generalmain} for rational $\alpha$ is stated in Theorem \ref{T:rationalgeneral}. 
\begin{remark} \label{kintensity}
As before, the $PD(\theta)$-distributed vector $(L_1, L_2,...)$ could be replaced by a $GEM(\theta)$-distributed vector $(V_1, V_2,...)$ without changing the distribution of $\mathfrak{E}_{\infty,k}^*$.  The point process $\mathfrak{E}_{\infty,k}^*$ is clearly stationary.  If $\theta = 1$, the mean measure is  \begin{equation} \mathbb{E}[\mathfrak{E}_{\infty,k}^*] = k! \sum_{i_1 < ... < i_k} \mathbb{E}[L_{i_1}...L_{i_k}] \lambda = \frac{1}{k!} \lambda 
\end{equation}
by equation \eqref{summoments} where $\lambda$ denotes Lebesgue measure on $\mathbb{R}$.  Note that for $k > 1$ the mean intensity is less than 1, which reflects the fact that mass escapes to infinity as $\mathfrak{E}_{n, k}^\alpha \to \mathfrak{E}_{\infty,k}^*$.
\end{remark}

\begin{remark}
Numbers that have infinite irrationality measure are called Liouville numbers and are known to have Hausdorff dimension 0.  It seems unlikely that any nice limiting behavior exists for $\alpha$ an arbitrary Liouville number.  
\end{remark}
In contrast to the $\rho_{n,k}$ representations, recall that the $U(n)$ eigenvalue point processes are translation invariant and the limiting process when zooming in at any point on the unit circle is a determinantal point process characterized by the Dyson sine kernel. 

Finally, we use Theorem \ref{T:generalmain} to compute the limiting eigenvalue gap probability for $\mathfrak{E}_{n, k}^\alpha$ as $n \to \infty$.

\begin{definition}
Let $\alpha$ be irrational with finite irrationality measure.  If $\sigma \in \mathfrak{S}_n$ is a randomly chosen permutation from the Ewens measure with parameter $\theta$, then let $P^\theta_k(y_1,y_2)$ denote the limit as $n \to \infty$ of the gap probability that there are no eigenangles of $\rho_{n,k}(\sigma)$ in the interval $\displaystyle{ \left(\alpha + \frac{y_1}{n^k}, \alpha + \frac{y_2}{n^k} \right)}$.  
\end{definition}

By Proposition \ref{fididist}, the limiting eigenvalue gap probability is the gap probability of the limiting process $\mathfrak{E}_{\infty,k}^*$.  Since the $U_{i_1,...,i_k}$ are i.i.d. random variables distributed uniformly on $[0,1]$, it is easy to see that conditional on the $PD(\theta)$ process $(L_1,L_2,...)$ and independent variables $X_{mi}$, the gap probability is \[\lim_{r \to \infty} \prod_{1 \le i_1 <...< i_k \le r } \bigg(1 - \min \Big(\frac{(y_2-y_1) \prod_{u=1}^k L_{i_u}}{g_k(\mathbf{X}_\infty^{i_1,...,i_k})}, 1 \Big) \bigg) \]
The limit clearly exists since this is a decreasing sequence in $r$.  Taking the expectation, we get

\begin{corollary}[Formula for eigenvalue gap probability] \label{C:eigengapformula}
Let the random variables $X_{mi}$ be independent from the $PD(\theta)$ process $(L_1,L_2,...)$.  Then the eigenvalue gap probability is given by
\begin{equation}
P^\theta_k(y_1,y_2) = \mathbb{E}\bigg[ \prod_{i_1 <...< i_k } \bigg(1 - \min \Big(\frac{(y_2-y_1) \prod_{u=1}^k L_{i_u}}{g_k(\mathbf{X}_\infty^{i_1,...,i_k})}, 1 \Big) \bigg) \bigg]
\end{equation}
\end{corollary}

Note that $\displaystyle{\prod_{u=1}^k L_{i_u} < \frac{1}{k^k}}$ almost surely.  Therefore, if $\displaystyle{y_2 - y_1 \le k^k}$, we can remove the $\min$ and get 
\begin{equation}\label{gapformula}
P^\theta_k(y_1,y_2) = \mathbb{E}\bigg[ \prod_{i_1 <...< i_k } \bigg(1 + \frac{(y_1-y_2) \prod_{u=1}^k L_{i_u}}{g_k(\mathbf{X}_\infty^{i_1,...,i_k})} \bigg) \bigg]
\end{equation}

Expanding out the product, $P^\theta_k(y_1,y_2)$ has a power series representation in $y_1 - y_2$.  For $\theta=1$, nicer formulas can be obtained and we will provide a combinatorial procedure for computing the series coefficients in Section \ref{S:power}. 

In the next section we begin by discussing joint cycle length statistics, which play an important role in understanding the eigenvalue point processes.  In essence, we want to sum over the events  $(C_1^{(n)},...,C_n^{(n)}) = (c_1,...,c_n)$ such that $\sum_{i=1}^n i c_i = n$ and determine when the $\sigma_k$-cycle of length $\lcm(C_{i_1},...,C_{i_k})$ contains an eigenangle in a given interval.  Thus, the proofs of Theorems \ref{T:1dmain} and \ref{T:generalmain} involve combining asymptotic point probability estimates of cycle length distributions with probabilistic number theory and Diophantine approximation methods.
  
\section{Joint Cycle Length Statistics} \label{S:JointCycleLengths}

We borrow much of the notation in this section from Arratia, Barbour, and Tavare's book \cite{a}.  This monograph contains essentially all the material we need about joint cycle length distributions of permutations drawn from the Ewens distribution.  

The Conditioning Relation \cite[(4.1)]{a} states that the distribution of cycle lengths $C_j^{(n)}$ behaves the same as independent Poisson random variables conditioned on the necessary restriction that the size of the permutation is $n$.  This relation in fact holds for a broad class of classical combinatorial structures such as mappings, partitions, and trees.  See \cite{a} for details.  To state the Conditioning Relation formally, we define the following random variables.
\begin{definition} \label{T_{0n}}
Let $Z_j$ be independent Poisson variables with mean $\theta/j$.  Then we define \begin{equation} T_{bn} := \sum_{j=b+1}^n j Z_j \end{equation}  For any subset $B \subset [n]$, we define \begin{equation} T_B := \sum_{j \in B} j Z_j \end{equation}
\end{definition}

\begin{proposition}[Conditioning Relation]
 For any vector $c = (c_1,...,c_n) \in \mathbb{Z}_+^n$, we have \begin{equation} \label{Conditioning Relation}
\mathbb{P}\{C_1^{(n)} = c_1,...,C_n^{(n)} = c_n \} = \mathbb{P}\{Z_1 = c_1,..., Z_n=c_n \mid T_{0n} = n\}
\end{equation} 
\end{proposition}

Using the Conditioning Relation, one can compute

\begin{proposition}[Ewens Sampling Formula]
Under the Ewens measure with parameter $\theta$,
\begin{equation} \mathbb{P}\{ (C_1^{(n)},...,C_n^{(n)}) = (b_1,...,b_n) \} = \frac{n!}{\theta(\theta+1)...(\theta+n-1)} \prod_{j=1}^n \bigg( \frac{\theta}{j} \bigg)^{b_j} \frac{1}{b_j!} \mathbbm{1}\bigg(\sum_{i=1}^n i b_i = n \bigg) 
\end{equation} 
\end{proposition}

By \cite[Thm. 4.6]{a}, $\displaystyle{n^{-1} T_{0n} \buildrel d \over \rightarrow X_\theta}$ with Laplace transform given by 
\begin{equation} \label{Xtheta} \mathbb{E}[e^{-s X_\theta}] = \exp \left( - \int_0^1 (1 - e^{-sx}) \frac{\theta}{x} dx \right)
\end{equation} and $\mathbb{E}[X_\theta] = \theta$.

Equation \eqref{Xtheta} can be used to find the continuous density function $p_\theta$ of $X_\theta$ for $x > 0$ (see \cite[Lemma 4.7]{a}).  We won't need the exact form in the sequel.  However, for  $0 \le x \le 1$, we have the simple expression (see \cite[Cor. 4.8]{a})
\begin{equation}
p_\theta(x) = \frac{e^{-\gamma \theta} x^{\theta - 1}}{\Gamma(\theta)} 
\end{equation}

Now we can state the Poisson-Dirichlet density (c.f. \cite[(5.40)]{a}).  Here, $\Gamma(z)$ is the Gamma function and $\gamma$ is the Euler-Mascheroni constant.  

\begin{proposition} \label{PD}
Let $(L_1, L_2,...)$ follow the Poisson-Dirichlet distribution $PD(\theta)$.  Then the density of the finite-dimensional distributions $(L_1,...,L_r)$ are given by \begin{equation} \label{PDthetadensity} f_\theta^{(r)}(x_1,...,x_r) = \frac{e^{\gamma \theta} \theta^r \Gamma(\theta) x_r^{\theta-1}} {x_1...x_r} p_\theta \left(\frac{1-x_1-...-x_r}{x_r} \right) 
\end{equation}
and supported in the region given by $0 < x_r < ...< x_1 < 1$ and $x_1 + ... + x_r < 1$.  
\end{proposition}

The following uniform local limit theorem for the large cycle distributions will play a fundamental role in our proofs.  
We follow the proof in the local limit theorem \cite[Thm. 5.10]{a}.  There, pointwise convergence for each choice of $x_i$ is established.  For our purposes, this needs to be strengthened to uniform convergence.      

\begin{lemma}\label{L:uniformbound}
Let $0 < \lambda < 1$ and $r \ge 1$.  For $1 \le i \le r$, let $\Delta$ be the set of integer tuples $(m_1,...,m_r)$ such that $\lambda n < m_r < m_{r-1} <...< m_1 \le n$ and $0 < m = m_1+...+ m_r < (1 - \lambda)n$.  Set $\displaystyle{x_i = \frac{m_i}{n}}$.  Then for each $\varepsilon > 0$ and for sufficiently large $n$, \[ \left|n^r \mathbb{P}\{L_i^{(n)} = m_i, 1 \le i \le r\} - f_\theta^{(r)}(x_1,...,x_r)\right| < \varepsilon \] for all $(m_1,...,m_r) \in \Delta$.

\end{lemma}  

\begin{proof}

Let $\displaystyle{A_n(C^{(n)}) = A_n(C^{(n)}; m_1,...,m_{r})}$ denote the event such that $C_{m_j}^{(n)}=1$ for $1 \le j \le r-1$ and $C_i^{(n)}=0$ for all other $m_r+1 \le i \le n$.  Then
\begin{equation} \label{An1}
\mathbb{P}\{L_i^{(n)} = m_i, 1 \le i \le r\} = \mathbb{P}\{A_n(C^{(n)}), C_{m_r}^{(n)} = 1\} + \sum_{ l \ge 2} \mathbb{P}\{ A_n(C^{(n)}), C_{m_r}^{(n)} = l\}
\end{equation}  
As shown in \cite[Thm. 5.10]{a}, the second term is $o(m_r^{-r}) = o(n^{-r})$.  
By the Conditioning Relation \eqref{Conditioning Relation}, the first term is \begin{equation}\label{An2} \mathbb{P}\{A_n(Z), Z_{m_r} = 1 \mid T_{0n} = n\} = \mathbb{P}\{A_n(Z)\} \mathbb{P}\{Z_{m_r} = 1\} \frac{\mathbb{P}\{T_{0, m_r-1} = n-m\}}{\mathbb{P}\{T_{0n} = n\}} \end{equation} which reduces to \begin{equation} \label{An3} \frac{\theta^r e^{-\theta(h(n+1)-h(m_r))}}{m_1...m_r} \frac{\mathbb{P}\{T_{0, m_r-1} = n-m\}}{\mathbb{P}\{T_{0n} = n\}}  \end{equation} 
where $\displaystyle{h(n) = 1 + \frac{1}{2}+...+\frac{1}{n-1}}$.

The size biasing equation \cite[(4.8)]{a} states that for any $B \subset [n]$, we have
\begin{equation} \label{sizebiasewens} m \mathbb{P}\{T_B = m\} = \theta \sum_{l \in B} \mathbb{P}\{T_B = m-l\}
\end{equation}  

Then by \eqref{sizebiasewens}, 
\begin{equation}
(n-m) \mathbb{P}\{T_{0, m_r-1} = n-m\} =
\theta \mathbb{P} \left\{\frac{n-m}{m_r}-1 \le \frac{T_{0, m_r-1}}{m_r} < \frac{n-m}{m_r} \right\}
\end{equation}
and we see that the local point probabilities can be expressed in terms of interval probabilities.  

If distribution functions $F_n$ converge weakly to a continuous distribution $F$, then the convergence is in fact uniform on the real line.  This applies to the convergence $\displaystyle{\frac{T_{0n}}{n} \buildrel d \over \to X_\theta}$.  By \cite[(4.23)]{a}, $\displaystyle{xp_\theta(x) = \theta \int_{x-1}^x p_\theta(u)du}$.  Thus, for any $\varepsilon > 0$ and for sufficiently large $n$, \[  \left|\frac{m_r}{n-m} \theta \mathbb{P} \bigg\{ \frac{n-m}{m_r}-1 \le \frac{T_{0, m_r-1}}{m_r} < \frac{n-m}{m_r} \bigg\} - p_\theta\left(\frac{1-x_1-...-x_r}{x_r}\right) \right| < \varepsilon \] 

The uniform bound follows from the fact that $\displaystyle{ \lim_{ n\to \infty} n\mathbb{P}\{ T_{0n} = n\} = p_\theta(1) = \frac{e^{-\gamma \theta} }{\Gamma(\theta)} }$.

\end{proof}

\section{Eigenvalue process for random permutation matrices at angles of finite irrationality measure} \label{S:k=1irrational}
  
In this section, we prove Theorem \ref{T:1dmain}.  First, we introduce some notation for the discrepancy of a sequence.

\begin{definition}\label{multidiscrepancy}
Let $J$ be the set of $d$-dimensional boxes of the form $\displaystyle{\prod_{i=1}^d [a_i, b_i] }$ where $0 \le a_i < b_i \le 1$ and let $\mathcal{P}_n$ be a multiset or sequence of $n$ elements of the $d$-dimensional torus $\mathbb{T}^d$. 
The multidimensional discrepancy is given by \[D(\mathcal{P}_n) := \sup_{B \in J} \bigg|\frac{A(B; n)}{n} - \prod_{i=1}^d (b_i - a_i) \bigg| \] where $A(B; n)$ counts the number of elements of $\mathcal{P}_n$ in the box $B$.  
\end{definition} 

\begin{definition}
If $\omega$ is an infinite sequence, let $D_{i,n}(\omega)$ be the discrepancy of the $(i+1)^\textrm{st}$ through $(i+n)^\textrm{th}$ terms of the sequence.  
\end{definition}

We will need the following discrepancy bound:  
\begin{proposition}[c.f. {\cite[Thm. 2.3.2]{kuipers}}] \label{P:1Ddisc}
Let $\alpha$ have irrationality measure $\mu$ and let $x_n = n \alpha + \beta$.  Then for every $\varepsilon > 0$,
\begin{equation} D_{i,n}(x) = O(n^{-\frac{1}{\mu-1} + \varepsilon})
\end{equation}     
\end{proposition}

\begin{definition}
For any real number $\alpha$, let $\psi_\alpha(j)$ be the minimum positive element of the set $\{q - j \alpha : q \in \mathbb{Z} \}$.

\end{definition}

The following lemma shows that in the limit, the eigenangles associated to the largest cycles are randomly situated in the window centered at $\alpha$.

\begin{lemma}\label{L:uni}
Let $\alpha \in \mathbb{R}$ be irrational and $(L_1,L_2,...)$ be a $PD(\theta)$ process.  Then for any positive integer $r$ and i.i.d. random variables $U_i$ chosen independently from the vector $(L_1,L_2,...)$ and distributed uniformly on $[0,1]$, we have the distributional convergence \[(L_1^{(n)}/n,...,L_r^{(n)}/n, \psi_\alpha(L_1^{(n)}),...,\psi_\alpha(L_r^{(n)})) \buildrel d \over \to (L_1,...,L_r,U_1,...,U_r)\]
\end{lemma}

\begin{proof} 
Define the intervals 
\begin{equation}\label{bin} I_{m, b} = ( (m-1) n/b, m n/b ] \cap \mathbb{Z} 
\end{equation}  This defines a partition $[n] = \bigcup\limits_{m=1}^b I_{m,b}$.  For each $\lambda > 0$, $f_\theta^{(r)}(x_1,...,x_r)$ is uniformly continuous on the set of tuples \[\{(x_1,...,x_r) : (\lambda < x_r \le...\le x_1 < 1) \land (x_1+...+x_r < 1-\lambda) \}.\]  Then by Lemma \ref{L:uniformbound}, for $\displaystyle{\lambda b \le k_r \le ... \le k_1 }$ and $\displaystyle{k_1 +...+ k_r \le (1 - \lambda)b }$ and $m_i \in I_{k_i, b}$ distinct, we have the bound \begin{equation} \label{ub} \bigg|f_\theta^{(r)} \left(\frac{k_1}{b}, ...,\frac{k_r}{b}\right) - n^r \mathbb{P}\{L_i^{(n)} = m_i, 1 \le i \le r\} \bigg| < \varepsilon(n,b) \end{equation} for some error function $\varepsilon(n,b) \to 0$ as $n,b \to \infty$.  Let $\displaystyle{\varepsilon^{*}(n) = \mathbb{P}\bigg( \bigcup_{i=1}^r \{C_{L_i^{(n)}}^{(n)} > 1\} \bigg) }$, i.e. the probability that there is a duplicate cycle length among the $r$ largest cycles.  Clearly, $\varepsilon^*(n) \to 0$ as $n \to \infty$.  

Since $f_\theta^{(r)}(x_1,...,x_r)$ integrates to 1, for every $\lambda > 0$ we have \[\limsup_{n \to \infty} \mathbb{P}\{ L_{r}^{(n)} < \lambda n\} + \mathbb{P}\{L_1^{(n)} +...+L_r^{(n)} > (1- \lambda)n \} < \varepsilon(\lambda) \] for some error function $\varepsilon(\lambda)$ tending to 0 as $\lambda \to 0$.  Finally, note that $\psi_\alpha(j) \le x$ iff $\{j \alpha\} \ge 1-x$.  
Then for $c_i, d_i \in [0,1]$,
\begin{align*}
&\qquad \mathbb{P}\bigg(\bigcap_{i=1}^r \{L_i^{(n)}/n < c_i\} \cap \{\psi_\alpha(L_i^{(n)}) < d_i\} \bigg) \\
& \le \sum_{\substack{1 \le m_r < ...< m_1 \le n \cr m_i/n < c_i}} \mathbb{P}\bigg(\bigcap_{i=1}^r \{L_i^{(n)} = m_i\} \cap \{\psi_\alpha(m_i) < d_i\} \bigg) + \varepsilon^*(n) \\
& \begin{multlined} \le \sum_{\substack{\lambda b \le k_r \le...\le k_1 \le b \cr k_1 +...+ k_r \le (1 - \lambda)b \cr (k_i-1)/b < c_i}}  \bigg(f_\theta^{(r)} \left(\frac{k_1}{b}, ...,\frac{k_r}{b}\right) + \varepsilon(n,b) \bigg) \frac{1}{b^r} \prod_{i=1}^r \left(d_i + D_{\frac{(k_i-1) n}{b}, \frac{n}{b}}(\{j\alpha\}) \right) \\ + \mathbb{P}\{ L_{r}^{(n)} < \lambda n\} + \mathbb{P}\{L_1^{(n)} +...+L_r^{(n)} > (1- \lambda)n \} + \varepsilon^*(n)  
\end{multlined} 
\end{align*}
Taking $b = o(n)$ so that $n/b \to \infty$,  \[\limsup_{n \to \infty} \mathbb{P}\bigg(\bigcap_{i=1}^r \{L_i^{(n)}/n < c_i\} \cap \{\psi_\alpha(L_i^{(n)}) < d_i\} \bigg) \le  \int\limits_{\substack{\lambda < x_r <...< x_1 < 1 \cr x_1+...+x_r < 1- \lambda \cr x_i < c_i}} f_\theta^{(r)}(x_1,..,x_r) dx_i \prod_{i=1}^r d_i + \varepsilon(\lambda) \]
Similarly, \[\liminf_{n \to \infty} \mathbb{P}\bigg(\bigcap_{i=1}^r \{L_i^{(n)}/n < c_i\} \cap \{\psi_\alpha(L_i^{(n)}) < d_i\} \bigg) \ge  \int\limits_{\substack{\lambda < x_r <...< x_1 < 1 \cr x_1+...+x_r < 1- \lambda \cr x_i < c_i}} f_\theta^{(r)}(x_1,..,x_r) dx_i \prod_{i=1}^r d_i \]
Taking $\lambda \to 0$ completes the proof.
\end{proof}

Now we are ready to prove the eigenvalue point process convergence for the $k=1$ case of random permutation matrices at irrational $\alpha$ with finite irrationality measure.

\begin{proof}[Proof of Theorem \ref{T:1dmain}]
Let $r > 1$ be a positive integer.  
To isolate the contribution from the $r$ largest cycles, we define \begin{equation}\label{xinr} \xi_n^r = \sum_{i=1}^{r} \sum_{q \in \mathbb{Z}} \delta_{n (q/L_i^{(n)} -\alpha)} \end{equation}
and we let $\displaystyle{\eta_n^r = \mathfrak{E}_n^\alpha - \xi_n^r}$ denote the remaining measure.  Similarly, define \begin{equation}\label{xiinftyr} \xi_\infty^r = \sum_{i=1}^r \sum_{q \in \mathbb{Z}} \delta_{(U_i + q)/L_i} \end{equation} and set $\eta_\infty^r = \mathfrak{E}_\infty^* - \xi_\infty^r$.   

Let $f$ be a continuous function with support contained in some interval $(-T, T)$.  Then \[\int f d \xi_n^r =  \sum_{i=1}^r \sum_{q \in \mathbb{Z}} f \bigg(\frac{q - L_i^{(n)} \alpha}{L_i^{(n)}/n} \bigg) = \sum_{i=1}^r \sum_{q \in [-T+1, T+1]} f \bigg(\frac{q + \psi_\alpha(L_i^{(n)}) }{L_i^{(n)}/n} \bigg)  \]
and \[\int f d\xi_\infty^r = \sum_{i =1 }^r \sum_{q \in [-T+1, T+1]} f \left(\frac{U_i+q}{L_i} \right) \]

By Lemma \ref{L:uni} and the continuous mapping theorem, \[\int f d \xi_n^r \buildrel d \over \to \int f d\xi_\infty^r \]

Now we turn to the remainder measures $\eta_{n}^r$ and $\eta_\infty^r$.  Since $\mathbb{E}[\sum_i L_i] = \mathbb{E}[1] = 1$, by the Borel-Cantelli lemma $\displaystyle{\int f d \eta_\infty^r \buildrel a.s. \over \to 0}$.  Note that this just verifies that $\mathfrak{E}_\infty^*$ is in fact a point process.  

For $\eta_n^r$, we need to show that 
\begin{equation}\label{1Dbound} \lim_{r \to \infty} \limsup_{n \to \infty} \mathbb{P} \bigg\{ \int |f| d\eta_n^r  > 0 \bigg\} = 0
\end{equation}
  
Note that $L_i^{(n)} < n/r$ for $i > r$ with probability 1.  Thus, we have the inequality of measures
\[ \eta_n^r < \sum_{j=1}^{\lceil n/r \rceil} C_j^{(n)} \sum_{i \in \mathbb{Z}} \delta_{n (i/j-\alpha)} \]
Define the set \begin{equation}\label{J(n)} J(n) = \{j \in [n]: (\exists q \in \mathbb{Z})[\alpha - T/n < q/j < \alpha + T/n] \}\end{equation}  Specializing Watterson's factorial moment formula \cite{watterson} gives 
\begin{equation} \mathbb{E}[C_j^{(n)}] = \frac{\theta}{j} \prod_{i=1}^{j-1} \frac{n-i}{\theta+n-i-1}  \le \frac{\theta}{j} \frac{n}{n-j+\theta} 
\end{equation}
Then
\begin{equation*} \mathbb{P} \bigg\{ \int |f| d\eta_n^r  > 0 \bigg\} \le \mathbb{P}\bigg( \bigcup_{j \in J \cap [ \lceil n/r \rceil ]} \{C_j^{(n)} > 0\} \bigg) \le \sum_{j \in J \cap [ \lceil n/r \rceil]}  \mathbb{E}(C_j^{(n)}) \le \sum_{j \in J \cap [\lceil n/r \rceil]} \frac{2\theta}{j} \end{equation*}

Looking at each interval $I_{m,b}$ (defined in \eqref{bin}) separately, 
\begin{equation} \label{Jineq}
|J \cap I_{m,b}| \le \frac{n}{b} \Big(\min \left(\frac{2mT}{b}, 1\right) + D_{\frac{(m-1)n}{b},\frac{n}{b}}(\{j \alpha\}) \Big)  
\end{equation}   

Recall $\alpha$ has some finite irrationality measure $\mu(\alpha)$.  Now set $\displaystyle{b = n^{1 - 1/(\mu(\alpha)+1)}}$.  Then the first interval $I_{1,b}$ contains the integers $0 < j \le n^{1/(\mu(\alpha)+1)}$.  Setting $\varepsilon = 1/2$ (say), we have $\displaystyle{\left\lvert \alpha-\frac{p}{j}\right\rvert > \frac{C}{n^{1 - 1/(2\mu(\alpha)+2)}}}$ for some absolute constant $C$ and all $j \in I_{1,b}$.  
Since the length of the interval $\displaystyle{( \alpha - T/n, \alpha +  T/n) }$ is $O(1/n)$, the intersection $J \cap I_{1,b}$ is empty for sufficiently large $n$.  Then 
\begin{equation} \label{Jfrac}
\sum_{j \in J \cap [\lceil n/r \rceil]} \frac{1}{j} \le \sum_{m=2}^{\lceil b/r \rceil} \frac{b}{(m-1) n} \frac{n}{b}\Big(\min \left(\frac{2mT}{b}, 1 \right) + D_{\frac{(m-1)n}{b},\frac{n}{b}} \left(\{j \alpha\} \right) \Big) 
\end{equation}

By Proposition \ref{P:1Ddisc}, if $1/r < 1/(2T)$ we have 
\begin{equation} \limsup_{n \to \infty} \sum_{j \in J \cap [\lceil n/r \rceil]} \frac{1}{j} \le 4 T/r 
\end{equation} 
and \eqref{1Dbound} follows.
Recall we want to show \[\int f d \mathfrak{E}_n^\alpha \buildrel d \over \to \int f d \mathfrak{E}_\infty^*. \]
For ease of notation, set $X_n = \int f d \mathfrak{E}_n^\alpha$, $X_\infty = \int f d \mathfrak{E}_\infty^*$, and $X_n^r = \int f d \xi_n^r$, $X_\infty^r = \int f d\xi_\infty^r$.  Finally, let $\varepsilon_n^r = \int f d \eta_n^r$.  Then it is sufficient to show that $\displaystyle{\lim_{n \to \infty} \mathbb{E}[g(X_n)] = \mathbb{E}[g(X_\infty)]}$ for all continuous and compactly supported $g$.  

Since $X_n^r \buildrel d \over \to X_\infty^r$ and $X_\infty^r \buildrel d \over \to X_\infty$, we already have $\displaystyle{ \mathbb{E}[g(X_n^r)] \to \mathbb{E}[g(X_\infty^r)]}$ and $\mathbb{E}[g(X_\infty^r)] \to \mathbb{E}[g(X_\infty)]$ and therefore $\displaystyle{\lim_{n, r \to \infty} \mathbb{E}[g(X_n^r)] = \mathbb{E}[g(X_\infty)]}.$  

By \eqref{1Dbound}, 
\begin{equation}
\lim_{n \to \infty} \mathbb{E}[g(X_n)] =\lim_{n, r \to \infty} \mathbb{E}[g(X_n^r)] = \mathbb{E}[g(X_\infty)]  
\end{equation} 

\end{proof}

For $q \ge 3$, the $q$-point correlation function of the limiting point process $\mathfrak{E}_\infty^*$ is not absolutely continuous with respect to Lebesgue measure since $\mathfrak{E}_\infty^*$ almost surely has 3 points $x,y,z \in \mathbb{R}$ such that $z - y = y -x$.  However, the 2-point correlation function is absolutely continuous.  Najnudel and Nikeghbali \cite[Prop. 5.6]{najnudel} show that 
\begin{proposition}
The 2-point correlation function for the point process $\mathfrak{E}_\infty^*$ is given by $\rho_2(x,y) = \phi_\theta(x-y)$ where 
\begin{equation}
\phi_\theta(x) = \frac{\theta}{\theta+1} + \frac{\theta}{x^2} \sum_{a \in \mathbb{N}, a \le |x|} a \bigg(1 - \frac{a}{|x|} \bigg)^{\theta - 1} 
\end{equation}
\end{proposition}

They also give a characterization of the eigenvalue gap $P_1^\theta(0,x)$ in the case $k=1$ as the solution to an integral equation \cite[Prop. 5.9]{najnudel}.  
\begin{proposition}
For all $x \in \mathbb{R}$, set \[H(x) := \mathbbm{1}(x > 0) x^{\theta - 1} P_1^\theta(0,x)\]
Then $H$ is integrable and satisfies 
\begin{equation}\label{integralequation} x H(x) = \theta \int_0^1 (1 - y)H(x-y)dy.
\end{equation}
\end{proposition}

If $y_2 - y_1 \le 1$, we can use \eqref{gapformula} to obtain a simple power series expression for $P_1^\theta(y_1,y_2)$ in the special case $\theta = 1$.  We have  
\[ P_1^1(y_1,y_2) = \mathbb{E}\bigg[ \prod_{i=1}^\infty \Big(1 + (y_1-y_2)L_i \Big) \bigg] = \sum_{m=0}^\infty (y_1-y_2)^m \sum_{i_1 < ... < i_m} \mathbb{E}[L_{i_1}...L_{i_m}] \] 

Griffiths \cite{griffiths} computed the mixed moments of the random variables $L_i$:  \begin{equation}\label{griffiths} \mathbb{E}[L_1^{j_1}...L_r^{j_r}] = \frac{\theta^r \Gamma(\theta)}{\Gamma(\theta+j)} \int\limits_{y_1 >... >y_r > 0} y_1^{j_1 - 1}...y_r^{j_r - 1} \exp \bigg(- \sum_{l=1}^r y_l - \theta E_1(y_r)\bigg) dy_1...dy_r \end{equation}
where $j = j_1 + ...+j_r$ and $\displaystyle{E_1(s) = \int_1^\infty \frac{e^{-sx}}{x}dx}$.
Taking $\theta = 1$ in \eqref{griffiths} and using symmetry in the domain, 

\begin{equation*}
\begin{multlined}
\mathbb{E}[L_{i_1}...L_{i_m}] = \frac{1}{m!} \int_0^\infty \int_{x_m}^\infty...\int_{x_2}^\infty \frac{E_1(x_1)^{i_1 - 1}}{(i_1 - 1)!} \frac{(E_1(x_2) - E_1(x_1))^{i_2 - i_1 - 1}}{(i_2 - i_1 - 1)!}... \\ ...\frac{(E_1(x_m) - E_1(x_{m-1}))^{i_m - i_{m-1} - 1}}{(i_m - i_{m-1} - 1)!} e^{-x_1}...e^{-x_m} e^{-E_1(x_m)} dx_1 dx_2... dx_m 
\end{multlined}
\end{equation*} 

We have  
\begin{equation} \label{telescope}
\begin{multlined} 
\sum_{l_1,...,l_m = 0}^\infty \frac{E_1(x_1)^{l_1}}{l_1!} \frac{(E_1(x_2) - E_1(x_1))^{l_2}}{l_2!}...\frac{(E_1(x_m) - E_1(x_{m-1}))^{l_m}}{l_m!} \\
 = e^{E_1(x_1)} e^{E_1(x_2) - E_1(x_1)} ... e^{E_1(x_m) - E_1(x_{m-1})} = e^{E_1(x_m)} 
\end{multlined}
\end{equation}

Thus, \begin{equation}\label{summoments} \sum_{ i_1 < ... < i_m } \mathbb{E}[L_{i_1}...L_{i_m}] = \frac{1}{m!} \int_0^\infty \int_{x_m}^\infty...\int_{x_2}^\infty e^{-x_1}...e^{-x_m} dx_1...dx_m = \frac{1}{(m!)^2} \end{equation}
and therefore

\begin{equation} P_1^1(y_1,y_2) = \sum_{m=0}^\infty \frac{(y_1-y_2)^m}{(m!)^2} 
\end{equation} 
for $y_2 - y_1 \le 1$.

\begin{remark}
For $y_2 - y_1 \le 1$, the gap probability $P_1^1(y_1,y_2) = J_0(2 \sqrt{y_2-y_1})$ where for real $\beta$, \[ J_\beta(x) = \sum_{m=0}^\infty \frac{(-1)^m}{m! \Gamma(m+\beta+1)} \left( \frac{x}{2}\right)^{2m + \beta}\] are the Bessel functions of the first kind.  
\end{remark}
\begin{remark}
One can check that $\displaystyle{\sum_{m=0}^\infty \frac{(-x)^m}{(m!)^2} }  $ does indeed satisfy the integral equation \eqref{integralequation} for $0 < x < 1$.  
\end{remark}

\section{Eigenvalue process for random permutation matrices at rational angles}

Equation \eqref{e0} states the limiting process $\mathfrak{E}_\infty^0$ when zooming in at the origin.  In general, it is not difficult to obtain the limiting process $\displaystyle{\lim_{n \to \infty} \mathfrak{E}_n^\alpha}$ at any rational angle $\alpha = s/t$.  As when zooming in at 0, $\displaystyle{\lim_{n \to \infty} \mathfrak{E}_n^\alpha}$ will have an infinite dirac mass at 0 since there will be an eigenangle at $\alpha$ corresponding to every $j$-cycle with $j \in t \mathbb{Z}$ and the number of such cycles is approximately $(\log n)/t$ for large $n$.  This follows from the following result by Arratia, Barbour, Tavare \cite{arratiatavare2} (for a method involving generating functions see \cite{shepp}).  

\begin{proposition}[Arratia,Barbour,Tavare] \label{P:Poissontotal}
Let $Z_i$ be independent Poisson distributed random variables with parameter $\theta/i$.  Then as $n \to \infty$ and for $b = o(n)$, \[(C_1^{(n)},...,C_b^{(n)}) \buildrel TV \over \to (Z_1,...,Z_b) \] where the convergence is in total variation.    
\end{proposition}

Let $f$ be continuous with support contained in the interval $(-T, T)$ and let $J(n)$ be defined as in \eqref{J(n)}.  Then we have the set equality \begin{equation}\label{J set} J \cap [n/(t T)] = t\mathbb{Z} \cap [n/(tT)]. \end{equation}  Also, if $j \in t\mathbb{Z}$ and $j < n/ (t T)$, then the interval $(\alpha - T/n, \alpha+ T/n)$ contains exactly one eigenangle corresponding to the $j$-cycle (namely $\alpha$).  To isolate the mass at 0, we define  $\xi_n^{tT}$ and $\eta_n^{tT}$ as in \eqref{xinr}.  Then
\begin{equation} \int f d \xi_n^{tT} = \sum_{1 \le i \le tT} \sum_{q \in \mathbb{Z}} f(n(q/L_i^{(n)} - \alpha) ) \end{equation} 
and if $f(0) \ge 0$, since $L_{r}^{(n)} < n/(tT)$ for $r > tT$, we have 
\begin{equation}
f(0) \Big(\sum_{\substack{1 \le j \le n/(tT) \\ t \mid j}} C_j^{(n)} - tT \Big) \le \int f d\eta_n^{tT} \le f(0)  \sum_{\substack{1 \le j \le n/(tT) \\ t \mid j}} C_j^{(n)}
\end{equation} 
Let $U_1,U_2,...$ be i.i.d. random variables distributed uniformly on the set $\{1/t, 2/t,...,(t-1)/t, 1\}$ and define the process 
\begin{equation}
\mathfrak{E}_\infty^t = \sum_{i=1}^\infty \sum_{q \in \mathbb{Z}} \delta_{(U_i + q)/L_i}
\end{equation} 
Define $\xi_\infty^{tT}$ and $\eta_\infty^{tT}$ as in equation \eqref{xiinftyr} with $\mathfrak{E}_\infty^t$ in place of $\mathfrak{E}_\infty^*$ and with our new definition of the variables $U_i$.  Then
\begin{equation} \int f d \xi_\infty^{tT} = \sum_{1 \le i \le tT} \sum_{q \in \mathbb{Z}} f((U_i+q)/L_i)  \end{equation} 
and since $L_r < 1/(tT)$ for $r > tT$, we have 
\begin{equation}
\int f d\eta_\infty^{tT} = f(0)  \sum_{i > tT} \mathbbm{1}(U_i = 1)
\end{equation} 
which is infinite almost surely if $f(0) > 0$.

Let $\overline{\mathbb{R}} = \mathbb{R} \cup \{\pm \infty\}$ denote the set of extended real numbers. 
Then if $f(0) > 0$, $\int f d\mathfrak{E}_n^\alpha$ converges weakly to the $\overline{\mathbb{R}}$-valued random variable $\int f d\mathfrak{E}_\infty^t$, which is infinite almost surely.

Now we can assume $f(0) = 0$.  By partitioning $[n]$ into intervals of size $t$, it is easy to see that Lemma \ref{L:uni} holds for $\alpha = s/t$ and $U_i$ as defined in this section instead of uniform on $[0,1]$.  This proves 
\begin{theorem} \label{T:k=1rational}
Let $\displaystyle{\alpha = \frac{s}{t}}$ be a rational number in reduced form.  Then $\mathfrak{E}_n^\alpha \to \mathfrak{E}_\infty^t$ weakly in the sense that we have the weak convergence $\displaystyle{\int f d \mathfrak{E}_n^\alpha \buildrel d \over \to \int f d \mathfrak{E}_\infty^t}$ of $\overline{\mathbb{R}}$-valued random variables for all continuous, compactly supported functions $f$. 
\end{theorem}

\section{Factorization classes and Well-distribution} \label{S:Factorization classes}

To set the stage for the multidimensional case, we now define factorization classes and establish well-distribution for polynomial sequences of a given factorization class in the torus. 
The notation $\bigm| \bigm| $ is shorthand for ``exactly divides,'' i.e. a prime power $p^e \bigm| \bigm| a$ iff $p^e \bigm| a$ and $p^{e+1} \nmid a$. 
   
\begin{definition}\label{factorization}
\tolerance=10000\hbadness=10000

Let $\displaystyle{p_1 <... < p_l}$ be the first $l$ primes and let $\displaystyle{(e_{1},...,e_{l})}$ be an $l$-tuple of non-negative integers.  The factorization class $\mathscr{P}_{e_1,...,e_l}$ is defined to be the set $\displaystyle{ \{m \in \mathbb{N} : p_j^{e_{j}} \bigm| \bigm| m, 1 \le j \le l \} }$.  For a matrix $\mathbf{A} = [\mathbf{a}_1,..., \mathbf{a}_r] \in \mathbb{N}^{l \times r}$, let $\mathscr{P}_\mathbf{A} = \prod\limits_{i=1}^r \mathscr{P}_{\mathbf{a}_i}$.  Then $\mathscr{P}_\mathbf{A}$ will also be called a factorization class. 

\end{definition}

The following simple lemma gives the asymptotic density of each factorization class.

\begin{lemma} \label{L:chinese}
Let $I_n$ be an interval of $n$ consecutive positive integers and let $(e_1,...,e_l)$ be a vector of non-negative integers.  Then as $n \to \infty$, the fraction of integers $a \in I_n$ in the factorization class $\displaystyle{\mathscr{P}_{e_1,...,e_l}}$ approaches $\prod\limits_{m=1}^l (1 - 1/p_m)(1/p_m)^{e_m}$.
\end{lemma}

\begin{proof}
Let $N = \prod\limits_{m=1}^l p_m^{e_m+1}$ and apply the Chinese Remainder theorem to $\mathbb{Z}/N\mathbb{Z}$.  Then consider multiples $aN$ as $a \to \infty$.  
\end{proof}

We will need the following notion of equidistribution and well-distribution for multi-indexed multidimensional sequences.

\begin{definition}
Let $(\mathbf{s}_{j_1,...,j_r})$ for $(j_1,...,j_r) \in \mathbb{N}^r$ be a multi-indexed sequence of elements in the $d$-dimensional torus $\mathbb{T}^d$.  Let $\displaystyle{ \mathcal{P}_N^{t_1,...,t_r} = \{ \mathbf{s}_{j_1,...,j_r}: t_i \le j_i \le t_i + N \} }$ as a multiset.  Then we say that $(\mathbf{s}_{j_1,...,j_r})$ is equidistributed if $\displaystyle{ \lim_{N \to \infty} D(\mathcal{P}_N^{0,...,0}) = 0}$.  The sequence $(\mathbf{s}_{j_1,...,j_r})$ is well-distributed if $\displaystyle{ \lim_{N \to \infty} D(\mathcal{P}_N^{t_1,...,t_r}) = 0}$ uniformly over the tuples $(t_1,...,t_r) \in \mathbb{N}^r$.

\end{definition}

Before stating the main result Lemma \ref{L:welltorus} of this section, we define the following sets.

\begin{definition}\label{discfac}
Let the intervals $I_{m,b}$ be as in \eqref{bin}, let $\mathbf{A} \in \mathbb{N}^{l \times r}$ be a matrix of non-negative integers, and let $\boldsymbol\alpha = (\alpha_{i_1,...,i_k})_{1 \le i_1 < ... < i_k \le r}$ be a vector of irrational numbers.  Define the multiset \begin{equation} \mathcal{P}_{(k_i), b}^{\mathscr{P}_\mathbf{A}, \boldsymbol \alpha} = \Big \{ ( m_{i_1}...m_{i_k} \alpha_{i_1,...,i_k}  )_{1 \le i_1 < ...< i_k \le r} : (m_1,...,m_r) \in \mathscr{P}_\mathbf{A} \cap \prod_{i=1}^r I_{k_i, b} \Big\}  \end{equation}
of tuples in $\displaystyle{\mathbb{T}^{ \binom{r}{k} }}$. 
\end{definition}

\begin{lemma} \label{L:welltorus}
The discrepancies $\displaystyle{D \big(\mathcal{P}_{(k_i), b}^{\mathscr{P}_\mathbf{A}, \boldsymbol \alpha} \big) \to 0}$  as $n/b \to \infty$ uniformly over the tuples $(k_1,...,k_r)$.
\end{lemma}

To prove this, we use a result about well-distribution in the torus stated in Lemma \ref{L:wellt} below.  We briefly review the essential ingredients needed, which are given as exercises in Tao \cite[pp. 12-13]{tao}.  To simplify notation involving exponential sums, we set $e(x) := e^{2 \pi i x}$.   

\begin{proposition}[Multi-dimensional Weyl Criterion for well-distribution]
The sequence $(\mathbf{s}_{j_1,...,j_m})$ is well-distributed in $\mathbb{T}^d$ if and only if for each non-zero $\mathbf{h} \in \mathbb{Z}^d$, 
\[\lim_{N \to \infty} \frac{1}{N^m} \sum_{(j_1,...,j_m) \in B } e(\mathbf{h} \cdot \mathbf{s}_{j_1,...,j_m}) = 0\] uniformly over all cubes $B$ of side length $N$ with vertices in $\mathbb{N}^m$.  
\end{proposition}

To apply the Weyl Criterion, we need a uniform version of the multidimensional van der Corput Lemma.  

\begin{lemma}[Multidimensional van der Corput Lemma] \label{L:vanderCorput} Fix $k \in \mathbb{Z}^d$.  Let x: $\mathbb{N}^m \mapsto \mathbb{T}^d$ be such that for every $h$ outside of a hyperplane of $\mathbb{R}^m$, the difference sequence $\partial_h x: n \mapsto x(n+h) - x(n)$ satisfies $\displaystyle{ \frac{1}{N^m} \Big|\sum_{n \in B} e(k \cdot \partial_h x(n)) \Big| \to 0}$ uniformly over all cubes $B$ of side length $N$ with vertices lying in $\mathbb{N}^m$.  Then $\displaystyle{ \frac{1}{N^m} \Big|\sum_{n \in B} e(k \cdot x(n)) \Big| \to 0}$ uniformly over all $B$ as well.  
\end{lemma}

\begin{proof} 
This is given as exercise 1.1.12 (without the condition on uniformity) in \cite{tao}.  The argument remains unaffected by adding this condition in both the hypothesis and conclusion. 
\end{proof}
 
Using Lemma \ref{L:vanderCorput}, we get

\begin{lemma}\label{L:wellt}
Let \[\mathbf{P}(n_1,...,n_m) = \sum_{\substack{s_1,...,s_m \ge 0 \\ s_1+...+s_m \le s}} \boldsymbol\alpha_{s_1,...,s_m} n_1^{s_1}...n_m^{s_m}\]
be a polynomial map from $\mathbb{N}^m$ to $\mathbb{T}^d$ of degree $s$, where $\boldsymbol \alpha_{s_1,...,s_m} \in \mathbb{T}^d$ are coefficients.  Suppose there does not exist a non-zero $\mathbf{k} \in \mathbb{Z}^d$ such that $\mathbf{k} \cdot \boldsymbol \alpha_{s_1,...,s_m} = 0$ for all $(s_1,...,s_m) \neq 0$.  Then $\mathbf{P}$ is well-distributed in the torus. 
\end{lemma}

\begin{proof}
This is given as exercise 1.1.13 in \cite{tao}, but in terms of equidistribution rather than well-distribution.  The idea of the proof is to use Weyl's criterion for well distribution and proceed by induction on the degree of the polynomial.  The inductive step uses Lemma \ref{L:vanderCorput}.  The reason that we can upgrade from equidistribution to well-distribution is that in the base case, the sequence $\{n \alpha\}$ is not only equidistributed, but well-distributed (because of the summation formula for geometric series).  
\end{proof}

\begin{proof}[Proof of Lemma \ref{L:welltorus}]
For $\displaystyle{1 \le i \le r}$, set $N_i = \prod\limits_{m=1}^l p_m^{\mathbf{A}_{mi}+1} $ and let $s_i$ be an element of $\mathscr{P}_{\mathbf{a}_i}\cap [N_i]$.   Define the sequences $s^i(t) = s_i + t N_i$ where $t$ runs over the non-negative integers.  Note that $s^i(t) \in \mathscr{P}_{\mathbf{a}_i}$ for all $t$.  Thus, we can partition each factorization class $\mathscr{P}_{\mathbf{a}_i}$ into congruence classes modulo $N_i$.  Then it suffices to show that the multidimensional polynomial sequence
\[\mathbf{P}(t_1,...,t_r) := \bigg( \alpha_{i_1,...,i_k} \prod_{j=1}^k s^{{i_j}}(t_{i_j}) \bigg)_{1 \le i_1 < ... < i_k \le r} \]
is well-distributed.  This follows from Lemma \ref{L:wellt} since the polynomials $\prod\limits_{j=1}^k s^{{i_j}}(t_{i_j})$ are linearly independent.  
\end{proof}

\section{Eigenvalue process for $\rho_{n,k}$ at angles of finite irrationality measure} \label{S:main theorem}

In this section, we prove Theorem \ref{T:generalmain} assuming the main technical result Theorem \ref{T:technical} we will prove in the following sections.  

\begin{remark}
Let $Y_{n,k}$ be the scaled eigenvalue fluctuation statistics of $\rho_{n,k}$ studied in \cite{tsou}.  In that case, the non-negligible contributions as $n \to \infty$ come from $\sigma_k$-cycles containing tuples $(t_1,...,t_k)$ such that $t_1,...,t_k$ are all in the same $\sigma$-cycle.  For the eigenvalue point process on the other hand, we will show that the non-negligible contributions to $\mathfrak{E}_{n,k}^\alpha$ all come from $\sigma_k$-cycles containing tuples $(t_1,...,t_k)$ such that $t_1,...,t_k$ are all in different $\sigma$-cycles.  
\end{remark}

\begin{definition}
Let $p_m$ be the $m^{\textrm{th}}$ prime and $r$ be a positive integer.  For each positive integer $a$, let $e_m(a)$ satisfy $p_m^{e_m(a)} \bigm| \bigm| a$.  Then let $\mathbf{E}_l(a_1,...,a_r)$ denote the $l \times r$ matrix whose $mi^{\textrm{th}}$ entry is $e_m(a_i)$.
\end{definition}

The following lemma shows that the $\lcm$ of $k$-subsets of the largest $r$ cycle lengths can be expressed in terms of the first $l$ primes with a small error.
\begin{lemma} \label{L:firstl}
With the notation defined above, \[\limsup_{n \to \infty} \mathbb{P} \bigg( \bigcup_{1 \le i_1 < ...< i_k \le r} \Big\{ \frac{L_{i_1}^{(n)}...L_{i_k}^{(n)}}{\lcm(L_{i_1}^{(n)},...,L_{i_k}^{(n)})}  \neq g_k(\mathbf{E}_l(L_{i_1}^{(n)}, ...,L_{i_k}^{(n)})) \Big \} \bigg) \le \binom{r}{ 2}\frac{1}{p_l}\] 
\end{lemma}

\begin{proof}
 
Let $I_{m,b}$ be as in \eqref{bin}.  Define $\varepsilon(n,b), \varepsilon^*(n)$ and $\varepsilon(\lambda)$ as in Lemma \ref{L:uni}.  Set $b = o(n)$.  For $1 \le k_1, k_2 \le b$, let $\delta_{k_1,k_2}(l, n)$ be the fraction of tuples $(m_1,m_2) \in I_{k_1, b} \times I_{k_2,b}$ such that $p \bigm| m_1$ and $p \bigm| m_2$ for some prime $p > p_l$.  Let $\displaystyle{\delta_{k_1,...,k_r}(l,n) = \sum_{1 \le i < j \le r} \delta_{k_i,k_j}(l,n)}$.  Then  \[\limsup_{n \to \infty} \delta_{k_1,...,k_r}(l, n) \le \binom{r}{ 2} \limsup_{n \to \infty} \sum_{m = l+1}^n \frac{1}{p_m^2} + \frac{3b}{p_m n} + \frac{3b^2}{n^2} \le \binom{r}{ 2}\frac{1}{p_l} .\]  Therefore,
\begin{align*}
&\qquad \mathbb{P} \bigg( \bigcup_{1 \le i_1 < ...< i_k \le r} \Big\{ \frac{L_{i_1}^{(n)}...L_{i_k}^{(n)}}{\lcm(L_{i_1}^{(n)},...,L_{i_k}^{(n)})}  \neq g_k(\mathbf{E}_l(L_{i_1}^{(n)}, ...,L_{i_k}^{(n)})) \Big \} \bigg) \\
& \begin{aligned} \le \sum_{1 \le m_r < ...< m_1 \le n} \mathbb{P} \bigg(\bigcap_{i=1}^r \{L_i^{(n)} = m_i\} \bigcap \bigcup_{1 \le i_1 < ...< i_k \le r} \bigg\{ \frac{m_{i_1}...m_{i_k}}{\lcm(m_{i_1},...,m_{i_k})} \neq g_k(\mathbf{E}_l(m_{i_1}, ...,m_{i_k}))  \bigg\} \bigg) & \\ + \varepsilon^*(n)& \end{aligned} \\
&\begin{multlined} \le \sum_{\substack{\lambda b \le k_r \le...\le k_1 \le b  \cr k_1 +...+ k_r \le (1 - \lambda )b } }  \bigg(f_\theta^{(r)} \left(\frac{k_1}{b}, ...,\frac{k_r}{b}\right) + \varepsilon(n,b) \bigg) \frac{1}{b^r} \delta_{k_1,...,k_r}(l,n) \\ \shoveright{ + \mathbb{P}\{ L_{r}^{(n)} < \lambda n\} + \mathbb{P}\{L_1^{(n)} +...+ L_{r}^{(n)} > (1- \lambda) n\} + \varepsilon^*(n) }
\end{multlined}
\end{align*}
The result follows on taking $n \to \infty$ and $\lambda \to 0$.
\end{proof}

We establish the following generalization of Lemma \ref{L:uni}.

\begin{lemma}\label{L:unigeneral}
Let $\alpha \in \mathbb{R}$ be irrational and fix positive integers $r$ and $l$ with $r \ge k$.  For $\displaystyle{1 \le i_1 < ... < i_k \le r}$, let $U_{i_1,...,i_k}$ be i.i.d. random variables distributed uniformly on $[0,1]$ and independent of the $PD(\theta)$ process $(L_1,L_2,...)$.  For $1 \le m \le l$ and $1 \le i \le r$, let the random variables $X_{mi}$ from \eqref{Xmi} be independent of both $U_{i_1,...,i_k}$ and $(L_1,L_2,...)$.
Then we have the distributional convergence 
\begin{multline*} \bigg(\textbf{E}_l(L_{1}^{(n)},...,L_{r}^{(n)}), L_1^{(n)}/n,...,L_r^{(n)}/n, \Big(\psi_\alpha \Big(\frac{L_{i_1}^{(n)}...L_{i_k}^{(n)}}{ g_k(\textbf{E}_l(L_{i_1}^{(n)}, ...,L_{i_k}^{(n)}))} \Big) \Big)_{1 \le i_1 < ...< i_k \le r} \bigg) \\ \buildrel d \over \to \left(\textbf{X}_l^{1,...,r}, L_1,...,L_r, (U_{i_1,...,i_k})_{1 \le i_1 < ... < i_k \le r} \right) 
\end{multline*}
\end{lemma}

\begin{proof} 
Let the intervals $I_{m,b}$ be as in \eqref{bin}.  Define $ \varepsilon(n,b)$, $\varepsilon^*(n)$ and $\varepsilon(\lambda)$ as in the proof of Lemma \ref{L:uni}.  Recall $\psi_\alpha(j) \le x$ iff $\{j \alpha \} \ge 1-x$.  

For $\displaystyle{1 \le i_1 < ... < i_k \le r}$ and $1 \le i \le r$, let $c_{i}, d_{i_1,...,i_k} \in [0,1]$ be real numbers.  Let $\textbf{A} = [\textbf{a}_1,...,\textbf{a}_r] \in \mathbb{N}^{l \times r}$.  Then $\textbf{E}_l(m_1,...,m_r) = \textbf{A}$ iff $(m_1,...,m_r) \in \mathscr{P}_{\textbf{A}}$.  Define the fraction \begin{equation}\label{fraction} \gamma_{(k_i), b}^{\mathscr{P}_{\textbf{A}} } := \prod_{i=1}^r \frac{ |I_{k_i,b} \cap \mathscr{P}_{\textbf{a}_i} |}{ |I_{k_i,b}|}\end{equation}  Finally, let $\displaystyle{\boldsymbol \alpha = \Big(\frac{\alpha}{ g_k(\textbf{a}_{i_1},...,\textbf{a}_{i_k})} \Big)_{1 \le i_1 < ...< i_k \le r} }$.  Then

\begin{align} \label{unigeneralexpression}
&\begin{multlined} \mathbb{P}\bigg(\bigcap_{i=1}^r \{ L_{i}^{(n)}/n < c_{i} \} \cap \Big\{ \textbf{E}_l(L_{1}^{(n)},...,L_{r}^{(n)}) = \textbf{A} \Big\} \cap \\ \bigcap_{1 \le i_1 < ...< i_k \le r}  \Big \{\psi_\alpha\Big(\frac{L_{i_1}^{(n)}...L_{i_k}^{(n)}}{ g_k(\textbf{E}_l(L_{i_1}^{(n)}, ...,L_{i_k}^{(n)}))} \Big) < d_{i_1,...,i_k} \Big\}  \bigg)
\end{multlined} \\
&\begin{multlined}
\le \sum_{\substack{1 \le m_r < ...< m_1 \le n \cr m_i/n < c_i}} \mathbb{P} \bigg(\bigcap_{i=1}^r \{L_i^{(n)} = m_i\} \cap \big\{ \textbf{E}_l(m_{1},...,m_{r}) = \textbf{A} \big\} \cap \\ \bigcap_{1 \le i_1 < ...< i_k \le r}  \Big\{\psi_\alpha\Big(\frac{m_{i_1}...m_{i_k}}{g_k(\textbf{a}_{i_1},...,\textbf{a}_{i_k}) } \Big) < d_{i_1,...,i_k} \Big\} \bigg) + \varepsilon^*(n) \nonumber
\end{multlined} \\
& \begin{aligned} \le \sum_{\substack{\lambda b \le k_r \le...\le k_1 \le b \cr k_1+...+k_r \le (1- \lambda)b \cr (k_i-1)/b < c_i}}  \bigg(f_\theta^{(r)} \left(\frac{k_1}{b}, ...,\frac{k_r}{b}\right) + \varepsilon(n,b) \bigg) \frac{1}{b^r} \gamma_{(k_i), b}^{\mathscr{P}_{\textbf{A}} } \bigg( \prod_{1 \le i_1 < ...< i_k \le r}  d_{i_1,...,i_k} + D \big(\mathcal{P}_{(k_i), b}^{\mathscr{P}_\mathbf{A}, \boldsymbol \alpha} \big) \bigg) & \\ + \mathbb{P}\{ L_{r}^{(n)} < \lambda n\} + \mathbb{P}\{L_1^{(n)} +...+ L_{r}^{(n)} > (1- \lambda) n\} + \varepsilon^*(n) \nonumber &
\end{aligned}
\end{align}
Then by Lemma \ref{L:chinese}, $\displaystyle{\gamma_{(k_i), b}^{\mathscr{P}_{\textbf{A}} } \to \prod_{i=1}^r \prod_{m=1}^l (1 - 1/p_m) (1/p_m)^{\textbf{A}_{m i}} }$ uniformly as $n/b \to \infty$ and by Lemma \ref{L:welltorus}, $\displaystyle{ D \big(\mathcal{P}_{(k_i), b}^{\mathscr{P}_\mathbf{A}, \boldsymbol \alpha} \big) \to 0}$ uniformly as $n/b \to \infty$.   
Taking $b = o(n)$ so that $n/b \to \infty$, we can bound \eqref{unigeneralexpression} in the limit by 
\begin{equation*}
\int\limits_{\substack{\lambda < x_r <...< x_1 < 1 \cr x_1+...+x_r < 1- \lambda \cr x_i < c_i}} f_\theta^{(r)}(x_1,...,x_r) dx_i \prod_{i=1}^r \prod_{m=1}^l (1 - 1/p_m) (1/p_m)^{\textbf{A}_{m i}}\prod_{1 \le i_1 < ... < i_k \le r} d_{i_1,...,i_k} + \varepsilon(\lambda)
\end{equation*}
We obtain the analogous expression (without the $\varepsilon(\lambda)$ term) for the lower bound. 
 
Taking $\lambda \to 0$, the limiting probability is 
\[\mathbb{P}\{ L_i < c_i, 1 \le i \le r \} \mathbb{P}\{X^{1,...,r}_l = \textbf{A}\} \prod_{1 \le i_1 < ... < i_k \le r} \mathbb{P}\{U_{i_1,...,i_k} < d_{i_1,...,i_k} \}  \]
\end{proof}
\begin{corollary} \label{C:firstr}

Recall the notation from Lemma \ref{L:unigeneral} and as usual, allow the multi-dim\-en\-sion\-al indices $(i_1,...,i_k)$ to run over the range $\displaystyle{1 \le i_1 < ...< i_k \le r}$.  Then we have the distributional convergence
\begin{multline*}
 \bigg( \bigg(\frac{L_{i_1}^{(n)}...L_{i_k}^{(n)}}{\lcm(L_{i_1}^{(n)},...,L_{i_k}^{(n)})} \bigg), (L_{i_1}^{(n)}...L_{i_k}^{(n)}/n^k),(\psi_\alpha(\lcm(L_{i_1}^{(n)},...,L_{i_k}^{(n)}) ) ) \bigg) \\
\buildrel d \over \to \left( (g_{k}(\textbf{X}_\infty^{i_1,...,i_k})),  (L_{i_1}...L_{i_k}), (U_{i_1,...,i_k})\right)
\end{multline*}
\end{corollary}

\begin{proof}
By Lemma \ref{L:unigeneral} and the continuous mapping theorem, 
\begin{multline*}
\bigg( (g_k(\textbf{E}_l(L_{i_1}^{(n)},...,L_{i_k}^{(n)}))) , (L_{i_1}^{(n)}...L_{i_k}^{(n)}/n^k), \psi_\alpha\Big(\frac{L_{i_1}^{(n)}...L_{i_k}^{(n)}}{ g_k(\textbf{E}_l(L_{i_1}^{(n)}, ...,L_{i_k}^{(n)}))} \Big) \bigg) \\ \buildrel d \over \to \bigg( (g_{k}(\textbf{X}_l^{i_1,...,i_k})), (L_{i_1}...L_{i_k}), (U_{i_1}...U_{i_k}) \bigg) 
\end{multline*}

By Lemma \ref{L:firstl}, there exists an error function $\varepsilon(l) \to 0$ as $l \to \infty$ such that \[\limsup_{n \to \infty} \mathbb{P} \bigg\{ \bigg(\frac{L_{i_1}^{(n)}...L_{i_k}^{(n)}}{\lcm(L_{i_1}^{(n)},...,L_{i_k}^{(n)})} \bigg) \neq (g_k(\textbf{E}_l(L_{i_1}^{(n)}, ...,L_{i_k}^{(n)}))) \bigg\} < \varepsilon(l)\]  and \[ \mathbb{P} \{ (g_{k}(\textbf{X}_l^{i_1,...,i_k})) \neq (g_{k}(\textbf{X}_\infty^{i_1,...,i_k})) \} < \varepsilon(l). \]  

Taking $l \to \infty$, we obtain the desired convergence.
\end{proof} 

\begin{proof}[Proof of Theorem \ref{T:generalmain}]
Let $r$ be a positive integer.  To isolate the contribution from the $r$ largest cycles, define
\begin{equation}\label{e1} \xi_{n,k}^{r} = k! \sum_{1 \le i_1 <...< i_k \le r} \frac{L_{i_1}^{(n)}...L_{i_k}^{(n)}}{\lcm(L_{i_1}^{(n)},...,L_{i_k}^{(n)})} \sum_{q \in \mathbb{Z}} \delta_{n^k (q/\lcm (L_{i_1}^{(n)},...,L_{i_k}^{(n)}) - \alpha )} \end{equation} 
and let $\displaystyle{\eta_{n,k}^r = \mathfrak{E}_{n,k}^\alpha - \xi_{n,k}^r}$ denote the remaining measure.  Similarly, define
\begin{equation}\label{e2} \xi_{\infty, k}^{r} = k! \sum_{1 \le i_1 <...< i_k \le r} g_{k}(\textbf{X}_\infty^{i_1,...,i_k}) \sum_{q \in \mathbb{Z}} \delta_{(q + U_{i_1,...,i_k}) g_{k}(\textbf{X}_\infty^{i_1,...,i_k})  /(L_{i_1}...L_{i_k}) } \end{equation}
and set $\eta_{\infty,k}^r = \mathfrak{E}_{\infty,k}^* - \xi_{\infty,k}^r$.   

Let $f$ be a continuous function with support contained in some interval $(-T, T)$.  Then 
\begin{align*}
\int f d \xi_{n,k}^r &=  k! \sum_{1 \le i_1 < ...< i_k \le r} \frac{L_{i_1}^{(n)}...L_{i_k}^{(n)}}{\lcm(L_{i_1}^{(n)},...,L_{i_k}^{(n)})} \sum_{q \in \mathbb{Z}} f \bigg(\frac{q - \lcm(L_{i_1}^{(n)},...,L_{i_k}^{(n)}) \alpha} {\lcm(L_{i_1}^{(n)},...,L_{i_k}^{(n)})/n^k} \bigg) \\
&= k! \sum_{1 \le i_1 < ...< i_k \le r} \frac{L_{i_1}^{(n)}...L_{i_k}^{(n)}}{\lcm(L_{i_1}^{(n)},...,L_{i_k}^{(n)})} \sum_{q \in [-T+1, T+1]} f \bigg(\frac{q + \psi_\alpha(\lcm(L_{i_1}^{(n)},...,L_{i_k}^{(n)}) ) }{\lcm(L_{i_1}^{(n)},...,L_{i_k}^{(n)})/n^k} \bigg) 
\end{align*}
 and \[\int f d\xi_{\infty,k}^r = k! \sum_{1 \le i_1 < ...< i_k \le r } g_{k}(\textbf{X}_\infty^{i_1,...,i_k}) \sum_{q \in [-T+1, T+1]} f \bigg( (q + U_{i_1,...,i_k})\frac{g_{k}(\textbf{X}_\infty^{i_1,...,i_k}) }{L_{i_1}...L_{i_k}} \bigg) \]

By Corollary \ref{C:firstr} and the continuous mapping theorem, 
\begin{equation} \int f d \xi_{n,k}^r \buildrel d \over \to \int f d\xi_{\infty,k}^r \end{equation}

Now we turn to the remainder measures $\eta_{n,k}^r$ and $\eta_{\infty,k}^r$.  Since $\mathbb{E}[(\sum_i L_i)^k] = \mathbb{E}[1^k] = 1$, by the Borel-Cantelli lemma $\displaystyle{\int f d \eta_{\infty,k}^r \buildrel a.s. \over \to 0}$ as $r \to \infty$.  Note that this just verifies that $\mathfrak{E}_{\infty,k}^*$ is in fact a point process.

Showing that $\displaystyle{\lim_{r \to \infty} \limsup_{n \to \infty} \mathbb{P} \bigg\{ \int |f| d\eta_{n,k}^r  > 0 \bigg\} = 0}$ is considerably more involved than in the $k=1$ case.  We split the point process $\eta_{n,k}^r$ into two parts: 
\begin{enumerate}
\item points coming from $\sigma_k$-cycles containing tuples $(t_1,...,t_k)$ such that $t_1,...,t_k$ are all in different $\sigma$-cycles, not all of which are chosen from the $r$ largest cycles.
\item points coming from $\sigma_k$-cycles containing tuples $(t_1,...,t_k)$ such that for some $i \neq j$, $t_i$ and $t_j$ are in the same $\sigma$-cycle.  
\end{enumerate}

More formally, we make the following:
\begin{definition}
For $0 \le c \le s \le k$ and each $\sigma \in \mathfrak{S}_n$, define $H_{n,s}^{r, c}(\sigma)$ to be the set of all integers $\displaystyle{\lcm(L_{i_1}^{(n)}(\sigma),..., L_{i_s}^{(n)}(\sigma))}$ such that $L_{i_u}^{(n)}(\sigma)$ are all distinct where $1 \le i_1 < ...< i_c \le r < i_{c+1} <... < i_s$. 
\end{definition}
Define the corresponding point process \begin{equation} \label{nu} \nu_{n,s}^{r,c} = \sum_{j \in H_{n,s}^{r,c}} \sum_{q \in \mathbb{Z}} \delta_{n^k(q/j - \alpha)} \end{equation}
It is easy to check that 
\begin{equation} \eta_{n,k}^{r,\star} = \Big(\sum_{c=0}^{k-1} \nu_{n,k}^{r,c} + \sum_{s=1}^{k-1} \sum_{c=0}^{s} \nu_{n,s}^{r,c} \Big)^\star \end{equation}
where $\xi^{\star}$ denotes the corresponding simple point process of a point process $\xi$. 
Then
\begin{equation}\label{nubound} \mathbb{P} \bigg\{ \int |f| d\eta_{n,k}^{r,\star} > 0 \bigg\} \le \sum_{c=0}^{k-1} \mathbb{P} \bigg\{ \int |f| d\nu_{n,k}^{r,c,\star} > 0 \bigg\} + \sum_{s=1}^{k-1} \sum_{c=0}^s \mathbb{P} \bigg\{ \int |f| d\nu_{n,s}^{r,c,\star} > 0 \bigg\} 
\end{equation}

If $c=s$, then $\displaystyle{H_{n, s}^{r,s} \subset \{ \lcm(L_{i_1}^{(n)},..., L_{i_s}^{(n)}) : 1 \le i_1 < ...< i_s \le r \} }$.  Note that $\nu_{n,s}^{r, s,\star}$ is the simple point process $\xi_{n,s}^{r,\star}$ dilated by a factor of $n^{k-s}$.  In other words, \begin{equation} \label{scale} \int |f(x)| d \nu_{n,s}^{r, s, \star}(x) \le \int |f(x n^{k-s})| d \xi_{n,s}^{r, \star}(x) \end{equation}  

The support of $f(x n^{k-s})$ is contained in $(-T/n^{k-s}, T/n^{k-s})$.  Let $g_\varepsilon$ be a continuous function with support contained in a small interval $(\varepsilon, \varepsilon)$.  The mean intensity of $\mathfrak{E}_{\infty, s }^*$ is some finite constant $\lambda_s$.  We have \[\limsup_{n \to \infty} \mathbb{P} \bigg\{ \int |g_\varepsilon| d \xi_{n,s}^{r, \star} > 0 \bigg\} < 2\varepsilon \lambda_s \]

Thus, for $s < k$ and for each $r$, $\displaystyle{ \lim_{n \to \infty} \mathbb{P} \bigg\{ \int |f| d \nu_{n,s}^{r, s, \star} > 0 \bigg\} = 0}$ by \eqref{scale}. 

Define the set \begin{equation}\label{J_k} J_k(n) := \{j \in [n^k]: (\exists q \in \mathbb{Z})[\alpha - T/n^k < q/j < \alpha + T/n^k] \} \end{equation}  Then by Theorem \ref{T:technical} below, for $\displaystyle{ 1 \le s \le k}$ and $\displaystyle{0 \le c < s}$,
\begin{equation} \label{usetechnical}
\lim_{r \to \infty} \limsup_{n \to \infty} \mathbb{P} \bigg\{ \int |f| d\nu_{n,s}^{r,c,\star} > 0 \bigg\} \le \lim_{r \to \infty} \limsup_{n \to \infty} \mathbb{P} \left\{ H_{n,s}^{r,c} \cap J_k(n) \neq \varnothing \right\} = 0 \end{equation}

By \eqref{nubound}, \[\lim_{r \to \infty} \limsup_{n \to \infty} \mathbb{P} \bigg\{ \int |f| d\eta_{n,k}^{r,\star} > 0 \bigg\} = 0. \]  Clearly, $\mathbb{P}\{ \int |f| d\eta_{n,k}^{r} > 0 \} = \mathbb{P}\{ \int |f| d\eta_{n,k}^{r,\star} > 0 \}$.
The rest of the proof follows the proof of Theorem \ref{T:1dmain} identically.
\end{proof}

\begin{remark} \label{otherreps}
For $\alpha$ of finite irrationality measure, let $^\text{sub}\mathfrak{E}_{n,k}^\alpha$ and $^{\text{irrep}}\mathfrak{E}_{n,k}^\alpha$ be the corresponding versions of the scaled eigenvalue point processes $\mathfrak{E}_{n,k}^\alpha$ for the $k$-subset permutation representations and $S^{(n-k,1^k)}$ irreducible representations of the symmetric group $\mathfrak{S}_n$ (see \cite{tsou} for details).  Then it is easy to see that with $\xi_{n,k}^r$ and $\eta_{n,k}^{r}$ as defined above, \[^\text{sub}\mathfrak{E}_{n,k}^\alpha = \xi_{n,k}^r/k! + ^\text{sub}\eta_{n,k}^r\] and \[^\text{irrep}\mathfrak{E}_{n,k}^\alpha = \xi_{n,k}^r/k! + ^\text{irrep}\eta_{n,k}^r \] where the supports of $^\text{sub}\eta_{n,k}^{r}$ and $^\text{irrep}\eta_{n,k}^{r}$ are both contained in the support of $\eta_{n,k}^{r}$ almost surely.  Thus, both $^\text{sub}\mathfrak{E}_{n,k}^\alpha$ and $^{\text{irrep}}\mathfrak{E}_{n,k}^\alpha$ converge weakly to  $\frac{1}{k!}\mathfrak{E}_{\infty,k}^*$.  
\end{remark}

\section{Bound on small cycle contribution}\label{S:boundsmallcycle}

We devote this section to the task of establishing the technical result Theorem \ref{T:technical} used in \eqref{usetechnical} for the proof of Theorem \ref{T:generalmain}.  We will use a key discrepancy bound proved in Theorem \ref{T:maindisc} of the next section.  

\begin{theorem} \label{T:technical}
For $1 \le s \le k$ and $0 \le c < s$, 
\begin{equation}
\lim_{r \to \infty} \limsup_{n \to \infty} \mathbb{P} \left\{ H_{n,s}^{r,c} \cap J_k(n) \neq \varnothing \right\} = 0
\end{equation}
\end{theorem}

In the sequel, we fix some choice of $c$,$s$, and $k$ (according to the restrictions of the theorem) for the sake of concreteness.  Recall the definition of factorization classes from Definition \ref{factorization}.  The following notation will be useful:
\begin{definition}\label{Q10}
\tolerance=10000\hbadness=10000

For $n > 10$, let $l = \lfloor (\log n)^{2s} \rfloor$.  Define the set of $l$-tuples 
\begin{equation} \label{Qn} Q_n = \{ (e_{1},...,e_{l}) \in \mathbb{N}^l : (0 \le e_{j} < 2 s \log \log n) \land \big(|\{j : e_{j} > 0\} | < (\log \log n)^2 \big) \} \end{equation}  We also define the (disjoint) union \begin{equation} \mathscr{Y}_n = \bigcup_{(e_1,...,e_l) \in Q_n} \mathscr{P}_{e_1,...,e_l} \end{equation} 
\end{definition} 

\begin{remark}
The choice $l = \lfloor (\log n)^{2s} \rfloor$ is motivated by Lemma \ref{L:log2s} below.  All three parameters $\lfloor (\log n)^{2s} \rfloor$, $2s \log \log n$, and $(\log \log n)^2$ are chosen so that we can obtain the discrepancy bound in Theorem \ref{T:maindisc}.
\end{remark} 

The first major task is to bound the probability $\displaystyle{\mathbb{P} \left\{ H_{n,s}^{r,c} \cap J_k(n) \neq \varnothing \right\}}$ by a sum involving factorization classes in $\mathscr{Y}_n$ and to replace the $\lcm$ in the definition of $H_{n,s}^{r,c}$ with a product involving the functions $g_s$ from \eqref{gk}, which will be constant over each factorization class.  This is essential so that we can apply the equidistribution theory of polynomials.  

To help streamline the presentation of inequalities, we will use the following conventions:

\begin{definition} \label{D:inequality}
For two non-negative functions $f(n)$ and $g(n)$, we say that $$f \ll g$$ if $f \le Cg$ for some absolute constant $C$ that does not depend on $n$.  We say that $$f \lesssim g$$ if $f \le Cg + h$ for some absolute constant $C$ and function $h(n)$ such that $\lim\limits_{n \to \infty} h(n) = 0$.
\end{definition} 
In the following, we allow the indices of $j_v$ to run over $c+1 \le v \le s$ instead of $1 \le v \le s-c$ for notational convenience.  For $i \ge 1$, let $Z_i$ be independent Poisson random variables with mean $\theta/i$.  We will let $\mathbf{1}(E)$ denote the indicator function of the event $E$, i.e. it takes the value $1$ if $E$ holds and $0$ otherwise.
To state the next theorem, we introduce the functions
\begin{equation} 
f_{j_{c+A+1},...,j_s}^{n, u_{c+A+1},...,u_s}(x_1,...,x_r) := \frac{e^{\gamma \theta} \theta^r \Gamma(\theta) x_r^{\theta-1} }{x_1...x_r} p_\theta \bigg( \Big(1- \sum \limits_{i=1}^r x_i - \sum \limits_{v=c+A+1}^s u_v \frac{j_v}{n} \Big) \Big/ x_r \bigg)
\end{equation}
for positive integers $u_v$ and $j_v$ with domain \[D := \bigg\{(x_1,...,x_r): (0 < x_r <....< x_1 < 1)\land \bigg(x_1+...+x_r <  1 - \sum_{v=c+A+1}^s u_v \frac{j_v}{n} \bigg) \bigg\} \]
This definition is motivated by the form of the density $f_\theta^{(r)}(x_1,...,x_r)$ in \eqref{PDthetadensity}.  By scaling linearly, one can check that  
\begin{equation} \int\limits_{D} f_{j_v}^{n, u_v}(x_1,...,x_r)dx_1...dx_r = \bigg(1 - \sum \limits_{v=c+A+1}^s u_v j_v/n \bigg)^{\theta-1} 
\end{equation} 

\begin{theorem} \label{T:Big}
Let $Q_n$ be as defined in \eqref{Qn}.  Fix $0 < \lambda < 1$ and $0 < \varepsilon < 1$.  Then as $n \to \infty$,

\begin{equation} \label{big equation}
\begin{multlined}
\mathbb{P} \left\{ H_{n,s}^{r,c} \cap J_k(n) \neq \varnothing \right\} \lesssim \sum_{\substack{\mathbf{e}_1,...,\mathbf{e}_s \in Q_n \cr 0 \le A \le s-c \cr 1 \le i_1 < ...< i_c \le r}} \sum_{\substack{ 1 \le j_{c+1}<...<j_s \le \lceil n/r \rceil  \cr j_{c+A} \le n^\varepsilon < j_{c+A+1} \cr j_v \in \mathscr{P}_{\mathbf{e}_v} }} \sum_{\substack{\lambda n < m_r < ...< m_1 \le n \cr \sum m_i + \frac{4}{\varepsilon} \sum j_v < (1-\lambda)n \cr  m_{i_u} \in \mathscr{P}_{\mathbf{e}_u} } }  \\ \frac{1}{n^r} \prod_{ v=c+1}^{s} \frac{1}{j_v} \mathbf{1} \bigg( \frac{\prod_{v} j_v \prod_u m_{i_u}}{ g_s(\mathbf{e}_1,...,\mathbf{e}_s)}  \in J_k \bigg) \sum\limits_{1 \le u_v \le \frac{4}{\varepsilon}} f_{j_{c+A+1},...,j_s}^{n, u_{c+A+1},...,u_s}\Big(\frac{m_1}{n},...,\frac{m_r}{n} \Big) + \varepsilon(\lambda)
\end{multlined} 
\end{equation} 
where the error term $\varepsilon(\lambda) \to 0$ as $\lambda \to 0$ and the implied absolute constant does not depend on $r$.  
\end{theorem}

Lemmas \ref{L:Poisson lemma} through \ref{L:log2s} below establish the bounds needed to prove Theorem \ref{T:Big}.  

\begin{lemma} \label{L:Poisson lemma} Fix $0 < \lambda < 1$ and a positive integer $m$.  For $1\le j_1 < ... < j_m \le n$, let $c_{j_v}$ be non-negative integers such that $\sum\limits_{v=1}^m j_v c_{j_v} < (1 - \lambda)n$.  Then 
\begin{equation}\label{Poisson approx} \mathbb{P} \bigg( \bigcap_{v=1}^m \{C_{j_v}^{(n)} = c_{j_v} \} \bigg) \ll \mathbb{P}\bigg( \bigcap_{v=1}^m \{Z_{j_v} = c_{j_v} \} \bigg)
\end{equation}
where the absolute constant does not depend on $j_v$ or $c_{j_v}$.  

\end{lemma}

\begin{proof}

Let $B = [n] \setminus \{j_{1},...,j_m\}$. 
Recall from Definition \ref{T_{0n}} that $T_B := \sum\limits_{j \in B} j Z_j$.  

By the size biasing equation \eqref{sizebiasewens}, we have the pair of equations 
\begin{align}
i \mathbb{P}\{T_B = i\} &= \label{tibet} \theta \sum_{l \in B} \mathbb{P}\{T_B = i-l\} \\ 
 i \mathbb{P}\{T_{0 n} = i\} &= \label{tibet2} \theta \sum_{l =1}^{i} \mathbb{P}\{T_{0 n} = i-l\}
\end{align}
Since \[\mathbb{P}\{T_B = 0\} \le \prod_{i=m+1}^n \frac{1}{e^{\theta/i}} \le (3m)^\theta \mathbb{P}\{T_{0 n} = 0\} , \] we get 
\begin{equation}\label{size inequality}
\mathbb{P}\{T_B = l\} \le (3m)^\theta \mathbb{P}\{ T_{0 n} = l \} \end{equation} for all $l \ge 0$ by comparing \eqref{tibet} and \eqref{tibet2} inductively.  By using \eqref{tibet2}, we have (see e.g. \cite[(4.10)]{a})
 \begin{equation}\label{T0nl} \mathbb{P}\{T_{0 n} = l\} = \exp(-\theta h(n+1)) \frac{\theta (\theta+1)...(\theta+l-1)}{l!} \end{equation} for $0 < l \le n$ where $h(n+1)$ as before is the $n$th harmonic number.  Then \begin{equation} \label{T0n} n\mathbb{P}\{T_{0 n} = l\} \to \frac{e^{-\gamma \theta} x^{\theta-1}}{\Gamma(\theta)} \end{equation} uniformly over $\lambda n < l \le n$ and $l/n \to x \in [\lambda, 1]$.  Hence, by the Conditioning Relation \eqref{Conditioning Relation},
\begin{align*}
& \quad \; \mathbb{P} \bigg( \bigcap_{v=1}^m \{C_{j_v}^{(n)} = c_{j_v} \} \bigg)
= \mathbb{P} \bigg( \bigcap_{v=1}^m \{ Z_{j_v} = c_{j_v} \} \biggm| T_{0 n} = n \bigg) \\
&=  \mathbb{P} \bigg(\bigcap_{v=1}^m \{Z_{j_v} = c_{j_v} \} \bigg) \frac{\mathbb{P}\{T_B = n - \sum j_v c_{j_v} \} }{\mathbb{P}\{T_{0 n} = n \}}
\le (3m)^\theta \mathbb{P} \bigg(\bigcap_{v=1}^m \{Z_{j_v} = c_{j_v} \} \bigg) \frac{\mathbb{P} \{T_{0 n} = n - \sum j_v c_{j_v} \} }{\mathbb{P}\{T_{0 n} = n\} } \\
& \ll \mathbb{P} \bigg(\bigcap_{v=1}^m \{Z_{j_v} = c_{j_v} \} \bigg)
\end{align*}   

\end{proof}

\begin{lemma} \label{L:numcyclebound} Fix $\varepsilon > 0$.  Then 
\begin{equation} \sum_{j > n^\varepsilon} \mathbb{P}  \{ C_j^{(n)} > \lceil 2/\varepsilon \rceil \} \ll \frac{1}{n}
\end{equation}
\end{lemma}
\begin{proof}
By using the Conditioning Relation along with \eqref{size inequality} and \eqref{T0nl}, we see that 
\begin{equation} \mathbb{P} \{ C_j^{(n)} = l \} \ll n^{\max(1-\theta, 0)} (\theta/j)^l \end{equation} uniformly over $j$ and $l$.  Let $a$ be a positive integer.  Then \begin{equation} \mathbb{P} \{ C_j^{(n)} > a \} = \sum_{l=a+1}^\infty \mathbb{P} \{ C_j^{(n)} = l\} \ll  
\frac{n^{\max(1-\theta, 0)}}{j^{a+1}} \end{equation}  Setting $a= \lceil 2/ \varepsilon \rceil$ completes the proof. 
\end{proof}

\begin{lemma} \label{L:Ugly Uniform}
Fix real numbers $\displaystyle{0 < \lambda < 1}$ and $\displaystyle{0 < \varepsilon < 1}$ and positive integers $u_{c+A+1},...,u_s$ where $A$ is an integer such that $\displaystyle{0 \le A \le s-c}$.  Let $j_{c+1},...,j_s$ be integers such that $1 \le j_{c+1}<...< j_{c+A} \le n^\varepsilon < j_{c+A+1} < ...< j_s \le \lceil n/r \rceil$.  Let $\Delta_{j_{c+A+1},...,j_s}^{u_{c+A+1},...,u_s}$ be the set of tuples $(m_1,...,m_r)$ such that $m = m_1+...+m_r < (1-\lambda)n - \sum\limits_{v=c+A+1}^s u_v j_v$ and $\max(j_s,\lambda n) < m_r <...< m_1 \le n$.  Set $x_i=m_i/n$.  Then for all $(m_1,...,m_r) \in \Delta_{j_{c+A+1},...,j_s}^{u_{c+A+1},...,u_s}$,  
\begin{equation*}
\begin{multlined}
n^r \mathbb{P}\bigg(\bigcap_{v = c+1}^{c+A} \{C_{j_{v}}^{(n)} > 0 \} \cap \bigcap_{v = c+A+1}^s \{C_{j_{v}}^{(n)} = u_v \} \cap \bigcap_{i=1}^r \{L_{i}^{(n)} = m_i\} \bigg) \\
\ll \mathbb{P}\bigg(\bigcap_{v=c+1}^{s} \{Z_{j_{v}} > 0\} \bigg) f_{j_{c+A+1},...,j_s}^{n, u_{c+A+1},...,u_s}(x_1,...,x_r) + n \bigg(\frac{\theta}{\lambda}\bigg)^r \frac{\theta^{n^{(1-\varepsilon)/2}}}{\lceil n^{(1-\varepsilon)/2} \rceil !}
\end{multlined}
\end{equation*}
where the absolute constant does not depend on $r$ or the choice of integers $j_v$ or $m_i$.  
\end{lemma}

\begin{proof}
Following analogous calculations to the steps in \eqref{An1}, \eqref{An2} and \eqref{An3}, we see that

\begin{equation*}
\mathbb{P}\bigg( \bigcap_{v = c+1}^{c+A} \{C_{j_{v}}^{(n)} > 0 \} \cap \bigcap_{v = c+A+1}^s \{C_{j_{v}}^{(n)} = u_v \} \cap \bigcap_{i=1}^r \{L_{i}^{(n)} = m_i\} \bigg)
= \frac{\theta^r e^{-\theta (h(n+1)-h(m_r))}}{m_1...m_r} \Delta_n 
\end{equation*}
where 
\begin{equation}
\Delta_n = \frac{\mathbb{P}\Big( \{T_{0, m_r-1} = n-m \} \cap \bigcap\limits_{v = c+1}^{c+A} \{Z_{j_{v}} > 0\} \cap \bigcap\limits_{v = c+A+1}^s \{Z_{j_{v}} = u_v \} \Big)}{\mathbb{P}\{T_{0n} = n\}} \label{ugly}
\end{equation}
Let $\displaystyle{B = [m_r - 1] \setminus \{j_{c+1},...,j_s\}}$.  If $Z_{j_v}$ is less than $n^{(1-\varepsilon)/2}$ for $c+1 \le v \le c+A$, then for sufficiently large $n$, $T_{0, m_r-1} = n-m$ implies that $T_B = n - m - \sum\limits_{v=c+A+1}^s u_v j_v - o(n) \ge \lambda n/2 $ (say).  The proof of Lemma \ref{L:uniformbound} shows that the point probabilities $n \mathbb{P} \{T_{0, m_r-1} = l\}$ and $n \mathbb{P} \{T_{0, n} = l\}$ converge uniformly for $l \ge \lambda n/2$.  Then since $\mathbb{P} \{T_B = l \} \le (3m)^\theta \mathbb{P} \{T_{0, m_r-1} = l \}$ by \eqref{size inequality},
\begin{equation*}
\Delta_n \ll \mathbb{P} \bigg( \bigcap_{v = c+1}^{s} \{Z_{j_{v}} > 0 \} \bigg)p_\theta \bigg( \Big(1- \sum \limits_{i=1}^r x_i - \sum \limits_{v=c+A+1}^s u_v\frac{j_v}{n} \Big) \Big/ x_r \bigg) + n \mathbb{P} \bigg( \bigcup_{v = c+1}^{c+A} \{ Z_{j_v} > n^{(1-\varepsilon)/2} \} \bigg)
\end{equation*}

The result follows since $\displaystyle{\mathbb{P}\{Z_i > n^{\delta}\} \ll \frac{\theta^{n^\delta}}{\lceil n^{\delta} \rceil !} }$ for any $\delta > 0$.  

\end{proof}

We define the following two sets:
\begin{align}
\mathscr{B}_1 &= \label{b1def} \{q \in \mathbb{N}: (\exists m,e \in \mathbb{N})[ (m \le (\log n)^{2s}) \land (e > 2 s \log \log n) \land (p_m^e \bigm| q) ]  \}  \\
\mathscr{B}_2 &= \label{b2def} \{q \in \mathbb{N}: \left|\{j : e_{j}(q) > 0\}_{1 \le j \le (\log n)^{2s}} \right| > (\log \log n)^2 \}
\end{align}
where recall $e_j(q)$ is the exponent of the $p_j^\textrm{th}$ prime in the prime factorization of $q$.  Then $\mathbb{N} \setminus \mathscr{Y}_n = \mathscr{B}_1 \cup \mathscr{B}_2$.  The next two lemmas show that as $n \to \infty$, the contribution to the probability for $j_v \notin \mathscr{Y}_n$ is negligible.  

\begin{lemma}\label{L:b1}
\tolerance=10000\hbadness=10000
Let $\mathscr{B}_1$ be as defined in \eqref{b1def}.  Then for $0 < \lambda < 1$,
\[\lim_{n \to \infty} \sum_{\substack{ 1 \le j_{c+1},...,j_{s} \le n \cr  j_{c+1}\neq... \neq j_{s} \cr j_s \in \mathscr{B}_1 } } \mathbb{P}\bigg( \bigcap_{v=c+1}^s \{C_{j_v}^{(n)} > 0\} \cap \bigg \{\sum_{v=c+1}^s j_v C_{j_v}^{(n)} < (1 - \lambda)n \bigg\} \bigg) = 0  \]
where $ j_{c+1} \neq... \neq j_{s}$ is shorthand for $j_{c+1},...,j_s$ all distinct.
\end{lemma}

\begin{proof}

By Lemma \ref{L:Poisson lemma}, 
\begin{align*}
& \sum_{\substack{ 1 \le j_{c+1},...,j_{s} \le n \cr  j_{c+1}\neq... \neq j_{s} \cr j_s \in \mathscr{B}_1 } } \mathbb{P}\bigg( \bigcap_{v=c+1}^s \{C_{j_v}^{(n)} > 0\}  \cap \bigg \{\sum_{v=c+1}^s j_v C_{j_v}^{(n)} < (1 - \lambda)n \bigg\} \bigg) \\
& \le \sum_{m = 1 }^{(\log n)^{2s} } \sum_{\substack{ 1 \le j_{c+1},...,j_{s} \le n \cr  j_{c+1}\neq... \neq j_{s} \cr p_m^{\lceil 2 s \log \log n \rceil} \bigm | j_s } }  \mathbb{P}\bigg( \bigcap_{v=c+1}^s \{C_{j_v}^{(n)} > 0\} \cap \bigg \{\sum_{v=c+1}^s j_v C_{j_v}^{(n)} < (1 - \lambda)n \bigg\} \bigg) \\
& \ll \sum_{m = 1 }^{(\log n)^{2s} }\sum_{\substack{ 1 \le j_{c+1},...,j_{s-1} \le n } } \frac{1}{j_{c+1}...j_{s-1}} \sum_{i=1}^{ \frac{n}{ p_m^{ 2 s \log \log n }} } \frac{1}{i p_m^{2 s \log \log n} }  \\ 
& \ll (\log n)^s \sum_{m=1}^{(\log n)^{2s} } \frac{1}{ p_m^{2 s \log \log n}} \\
& \ll (\log n)^s \frac{1}{2^{2 s \log \log n - 1}} \to 0
\end{align*}
as $n \to \infty.$ 

\end{proof}

Now we turn to the second set $\mathscr{B}_2$.  In words, $\mathscr{B}_2$ is the set of integers with more than $(\log \log n)^2$ distinct prime factors among the first $(\log n)^{2s}$ primes.  Let $\omega(m)$ denote the number of distinct prime factors of an integer $m$.  The Erd\H{o}s-Kac theorem (see e.g. \cite{durrett} or \cite{tenenbaum}) tells us that if $x$ is a randomly chosen integer from $[n]$, then $\displaystyle{\frac{\omega(x) - \log \log x}{\sqrt{\log \log x}} \buildrel d \over \to \mathcal{N}(0,1)}$ as $n \to \infty$.  Thus, the number of prime factors of $m$ is concentrated about $\log \log m$ and one would expect $\omega(m) > (\log \log m)^2$ to be a rare event.   
We need a large deviation inequality to bound this probability.  

Let $\mathbb{P}_{a,b}$ be the uniform distribution on the integers in $(a, b]$ and let $\mu_m = \mathbb{E}[\omega(x)]$ where the expectation is taken under the measure $\mathbb{P}_{0,m}$.  Let $\delta_p(m) = 1$ if $m$ is divisible by $p$ and $= 0$ otherwise.  Then (see e.g. \cite[Chapt. 22]{hardy})
 \begin{equation} \mu_m = \mathbb{E}[\omega(x)] = \mathbb{E} \bigg[ \sum_{p \le m} \delta_p(x) \bigg] = \sum_{p \le m} \frac{\lfloor m/p \rfloor}{m} = \log \log m + O(1) \end{equation}   
Srinivasan \cite{srinivasan} has derived an explicit exponential Chernoff-type bound for large deviations of $\omega(m)$:  

\begin{proposition}[Srinivasan] \label{P:srinivasan}

For $m \ge 2$ and any $\delta > 0$, \begin{equation} \mathbb{P}_{0, m} \{ \omega(x) \ge \mu_m (1 + \delta)\} \le \bigg(\frac{e^{\delta}}{(1+ \delta)^{1 + \delta}}\bigg)^{\mu_m'} \end{equation}
where $\displaystyle{\mu_m' = \sum_{ p \le m} \frac{1}{p} = \log \log m + O(1)}$.

\end{proposition}

Note that this is a bound for each $m$ rather than an asymptotic bound.  For more information on Erd\H{o}s-Kac large deviation results, see for instance \cite{zhu, radziwill, hildebrand, pomerance}.

\begin{definition}
The $n^{\textrm{th}}$ primorial $p_n \#$ is the product of the first $n$ primes, i.e. \begin{equation*} p_n \# := \prod_{m=1}^n p_m\end{equation*}
\end{definition}

\begin{remark}
By the Prime Number Theorem, we have the well-known asymptotic \begin{equation}\label{PNT} p_n \# = \exp( (1 + o(1))n \log n) \end{equation}

\end{remark}

\begin{lemma}\label{L:b2}
\tolerance=10000\hbadness=10000
Let $\mathscr{B}_2$ be as defined in \eqref{b2def}.  Then for $0 < \lambda < 1$,
\[\lim_{n \to \infty} \sum_{\substack{ 1 \le j_{c+1},...,j_{s} \le n \cr  j_{c+1}\neq... \neq j_{s} \cr j_s \in \mathscr{B}_2 } } \mathbb{P}\bigg( \bigcap_{v=c+1}^s \{C_{j_v}^{(n)} > 0\} \cap \bigg \{\sum_{v=c+1}^s j_v C_{j_v}^{(n)} < (1 - \lambda)n \bigg\} \bigg) = 0  \]

\end{lemma}

\begin{proof}

By \eqref{PNT}, \[c_n := \prod_{m=1}^{(\log \log n)^2} p_m = \exp((1 + o(1)) 2 (\log \log n)^2 \log \log \log n) \]

For each positive integer $m$, define the interval $I_m = (2^{m-1} c_n, 2^{m} c_n]$.  Let $I_M$ be the interval that contains the integer $n$.  Then $M \le 2\log n$.

We set $\delta = \log \log n$ in Proposition \ref{P:srinivasan}.  
For $c_n < m \le n$ and sufficiently large $n$,
\[\mathbb{P}_{0,m} \big\{ \omega(x) \ge (\log \log n)^2 \big\} \le \frac{1}{(\log \log n)^{(\log \log n)(\log \log \log n)}} \]
and therefore for $m \le M$, \[\mathbb{P}_{I_m} \big\{ \omega(x) \ge (\log \log n)^2 \big\} \le \frac{2}{(\log \log n)^{(\log \log n)(\log \log \log n)}} \le \frac{1}{(\log n)^{2s}} \]
for sufficiently large $n$.

Then \[\sum_{\substack{1 \le j \le n \cr \omega(j)> (\log \log n)^2}} \frac{1}{j} = \sum_{\substack{c_n \le j \le n \cr \omega(j)> (\log \log n)^2}} \frac{1}{j} \le \sum_{m=1}^{M} \sum_{\substack{j \in I_m \cr \omega(j)> (\log \log n)^2}} \frac{1}{j} \le 2 \log n \frac{1}{(\log n)^{2s}}. \]

The result then follows from Lemma \ref{L:Poisson lemma}. 

\end{proof}
The following is immediate. 

\begin{corollary}\label{C:c1}
\tolerance=10000\hbadness=10000
If $c > 0$, let $1 \le i_1 < ...< i_c \le r$.  Then \[\lim_{n \to \infty} \mathbb{P} \Big( \bigcup_{1 \le u \le c} \{L_{i_u}^{(n)} \notin \mathscr{Y}_n \} \Big) = 0 .\]

\end{corollary}

\begin{definition}
Let $l = \lfloor (\log n)^{2s} \rfloor$. For each $t \in \mathbb{N}$, we define the set of tuples \begin{equation}\label{scriptX} \mathscr{X}_n^t = \big\{(m_1,...,m_t) \in \mathbb{N}^t:  \gcd(m_i, m_j) = g_2(\mathbf{E}_l(m_i,m_j)) \textrm{ for } 1 \le i < j \le t \big\} \end{equation} 
\end{definition}

Lemmas \ref{L:b1} and \ref{L:b2} and Corollary \ref{C:c1} provided bounds on the probability of avoiding the ``good set'' $\mathscr{Y}_n$. 
The next two lemmas deal with the ``good sets'' $\mathscr{X}_n^t$ from \eqref{scriptX}.  In words, $\mathscr{X}_n^t$ consists of the tuples $(m_1,...,m_t)$ of integers that are pairwise relatively prime with respect to primes $p_m$ such that $m > (\log n)^{2s}$. 

\begin{lemma}\label{L:c2}
If $c > 0$, let $\displaystyle{1 \le i_1 < ...<i_c \le r}$.  Then
\[\lim_{n \to \infty} \mathbb{P}\{ (L_{i_1}^{(n)},...,L_{i_c}^{(n)}) \notin \mathscr{X}_n^c \} = 0.  \]
\end{lemma}

\begin{proof}
This follows from Lemma \ref{L:firstl}.   
\end{proof}

\begin{lemma}\label{L:log2s}
If $c > 0$, let $\displaystyle{1 \le i_1 < ...<i_c \le r}$.  If $s > 1$, let $l = \lfloor (\log n)^{2s} \rfloor$.  Define the sets (where appropriate)
\begin{align*}
\mathscr{A}_1 &= \bigcup_{\substack{1 \le u \le c \\ c+1 \le v \le s}} \{ \gcd(L_{i_u}^{(n)}, j_v) \neq g_2(\textbf{E}_l(L_{i_u}^{(n)}, j_v)) \} \\
\mathscr{A}_2 &= \bigcup_{c+1 \le u < v \le s} \{ \gcd(j_u, j_v) \neq g_2(\textbf{E}_l(j_u,j_v)) \} 
\end{align*}
Then for $0 < \lambda < 1$,
\begin{equation} \label{log2s}
\begin{multlined} \lim_{n \to \infty} \sum_{\substack{ 1 \le j_{c+1},...,j_s \le n \cr j_{c+1}\neq... \neq j_{s} } } \mathbb{P}\bigg( \bigcap_{v=c+1}^s \{C_{j_v}^{(n)} > 0\} \cap \bigcap_{i=1}^r \{C_{L_i^{(n)}}^{(n)} = 1\} \\  \cap \bigg\{\sum_{i=1}^r L_i^{(n)} + \sum_{v=c+1}^s j_v C_{j_v}^{(n)} < (1 - \lambda)n \bigg\} \cap (\mathscr{A}_1 \cup \mathscr{A}_2) \bigg) = 0 \end{multlined}
\end{equation}
\end{lemma}

\begin{proof}

First, consider the set $\mathscr{A}_1$.  By Lemma \ref{L:Poisson lemma}, the expression in \eqref{log2s} with $\mathscr{A}_1 \cup \mathscr{A}_2$ replaced by $\mathscr{A}_1$ can be bounded by
\begin{align*}
& \begin{multlined} \ll \sum_{\substack{1 \le j_{c+1},...,j_{s-1} \le n \cr  j_{c+1}\neq... \neq j_{s-1} \cr m > (\log n)^{2s} } } 
\sum_{\substack{ p_m \mid j_s \cr p_m \mid j_{s+1} }} \mathbb{P} \bigg( \bigcap_{v=c+1}^{s} \{C_{j_v}^{(n)} > 0\} \cap \{C_{j_{s+1}}^{(n)} = 1\} \cap \bigg\{ \sum_{v=c+1}^{s+1} j_v C_{j_v}^{(n)} < (1 - \lambda)n \bigg\} \bigg) \end{multlined} \\
& \ll \sum_{\substack{1 \le j_{c+1},...,j_{s-1} \le n \cr  j_{c+1}\neq... \neq j_{s-1} } } \frac{1}{j_{c+1}...j_{s-1}} \sum_{m > (\log n)^{2s}} 
 \sum_{i=1}^{ n/p_m} \sum_{j=1}^{ n/p_m} \frac{1}{p_m i} \frac{1}{p_m j}  \\
& \ll (\log n)^{s+1}  \sum_{m > (\log n)^{2s}} \frac{1}{p_m^2} \ll (\log n)^{s+1} \frac{1}{(\log n)^{2s}} \to 0
\end{align*}  
as $n \to \infty$.  The bound for $\mathscr{A}_2$ is immediate.  
\end{proof}

Now we are ready to put the ingredients together to obtain the bound in \eqref{big equation}.

\begin{proof}[Proof of Theorem \ref{T:Big}]
Let $l = \lfloor (\log n)^{2s} \rfloor$.  If $c=0$, the appropriate summations and events in the expressions below are ignored.
By Corollary \ref{C:c1} and Lemma \ref{L:c2},
\begin{align}
& \nonumber \qquad \mathbb{P} \left\{ H_{n,s}^{r,c} \cap J_k(n) \neq \varnothing \right\} \\
&\begin{multlined} < \nonumber \mathbb{P} \bigg( \bigcup_{\substack{i_1 <... < i_c \le r \cr r < i_{c+1} <... <i_s}} \{\lcm(L_{i_1}^{(n)},...,L_{i_s}^{(n)}) \in J_k \} \cap \bigcap_{1 \le u < v \le s}  \{ L_{i_u}^{(n)} \neq L_{i_v}^{(n)} \} \\ \cap \{L_r^{(n)} > \lambda n\} \cap \{L_1^{(n)} +...+ L_{4(r+s)/\varepsilon}^{(n)} < (1- \lambda) n\} \bigg) + \varepsilon(\lambda) \end{multlined}  \\
& \begin{multlined} \lesssim \nonumber \sum_{\mathbf{e}_1,...,\mathbf{e}_c \in Q_n } \mathbb{P} \bigg( \bigcup_{\substack{i_1 <... < i_c \le r \cr r < i_{c+1} <... <i_s}} \{\lcm(L_{i_1}^{(n)},...,L_{i_s}^{(n)}) \in J_k \} \cap \bigcap_{u=1}^c \{L_{i_u}^{(n)} \in \mathscr{P}_{\mathbf{e}_u}\} \cap  \bigcap_{1 \le u < v \le s}  \{ L_{i_u}^{(n)} \neq L_{i_v}^{(n)} \}  \\  \cap \{(L_{i_1}^{(n)},...,L_{i_c}^{(n)}) \in \mathscr{X}_n^c \} \cap \{L_r^{(n)} > \lambda n\} \cap \{L_1^{(n)} +...+ L_{4(r+s)/\varepsilon}^{(n)} < (1- \lambda) n\} \bigg) + \varepsilon(\lambda) \end{multlined} \label{eqnbig}
\end{align} 
where $\varepsilon(\lambda) \to 0$ as $\lambda \to 0$.  Applying Lemma \ref{L:numcyclebound}, this is bounded by 
\begin{equation*}
\begin{multlined} \lesssim 
 \sum_{\substack{ 1 \le i_1 < ...< i_c \le r \cr \mathbf{e}_1,...,\mathbf{e}_c \in Q_n \cr 0 \le A \le s-c} } \sum_{\substack{ 1 \le j_{c+1}<...<j_s \le \lceil n/r \rceil  \cr j_{c+A} \le n^\varepsilon < j_{c+A+1} }} \mathbb{P}\bigg( \bigcap_{ v=c+1}^{s} \{C_{j_v}^{(n)} > 0\} \cap \bigcap_{ v=c+A+1}^s \{C_{j_v}^{(n)} < 4/ \varepsilon\} \cap \bigcap_{i=1}^r \{C_{L_i^{(n)}}^{(n)} = 1\} \\ \cap \bigcap_{u=1}^c \{L_{i_u}^{(n)} \in \mathscr{P}_{\mathbf{e}_u}\} \cap \{(L_{i_1}^{(n)},...,L_{i_c}^{(n)}) \in \mathscr{X}_n^c \} \cap \{\lcm( L_{i_1}^{(n)},...,L_{i_c}^{(n)}, j_{c+1},...,j_s) \in J_k\} \\ \cap \{L_r^{(n)} > \max(j_s, \lambda n) \} \cap  \bigg\{\sum_{i=1}^r L_i^{(n)} + \frac{4}{\varepsilon} \sum_{v=c+1}^s j_v < (1- \lambda) n \bigg\} \bigg) + \varepsilon(\lambda) \end{multlined}
\end{equation*}

Then applying Lemmas \ref{L:b1}, \ref{L:b2} and \ref{L:log2s}, this expression is bounded by 
\begin{equation*}
\begin{multlined}
\lesssim \sum_{\substack{ 1 \le i_1 < ...< i_c \le r \cr \mathbf{e}_1,...,\mathbf{e}_s \in Q_n \cr 0 \le A \le s-c} } \sum_{\substack{ 1 \le j_{c+1}<...<j_s \le \lceil n/r \rceil  \cr j_{c+A} \le n^\varepsilon < j_{c+A+1} \cr j_v \in \mathscr{P}_{\mathbf{e}_v} }}  \mathbb{P} \bigg( \bigcap_{ v=c+1}^s \{C_{j_v}^{(n)} > 0\}  \cap \bigcap_{ v=c+A+1}^s \{C_{j_v}^{(n)} < 4/ \varepsilon\}  \cap \bigcap_{u=1}^c \{L_{i_u}^{(n)} \in \mathscr{P}_{\mathbf{e}_u}\} \\ \cap \bigg\{\frac{\prod_{v} j_v \prod_u L_{i_u}^{(n)}}{ g_s(\mathbf{e}_1,...,\mathbf{e}_s)}  \in J_k \bigg\} \cap \{L_r^{(n)} > \max(j_s, \lambda n) \} \cap \bigg\{\sum_{i=1}^r L_i^{(n)} + \frac{4}{\varepsilon} \sum_{v=c+1}^s j_v < (1- \lambda)n \bigg\} \bigg) + \varepsilon(\lambda)
\end{multlined} 
\end{equation*}

Finally, applying Lemma \ref{L:Ugly Uniform} completes the proof. 
 
\end{proof}
Recall that in the proof of the one-dimensional $k=1$ case, the contribution from the remainder measure $\eta_n^r$ was bounded by a sum of the form $\sum\limits_{j=1}^{n/r} \frac{1}{j} \mathbf{1}(j \in J_1)$.  This sum was in turn dealt with in \eqref{Jfrac} by partitioning the range of summation into intervals $I_{m,b}$  defined in \eqref{bin} and using equidistribution to estimate the number of $j$ in $|J_1 \cap I_{m,b}|$.  
To bound the sum in \eqref{big equation}, we similarly proceed by partitioning the range of summation into boxes and estimate for each tuple $(\mathbf{e}_1,...,\mathbf{e}_s)$ of vectors in $Q_n$ the count of the number of points $(j_v, m_{i_u})$ in each box such that $\displaystyle{j_v \in \mathscr{P}_{\mathbf{e}_v}}$ and  $\displaystyle{m_{i_u} \in \mathscr{P}_{\mathbf{e}_u}}$ and $\frac{\prod_{v} j_v \prod_u m_{i_u}}{ g_s(\mathbf{e}_1,...,\mathbf{e}_s)}  \in J_k$.  The definition of the intervals used in this multidimensional partition are somewhat more involved than in the $k=1$ case.  For the $m_i$ indices, we will use the intervals $I_{m,b}$ where we set $b = \log n$ (any choice of $b$ such that $b$, $n/b \to \infty$ will suffice).  For the $j_v$ indices, we define new intervals $T_m$ as follows: 
\begin{definition}
Let $\mu = \mu(\alpha)$ be the irrationality measure of $\alpha$.  Set $\varepsilon = 1/(100 k^2 \mu^3)$ and $\delta = 10 \varepsilon k \mu^2$.  
Let $M'$ be the smallest integer that satisfies the inequality $2^{M'} n^{\varepsilon} > n/r$ and let $M'' = \lfloor M' - \log \log n \rfloor$.  Then for $1 \le m \le M''$, define $T_m = [2^{m-1} n^\varepsilon, 2^{m} n^{\varepsilon}]$.  For $m > M''$, define $T_m$ to be consecutive intervals of integers all of length $2^{M''} n^{\varepsilon}$ starting with $T_{M''+1} = [2^{M''} n^\varepsilon, 2^{M''+1} n^{\varepsilon}]$.  Define $M$ to be the smallest integer that satisfies the inequality $2^{M} n^{\varepsilon} > n^{\delta}$ and define $M^*$ to be the integer such that $T_{M^*}$ contains $\lceil n/r \rceil$.  For $1 \le m \le M^*$, write $T_m = (F_m, F_{m+1}]$ where $F_m \in \mathbb{Z}$. 
\end{definition}
\begin{remark} The seemingly arbitrary definition of the parameters $\varepsilon$ and $\delta$ and the intervals $T_m$ is motivated by the need for them to satisfy several properties.  Most crucially, they need to be defined so that a satisfactory discrepancy bound can be obtained uniformly over all the boxes in the partition (proved in the next section).  
The parameters $\varepsilon$ and $\delta$ are chosen so that the multi-dimensional Weyl inequality in Lemma \ref{L:WeylBound} can be applied to obtain a non-trivial bound.  For this, we need $\delta > 2 \varepsilon k \mu$ while being small enough so that the interval $1 \le j \le n^{k \delta}$ does not contain any $j \in J_k$ for sufficiently large $n$.  Also, we need that $|T_m| \ge n^\delta$ for $m > M$.  We see that these two properties are indeed satisfied.  The first follows from the definition of irrationality measure of $\alpha$ and the second is a result of the geometric progression in the interval lengths up to $M''$.  We taper the interval lengths for $m > M''$ so that all the $T_m$ have cardinality at most $\displaystyle{2^{M''} n^{\varepsilon} \le n/\sqrt{\log n} = o(n)}$ so that \eqref{equicontinuous} below is satisfied. 
  
\end{remark}
   
\begin{proof}[Proof of Theorem \ref{T:technical}]
By Theorem \ref{T:Big}, $\mathbb{P} \left\{ H_{n,s}^{r,c} \cap J_k(n) \neq \varnothing \right\}$ is bounded by a sum $\sum\limits_{0 \le A \le s-c} S_A$ where each $S_A$ is a sum over ``small" indices $j_v$ for $c+1 \le v \le c+A$ which are all bounded by $n^\varepsilon$.  Splitting further, we can write this as a sum $\sum\limits_{\substack{A,B,C \ge 0 \cr s-c = A+B+C}} S_{A,B,C}$ where each term $S_{A,B,C}$ is a sum over the following ranges:  for $c+1 \le v \le c+A$ we have $1 \le j_v \le n^{\varepsilon}$; for $c+A+1 \le v \le c+A+B$ we have $j_v \in T_{l_v}$ where $1 \le l_v < M$; for $c+A+B+1 \le v \le s$ we have $j_v \in T_{l_v}$ where $l_v \ge M$. 

Note that the ranges of the variables $A$,$B$,$C$, and $u_{c+A+1},...,u_s$ depend only on the dimension $k$ and irrationality measure $\mu$.  Therefore, it suffices to consider a fixed choice of these parameters.  If $C+c = 0$, then $\frac{\prod_{v} j_v \prod_u m_{i_u}}{ g_s(\mathbf{e}_1,...,\mathbf{e}_s)} \le n^{k \delta}$ and hence $\frac{\prod_{v} j_v \prod_u m_{i_u}}{ g_s(\mathbf{e}_1,...,\mathbf{e}_s)} \notin J_k$ for sufficiently large $n$.  Thus, we may assume that $C+c \ge 1$.  

Let $\Delta_{j_{c+A+1},...,j_s}^{u_{c+A+1},...,u_s}$ be as defined in Lemma \ref{L:Ugly Uniform}.  Since $|T_{l_v}| = o(n)$, the collection of functions $f_{j_{c+A+1},...,j_s}^{n, u_{c+A+1},...,u_s}(x_1,...,x_r)$ satisfy 
\begin{equation}\label{equicontinuous} \left|f_{F_{l_{c+A+1}},...,F_{l_s}}^{n, u_{c+A+1},...,u_s}\Big(\frac{k_1}{b},...,\frac{k_r}{b} \Big) - f_{j_{c+A+1},...,j_s}^{n, u_{c+A+1},...,u_s}\Big(\frac{m_1}{n},...,\frac{m_r}{n} \Big)\right| < \varepsilon(n) 
\end{equation} for $j_v \in T_{l_v}$, $(m_1,...,m_r) \in \Delta_{j_{c+A+1},...,j_s}^{u_{c+A+1},...,u_s}$ and $m_i \in I_{k_i, b}$, and for some error function $\varepsilon(n) \to 0$ as $n \to \infty$.

Fix a choice of $\displaystyle{1 \le i_1 < ...< i_c \le r}$.  We have (where the condition $j_v \in T_{l_v}$ is over the range $c+A+1 \le v \le s$ and the condition $j_v \in \mathscr{P}_{\mathbf{e}_v}$ is over the range $c+1 \le v \le s$)
\begin{equation}
\begin{multlined}
\sum_{\substack{\mathbf{e}_1,...,\mathbf{e}_s \in Q_n}} \sum_{\substack{ 1 \le j_{c+1}<...<j_s \le \lceil n/r \rceil  \cr j_{c+A} \le n^\varepsilon < j_{c+A+1} \cr j_v \in \mathscr{P}_{\mathbf{e}_v} }} \sum_{\substack{\lambda n < m_r < ...< m_1 \le n \cr \sum m_i + \frac{4}{\varepsilon} \sum j_v < (1-\lambda)n \cr  m_{i_u} \in \mathscr{P}_{\mathbf{e}_u} } }  \\ \frac{1}{n^r} \prod_{ v=c+1}^{s} \frac{1}{j_v} \mathbf{1} \bigg( \frac{\prod_{v} j_v \prod_u m_{i_u}}{ g_s(\mathbf{e}_1,...,\mathbf{e}_s)}  \in J_k \bigg) f_{j_{c+A+1},...,j_s}^{n, u_{c+A+1},...,u_s}\Big(\frac{m_1}{n},..., \frac{m_r}{n} \Big)
\end{multlined} 
\end{equation} 

\begin{align}
& \label{wonder} \begin{multlined}
\ll \sum_{\substack{\mathbf{e}_1,...,\mathbf{e}_s \in Q_n \cr 1 \le l_{c+A+1} \le...\le l_{c+A+B} < \cr M \le l_{c+A+B+1} \le ...\le l_s \le M^*}} \sum_{\substack{ 1 \le j_{c+1}<...<j_s \le \lceil n/r \rceil \cr j_{c+1},...,j_{c+A} \le n^\varepsilon \cr j_v \in T_{l_v} \cr j_v \in \mathscr{P}_{\mathbf{e}_v} }}  \sum_{\substack{\lambda b \le k_r \le ...\le k_1 \le b \cr \sum \frac{k_i}{b} + \frac{4}{\varepsilon} \sum \frac{j_v}{n} < 1-\lambda }} \sum_{\substack{m_i \in I_{k_i, b} \cr   m_{i_u} \in \mathscr{P}_{\mathbf{e}_u} }}  \\  \frac{1}{n^r} \prod_{ v=c+1}^{s} \frac{1}{j_v} \mathbf{1} \bigg( \frac{\prod_{v} j_v \prod_u m_{i_u}}{ g_s(\mathbf{e}_1,...,\mathbf{e}_s)}  \in J_k \bigg) f_{j_{c+A+1},...,j_s}^{n, u_{c+A+1},...,u_s}\Big(\frac{m_1}{n},..., \frac{m_r}{n} \Big)
\end{multlined}
\end{align}

Now for each choice of $l_{c+A+1},...,l_s$ and $k_1,...,k_r$, define the fraction 
\begin{equation} \gamma^{\mathscr{P}_{\mathbf{e}_{1},...,\mathbf{e}_{s}}}_{(k_i), (l_v), b} := \prod_{v=c+A+1}^s \frac{ |T_{l_v} \cap \mathscr{P}_{\mathbf{e}_v} |}{ |T_{l_v}|} \prod_{u=1}^c \frac{|I_{k_{i_u}, b} \cap \mathscr{P}_{\mathbf{e}_u} |}{ |I_{k_{i_u}, b}|}  \end{equation}

and the multiset
\begin{equation} \label{compdisc}
\begin{multlined}\mathcal{P}^{\mathscr{P}_{\mathbf{e}_1,...,\mathbf{e}_s}}_{(k_i), (l_v), (j_w) } := \bigg \{\prod_{u=1}^c m_{i_u} \prod\limits_{v=c+A+1}^s j_{v} \frac{ \prod_{w=c+1}^{c+A} j_w}{g_s(\mathbf{e}_1,....,\mathbf{e}_s)} \alpha \in \mathbb{T}: \\ \bigwedge_{u=1}^c \{m_{i_u} \in (I_{k_{i_u}, b} \cap \mathscr{P}_{\mathbf{e}_u}) \} \land \bigwedge_{v=c+A+1}^s \{j_v \in (T_{l_v} \cap \mathscr{P}_{\mathbf{e}_v}) \} \bigg\} 
\end{multlined}
\end{equation}  
Recall that $j \in J_k$ if and only if $\displaystyle{ \{j \alpha \} \in [0, jT/n^k) \cup (1 - jT/n^k, 1] }$ and therefore we are interested in the fraction of points $\{j \alpha \}$ in $\mathcal{P}^{\mathscr{P}_{\mathbf{e}_1,...,\mathbf{e}_s}}_{(k_i), (l_v), (j_w) }$ that satisfy that property.  By definition, this fraction is bounded by 
\begin{equation} \frac{2T}{n^{k-s}} \prod_{u=1}^c \frac{k_{i_u}}{b}\prod_{w=c+1}^{c+A} \frac{j_w}{n} \prod_{v=c+A+1}^{s} \frac{ F_{l_v+1}}{n} 
\end{equation}
up to the discrepancy $D\big(\mathcal{P}^{\mathscr{P}_{\mathbf{e}_1,...,\mathbf{e}_s}}_{(k_i), (l_v), (j_w) } \big)$. 
Then the expression \eqref{wonder} is bounded by
\begin{equation} \label{wonder2}
\begin{multlined}
\ll \sum_{\substack{\mathbf{e}_1,...,\mathbf{e}_s \in Q_n \cr 1 \le l_{c+A+1} \le...\le l_{c+A+B} < \cr M \le l_{c+A+B+1} \le ...\le l_s \le M^* }} \sum_{\substack{1 \le j_{c+1},...,j_{c+A} \le n^{\varepsilon} \cr j_w \in \mathscr{P}_{\mathbf{e}_w}}} \sum_{\substack{\lambda b \le k_r \le ...\le k_1 \le b \cr \sum \frac{k_i}{b} + \frac{4}{\varepsilon} \sum \frac{F_{l_v}}{n} < 1 - \lambda} }  \\ \gamma^{\mathscr{P}_{\mathbf{e}_{1},...,\mathbf{e}_{s}}}_{(k_i), (l_v), b} \prod_{v=c+A+1}^{s} |T_{l_v}| \bigg(\frac{2T}{n^{k-s}} \prod_{u=1}^c \frac{k_{i_u}}{b}\prod_{w=c+1}^{c+A} \frac{j_w}{n} \prod_{v=c+A+1}^{s} \frac{ F_{l_v+1}}{n} +  D\big(\mathcal{P}^{\mathscr{P}_{\mathbf{e}_1,...,\mathbf{e}_s}}_{(k_i), (l_v), (j_w) }\big)  \bigg) \\ \frac{1}{b^r}  \prod_{w=c+1}^{c+A} \frac{1}{j_w} \prod_{v = c+A+1}^s \frac{1}{F_{l_v} } \Big(f_{F_{l_{c+A+1}},...,F_{l_s}}^{n, u_{c+A+1},...,u_s}\Big(\frac{k_1}{b},..., \frac{k_r}{b} \Big) + \varepsilon(n) \Big)
\end{multlined}
\end{equation} 

We split \eqref{wonder2} into two terms:
\begin{equation}
\begin{multlined}
A_1 :=  \sum_{\substack{\mathbf{e}_1,...,\mathbf{e}_s \in Q_n \cr 1 \le l_{c+A+1} \le...\le l_{c+A+B} < \cr M \le l_{c+A+B+1} \le ...\le l_s \le M^* }} \sum_{\substack{1 \le j_{c+1},...,j_{c+A} \le n^{\varepsilon} \cr j_w \in \mathscr{P}_{\mathbf{e}_w}}} \sum_{\substack{\lambda b \le k_r \le ...\le k_1 \le b \cr \sum \frac{k_i}{b} + \frac{4}{\varepsilon} \sum \frac{F_{l_v}}{n} < 1 - \lambda} }  \\  \gamma^{\mathscr{P}_{\mathbf{e}_{1},...,\mathbf{e}_{s}}}_{(k_i), (l_v), b} \prod_{v=c+A+1}^{s} |T_{l_v}| \bigg(\frac{2T}{n^{k-s}} \prod_{u=1}^c \frac{k_{i_u}}{b}\prod_{w=c+1}^{c+A} \frac{j_w}{n} \prod_{v=c+A+1}^{s} \frac{ F_{l_v+1}}{n}  \bigg) \\ \frac{1}{b^r}  \prod_{w=c+1}^{c+A} \frac{1}{j_w} \prod_{v = c+A+1}^s \frac{1}{F_{l_v} } \Big(f_{F_{l_{c+A+1}},...,F_{l_s}}^{n, u_{c+A+1},...,u_s}\Big(\frac{k_1}{b},..., \frac{k_r}{b} \Big) + \varepsilon(n) \Big)
\end{multlined}
\end{equation}
and
\begin{equation}
\begin{multlined} 
A_2 :=  \sum_{\substack{\mathbf{e}_1,...,\mathbf{e}_s \in Q_n \cr 1 \le l_{c+A+1} \le...\le l_{c+A+B} < \cr M \le l_{c+A+B+1} \le ...\le l_s \le M^* }} \sum_{\substack{1 \le j_{c+1},...,j_{c+A} \le n^{\varepsilon} \cr j_w \in \mathscr{P}_{\mathbf{e}_w}}} \sum_{\substack{\lambda b \le k_r \le ...\le k_1 \le b \cr \sum \frac{k_i}{b} + \frac{4}{\varepsilon} \sum \frac{F_{l_v}}{n} < 1 - \lambda} } \\  \gamma^{\mathscr{P}_{\mathbf{e}_{1},...,\mathbf{e}_{s}}}_{(k_i), (l_v), b} \bigg(\prod_{v=c+A+1}^{s} |T_{l_v}| \bigg)  D\big(\mathcal{P}^{\mathscr{P}_{\mathbf{e}_1,...,\mathbf{e}_s}}_{(k_i), (l_v), (j_w) } \big) \\ \frac{1}{b^r} \prod_{w=c+1}^{c+A} \frac{1}{j_w} \prod_{v = c+A+1}^s \frac{1}{F_{l_v} } \Big(f_{F_{l_{c+A+1}},...,F_{l_s}}^{n, u_{c+A+1},...,u_s}\Big(\frac{k_1}{b},..., \frac{k_r}{b} \Big) + \varepsilon(n) \Big)
\end{multlined}
\end{equation}
For the first term $A_1$, note that $\frac{F_{{l_v}+1}}{F_{l_v}} \le 2$ and that $\prod\limits_{u=1}^c \frac{k_{i_u}}{b} \le \frac{1}{i_1...i_c}$.  Then summing over $\mathbf{e}_i$, $k_i$, $j_w$ and $l_v$, we see that 
\begin{equation}
A_1 \ll \frac{1}{n^{k-s}} \frac{1}{i_1...i_c} \frac{1}{r^{s-c}} 
\end{equation}
where of course the implicit constant does not depend on $r$.  (Here, if $c=0$ then the $\frac{1}{i_1...i_c}$ term is omitted.)   

For the second term $A_2$, note that each index $l_v$ has a maximum range going from 1 to $M^* \ll \log n$.  Then using the bound on $D\big(\mathcal{P}^{\mathscr{P}_{\mathbf{e}_1,...,\mathbf{e}_s}}_{(k_i), (l_v), (j_w) } \big)$ proved in Theorem \ref{T:maindisc} below, we have
\begin{equation}
A_2 \ll \Big(\frac{100 s \log \log \log n}{\log \log n} \Big)^{\log \log n} (\log n)^s \to 0
\end{equation}

Putting this all together, the RHS of \eqref{big equation} is bounded by
\begin{equation} \ll \sum_{1 \le i_1 < ...< i_c \le r} \frac{1}{i_1...i_c} \frac{1}{r^{s-c}} + \varepsilon(\lambda) \ll \frac{(\log r)^c}{r^{s-c}} + \varepsilon(\lambda) 
\end{equation} 
Taking $\lambda \to 0$ completes the proof. 
\end{proof}

\section{Bound for Discrepancy} \label{S:logdisc}

It remains to obtain a bound for the discrepancy $\displaystyle{D\big(\mathcal{P}^{\mathscr{P}_{\mathbf{e}_1,...,\mathbf{e}_s}}_{(k_i), (l_v), (j_w) } \big) }$ of the set defined in \eqref{compdisc}.  The following bound is not intended to be tight.

\begin{theorem}\label{T:maindisc}
For all choices of $A, B,$ and $C$ (such that $C+c \ge 1$), the discrepancy satisfies the asymptotic bound given by 
\begin{equation}
D\big(\mathcal{P}^{\mathscr{P}_{\mathbf{e}_1,...,\mathbf{e}_s}}_{(k_i), (l_v), (j_w) } \big) \ll \Big( \frac{100 s \log \log \log n}{\log \log n} \Big)^{\log \log n}
\end{equation}
uniformly over all $j_w \in \mathscr{P}_{\mathbf{e}_w}$ and factorization classes $\mathscr{P}_{\mathbf{e}_1,...,\mathbf{e}_s}$ and boxes given by the subscripts $k_i$ and $l_v$. 
\end{theorem}
Recall that in the one-dimensional case, the discrepancy asymptotic from Proposition \ref{P:1Ddisc} played an important role in bounding the sum in \eqref{Jfrac}.  One can prove Proposition \ref{P:1Ddisc} by analyzing the corresponding exponential sums and using the Erd\H{o}s-Turan inequality, a quantitative version of Weyl's Criterion.  We will also use this inequality to prove Theorem \ref{T:maindisc}.     
\begin{proposition}[Erd\H{o}s-Turan inequality] \label{P:erdosturan}  Let $\omega = \omega_1, \omega_2,...$ be a sequence of real numbers.  Then for arbitrary natural numbers $m$ and $n$, 
\begin{equation}
D_{i,n}(\omega) \le C \bigg(\frac{1}{m} + \frac{1}{n} \sum_{h=1}^m \frac{1}{h} \Big|\sum_{j=i+1}^{i+n} e(h \omega_j) \Big| \bigg)
\end{equation} 
for some absolute constant $C$. 
\end{proposition}

Set $\displaystyle{t = \frac{ \prod_{w=c+1}^{c+A} j_w}{g_s(\mathbf{e}_1,....,\mathbf{e}_s)} }$.  To apply the Erd\H{o}s-Turan inequality, we wish to obtain bounds on the (normalized) exponential sum 
\begin{equation} \label{weylsum}
\frac{1}{\big|\mathcal{P}^{\mathscr{P}_{\mathbf{e}_1,...,\mathbf{e}_s}}_{(k_i), (l_v), (j_w) }\big|} \sum_{\substack{m_{i_u} \in I_{k_{i_u}, b} \cap \mathscr{P}_{\mathbf{e}_u} \cr j_v \in T_{l_v} \cap \mathscr{P}_{\mathbf{e}_v}}} e \bigg(ht \prod_{u=1}^c m_{i_u} \prod_{v=c+A+1}^s j_{v} \alpha \bigg)
\end{equation} 
for each $h \in \mathbb{N}$.
\begin{remark}
Throughout this section, we will use $\varepsilon'$ to denote an arbitrary positive number in an asymptotic or bound.  For example, writing $f(n) = o(n^{\varepsilon'})$ means that $f(n) = o(n^{\varepsilon'})$ for all $\varepsilon' > 0$.  
\end{remark}
 
For each $1 \le j \le s$, write $\displaystyle{\mathbf{e}_j = (e_{1j},...,e_{lj}) }$.  Let $\displaystyle{q_j = \prod_{i=1}^l p_i^{e_{i j} } }$.  Then from the definition of $Q_n$ in Definition \ref{Q10}, \begin{equation} \label{qj} q_j \ll (4s (\log n)^{2s} (\log \log n))^{2s (\log \log n)^3} = o(n^{\varepsilon'}) \end{equation} 

First, we establish a lower bound on the cardinality of the multiset $\displaystyle{\big|\mathcal{P}^{\mathscr{P}_{\mathbf{e}_1,...,\mathbf{e}_s}}_{(k_i), (l_v), (j_w) }\big| }$.

\begin{lemma} \label{L:cardinalityP}
\begin{equation}
\big|\mathcal{P}^{\mathscr{P}_{\mathbf{e}_1,...,\mathbf{e}_s}}_{(k_i), (l_v), (j_w) }\big| \gg \frac{1}{n^{\varepsilon'}} \prod_{v=c+A+1}^s  |T_{l_v}| \prod_{u=1}^c |I_{k_{i_u}, b}|   
\end{equation}  
\end{lemma}

\begin{proof}
By definition,
\begin{equation}
\big|\mathcal{P}^{\mathscr{P}_{\mathbf{e}_1,...,\mathbf{e}_s}}_{(k_i), (l_v), (j_w) }\big| = \prod_{v=c+A+1}^s |T_{l_v} \cap \mathscr{P}_{\mathbf{e}_v}| \prod_{u=1}^c |I_{k_{i_u}, b} \cap \mathscr{P}_{\mathbf{e}_u}| 
\end{equation}
Note that $|T_{l_v} \cap \mathscr{P}_{\mathbf{e}_v}|$ is the number of integers in the interval $\frac{F_{l_v}}{q_v} \le x_v \le \frac{F_{l_v+1}}{q_v}$ that are not divisible by the primorial $\prod\limits_{m =1}^{\lfloor (\log n)^{2s} \rfloor} p_m$.  Counting the number of such integers corresponds to ``sieving out'' by primes $p_m$ with $ m < (\log n)^{2s}$.  By the ``fundamental lemma of sieve theory'' (c.f. \cite[p. 60]{tenenbaum}), if $w = y^{1/u}$ then the fraction of integers in $[x, x+y]$ that are not divisible by any prime less than $w$ is $\sim \prod\limits_{p < w} \Big(1 - \frac{1}{p}\Big)(1 + O(u^{-u}))$.  As $u \to \infty$, we get the expected asymptotic which by Mertens's 3rd theorem (c.f. \cite[p. 17]{tenenbaum}) is $\displaystyle{ \frac{e^{-\gamma} + o(1)}{\log w}}$ where $w = (\log n)^{2s}$.  For our purposes, any finite $u$ suffices and we see that 
 
\begin{equation}\label{Tlvinequality} \frac{|T_{l_v} \cap \mathscr{P}_{\mathbf{e}_v}|}{|T_{l_v}|/q_v } \gg \frac{1}{\log \log n}
\end{equation}

By the same reasoning, \begin{equation}\label{Ikuinequality} \frac{|I_{k_{i_u}, b} \cap \mathscr{P}_{\mathbf{e}_u}| }{|I_{k_{i_u}, b}|/q_u } \gg \frac{1}{\log \log n}
\end{equation}
The bound then follows from \eqref{qj}.

\end{proof}

Set $G_m = (m-1)n/\log n$.  Then $I_{m,b} = (G_m, G_{m+1}]$.  Let $\displaystyle{\sum\nolimits_j^u}$ formally denote summation over the range
\begin{equation}
\sum_{m_1 <...< m_j < (\log n)^{2s}} \sum_{\substack{\frac{G_{k_{i_u}}}{q_u} \le x_u \le \frac{G_{k_{i_u}+1}}{q_u}  \cr p_{m_1}...p_{m_j} \mid x_u}}
\end{equation}
and let $\displaystyle{\sum\nolimits_j^v}$ formally denote summation over the range
\begin{equation}
\sum_{m_1 <...< m_j < (\log n)^{2s}} \sum_{\substack{\frac{F_{l_v}}{q_v} \le x_v \le \frac{F_{l_v+1}}{q_v} \cr p_{m_1}...p_{m_j} \mid x_v}}
\end{equation}
By the inclusion-exclusion principle, \eqref{weylsum} can be rewritten as an alternating sum over rectangular lattices:  
\begin{equation} \label{weylinex}
\begin{multlined} 
\frac{1}{\big|\mathcal{P}^{\mathscr{P}_{\mathbf{e}_1,...,\mathbf{e}_s}}_{(k_i), (l_v), (j_w) }\big|}  \prod_{u=1}^{c} \bigg( \sum_{j=0}^{(\log n)^{2s}} (-1)^j \sum\nolimits_j^u \bigg) \prod_{v=c+A+1}^{s} \bigg( \sum_{j=0}^{(\log n)^{2s}} (-1)^j \sum\nolimits_j^v \bigg) \\ e \bigg( ht \prod_{u=1}^c x_{u} q_{u} \prod_{v=c+A+1}^s x_{v} q_{v} \alpha \bigg)
\end{multlined}
\end{equation}
where here the symbols $\prod\limits_{u=1}^c$ and $\prod\limits_{v=c+A+1}^s j_{v}$ are to be interpreted as a formal product of the summation operators.  By Bonferroni's inequalities, we can truncate the alternating sums over $j$ at some even index $q$ to obtain upper bounds on the absolute value of \eqref{weylinex}.  To be precise, for each even index $q$ there exist multisets $A_u^q$ for $1 \le u \le c$ such that we have the equality of summation ranges
\begin{equation}
\sum_{j=0}^{q} (-1)^j \sum\nolimits_j^u = \sum_{j=0}^{q-1} (-1)^j \sum\nolimits_j^u + \sum_{x_u \in A_u^q}
\end{equation}
where 
\begin{equation}
\sum_{x_u \in A_u^q} 1 \le \sum\nolimits_q^u 1 \le \sum_{1 \le m_1 < ... < m_q < (\log n)^{2s}} \bigg(\frac{|I_{k_{i_u}, b}|/q_u}{p_{m_1}...p_{m_q}} + O(1) \bigg)
\end{equation}
Similarly, there exist multisets $A_v^q$ for $c+A+1 \le v \le s$ with the analogous properties.  For all $q$ and $r$, we have the inequality 

\begin{equation}
\sum_{1 \le m_1 < ... < m_q \le r} \frac{1}{p_{m_1}...p_{m_q}} \le \Big(\frac{e}{q} \Big)^q \bigg(\sum_{i=1}^r \frac{1}{p_{m_i}} \bigg)^q 
\end{equation}
For $q = 2 \lfloor \log \log n \rfloor$ and $r = (\log n)^{2s}$, this gives a bound of $\ll \Big( \frac{10 s \log \log \log n}{\log \log n} \Big)^{\log \log n}$.    Thus, truncating at the index $q = 2 \lfloor \log \log n \rfloor$, we can bound \eqref{weylinex} by    
\begin{equation} \label{weyltruncated}
\begin{multlined} 
\ll \frac{1}{\big|\mathcal{P}^{\mathscr{P}_{\mathbf{e}_1,...,\mathbf{e}_s}}_{(k_i), (l_v), (j_w) }\big|}  \prod_{u=1}^{c} \bigg( \sum_{j=0}^{2 \lfloor \log \log n \rfloor - 1} (-1)^j \sum\nolimits_j^u \bigg) \prod_{v=c+A+1}^{s} \bigg( \sum_{j=0}^{2 \lfloor \log \log n \rfloor - 1} (-1)^j \sum\nolimits_j^v \bigg) \\ e \bigg( ht \prod_{u=1}^c x_{u} q_{u} \prod_{v=c+A+1}^s x_{v} q_{v} \alpha \bigg) + E 
\end{multlined}
\end{equation}
where 
\begin{equation}
E = \Big( \frac{10 s \log \log \log n}{\log \log n} \Big)^{\log \log n} \bigg(\sum_v \frac{|T_{l_v}|}{q_v |T_{l_v} \cap \mathscr{P}_{\mathbf{e}_v}|}+ \sum_u \frac{|I_{k_{i_u}, b}|}{q_u |I_{k_{i_u}, b} \cap \mathscr{P}_{\mathbf{e}_u}|}\bigg)
\end{equation}
By \eqref{Tlvinequality} and \eqref{Ikuinequality}, 
\begin{equation} E \ll \log \log n \Big( \frac{10 s \log \log \log n}{\log \log n} \Big)^{\log \log n} \ll \Big( \frac{20 s \log \log \log n}{\log \log n} \Big)^{\log \log n}
\end{equation}  

Expanding out the alternating inclusion-exclusion sum, we see that the exponential sum in \eqref{weyltruncated} can be written as the sum of 
\begin{equation}\label{numterms} 
\ll \bigg(\sum_{j=0}^{2 \lfloor \log \log n \rfloor - 1} \binom{\lfloor (\log n)^{2s} \rfloor}{ j} \bigg)^s \ll (2 \log \log n)^{s} (\log n)^{4s^2 (\log \log n)} = o(n^{\varepsilon'}) 
\end{equation}
terms where each term is an exponential sum of the form
\begin{equation} \label{genericweyl}
S_{\mathbf{D}} = \sum_{(x_v, x_u) \in \mathbf{D} } e \bigg(h z \alpha \prod_{u=1}^c x_{u} \prod_{v=c+A+1}^s x_{v} \bigg) 
\end{equation}
where the domain of summation is given by  
\begin{equation} \label{domainD} 
\mathbf{D} := \prod_{v=c+A+1}^s \bigg[\frac{F_{l_v}}{r_v}, \frac{F_{l_v+1}}{ r_v} \bigg] \prod_{u=1}^c \bigg[\frac{G_{k_{i_u}}}{r_u}, \frac{G_{k_{i_u}+1}}{r_u} \bigg] 
\end{equation} 
and $\displaystyle{ z = t \prod_{u=1}^c q_u \prod_{v=c+A+1}^s q_v}$ is an integer.  
Here, by \eqref{qj} we see that $r_u = o(n^{\varepsilon'})$ and $r_v = o(n^{\varepsilon'})$ for $1 \le u \le c$ and $c+A+1 \le v \le s$.  Also, $z = o(n^{\varepsilon A + \varepsilon'})$ since $t = O(n^{\varepsilon A})$.  The domain of summation
has cardinality satisfying 
\begin{equation} \label{cardinalityD} 
\frac{1}{n^{\varepsilon'}} \prod_{v=c+A+1}^s  |T_{l_v}| \prod_{u=1}^c |I_{k_{i_u}, b}| \ll |\mathbf{D}| \le \prod_{v=c+A+1}^s  |T_{l_v}| \prod_{u=1}^c |I_{k_{i_u}, b}| 
\end{equation} 

Note that compared to \eqref{weylsum}, the sum in \eqref{genericweyl} is much simpler since the domain of summation is just a box of integer lattice points.  We will use a multivariate version of Weyl's inequality stated in Lemma \ref{L:WeylBound} to bound this sum.  In \cite[Lemma 2.1]{parsell2012}, Parsell gives a bound for the exponential sum when the range of summation is a box with equal side lengths.  The proof goes through unchanged to give the inequality in Lemma \ref{L:ParsellLemma} below for boxes with varying side lengths.  The basic idea is to apply coordinate-wise Weyl differencing (Cauchy's inequality) to the polynomial until we are left with a linear polynomial which can be summed directly as a geometric series.  Note that the advantage of Weyl's method is its simplicity.  It only gives a savings in the exponent that is exponentially small (but which is sufficient for our purposes).  Much stronger bounds on multidimensional exponential sums with polynomial savings in the exponent can be attained by following Vinogradov's mean value method (cf. \cite{parsell2005}).  Most recently, adapting Wooley's method of efficient congruencing in the 1-D case (see \cite{wooley2012}), near-optimal bounds have been obtained (see \cite{parsell2013}).  

Let $\mathbf{N} := \prod\limits_{i=1}^d [1, N_i]$ be a $d$-dimensional box of integer lattice points and let $j_1,...,j_d$ be non-negative integers such that $j_1+...+j_d = j$.  Let \begin{equation} \label{multivariableP}
P(n_1,...,n_d) = \alpha_{j_1,...,j_d} n_1^{j_1}...n_d^{j_d} + \sum_{\substack{l_1,...,l_d \ge 0 \\ l_1+...+l_d \le j-1}} \alpha_{l_1,...,l_d} n_1^{l_1}...n_d^{l_d}
\end{equation}
be a multivariable polynomial of degree $j$.  (The assumption of a unique highest degree term is simply for convenience.) 

\begin{lemma}[Parsell] \label{L:ParsellLemma}
Let $1 \le m \le j-1$ and $i_1,...,i_d$ be non-negative integers such that $i_1+...+i_d = m$.  Then 
\begin{equation*}
\begin{multlined}
\Big|\sum_{\mathbf{n} \in \mathbf{N}} e(P(\mathbf{n}))\Big|^{2^m} \ll \Big(\prod_{l=1}^d N_l \Big)^{2^m-1} \Big(\prod_{l=1}^d N_l^{i_l} \Big)^{-1} \sum_{\substack{h_i \in (-N_i, N_i) \cr 1 \le i \le m}} \Big|\sum_{\mathbf{n} \in B_{\mathbf{h}}(\mathbf{N})} e \Big(h_1...h_m P_{i_1,...,i_d}(\mathbf{n}, \mathbf{h}) \Big) \Big|
\end{multlined}
\end{equation*}  
where $B_{\mathbf{h}}(\mathbf{N})$ is a box contained in $\mathbf{N}$ for each $\mathbf{h}=(h_1,...,h_m)$ and $P_{i_1,...,i_d}(\mathbf{n}, \mathbf{h})$ is a polynomial with leading term 
\begin{equation*}
\alpha_{j_1,...,j_d} \frac{j_1! \cdot \cdot \cdot j_d!}{(j_1-i_1)! \cdot \cdot \cdot (j_d-i_d)!}n_1^{j_1-i_1} \cdot \cdot \cdot n_d^{j_d-i_d}
\end{equation*}
\end{lemma}
To apply Lemma \ref{L:ParsellLemma}, we will choose $i_1,...,i_d$ so that $P_{i_1,...,i_d}(\mathbf{n}, \mathbf{h})$ is a linear polynomial in $\mathbf{n}$.  Then we can evaluate the inner sum over $\mathbf{n} \in B_{\mathbf{h}}(\mathbf{N})$ via the geometric series estimate
\begin{equation} \label{geometricseriesestimate}
\sum_{n=1}^N e(\lambda n) \ll \min(N, ||\lambda||^{-1})
\end{equation}
for any non-integer $\lambda$ where $||\cdot||$ denotes distance to the nearest integer.  Since our goal is to obtain bounds that are uniform over polynomials with leading coefficient $hz \alpha$, we will need the following result to address the outer sum over $\mathbf{h}$. 

\begin{lemma} \label{L:FractionalPart}
Let $p, q$ be integers such that $|\alpha - p/q| \le q^{-2}$ where $(p,q) = 1$ and $q > 0$.  For any positive integers $M$, $N$, and $m$, we have \begin{equation}
\sum_{i=1}^M \min \{N, ||i m \alpha||^{-1} \} \ll \left(\frac{M}{q} + 1 \right)\left(mN + q\log q \right) 
\end{equation}

\end{lemma}

\begin{proof}
We divide the sum over $M$ terms into $O(M/q + 1)$ blocks of size $q$ (plus one possibly partial block) where each block has the form 
\begin{equation}\label{qblock}
\sum_{i=0}^{q-1} \min(N, ||(i_0 + i)m \alpha||^{-1})
\end{equation}
and $i_0$ is the first number in the block.  Since $|\alpha - p/q| \le q^{-2}$, we have $(i_0 + i)m \alpha = i_0 m \alpha + i m p /q + u$ where $|u| \le m/q$.    
As $i$ runs over the set $[q]$, $i m p$ runs over the residues $r\gcd(m,q) \mod q$ for $1 \le r \le q/ \gcd(m,q)$ with multiplicity $\gcd(m,q)$.  Let $b$ be the nearest integer to $q i_0 m \alpha$.  Then $||(i_0+i)m \alpha|| = ||(b + r \gcd(m,q))/q + u||$ where $|u| \le 3m/q$.  Note that there are $O(m)$ integers $i$ such that $||(i_0 + i)m \alpha|| \le ||m/q||$.  Therefore, \eqref{qblock} is bounded by
\begin{equation}
\ll mN + \sum_{r=1}^{q} \frac{q}{r} \ll mN + q \log q
\end{equation}
which proves the lemma.
 
\end{proof}

\begin{lemma} \label{L:WeylBound}
Let $p, q$ be integers such that $|\alpha - p/q| \le q^{-2}$ where $(p,q) = 1$ and $q > 0$.  Let $P(\mathbf{n}) = m \alpha (n_1 - a_1)...(n_d - a_d)$ for some integers $m$ and $a_1,...,a_d$ where $m \ge 1$ and $d \ge 2$.  Assume that $N_1 \ge ... \ge N_d$.  Then 
\begin{equation*}
\begin{multlined}
\Big|\sum_{\mathbf{n} \in \mathbf{N}} e(P(\mathbf{n}))\Big| \ll \Big(\prod_{l=1}^d N_l \Big)^{1+\varepsilon'} \bigg(\frac{m}{q} + \frac{1}{N_d} + \frac{m}{\prod_{l=1}^{d-1} N_l} + \frac{q}{\prod_{l=1}^d N_l} \bigg)^{2^{1-d}}
\end{multlined}
\end{equation*} 
where the absolute constant does not depend on $a_1,...,a_d$.  
\end{lemma}

\begin{proof}
Let $i_l = 1$ for $1 \le l \le d-1$ and $i_d = 0$.  Then by Lemma \ref{L:ParsellLemma},
\begin{equation*}
\begin{multlined}
\Big|\sum_{\mathbf{n} \in \mathbf{N}} e(P(\mathbf{n}))\Big|^{2^{d-1}} \ll \Big(\prod_{l=1}^d N_l \Big)^{2^{d-1}-1} \Big(\prod_{l=1}^{d-1} N_l \Big)^{-1} \sum_{\substack{h_i \in (-N_i, N_i) \cr 1 \le i \le d-1}} \Big| \sum_{\mathbf{n} \in B_{\mathbf{h}}(\mathbf{N})} e(h_1...h_{d-1} m \alpha n_d) \Big|
\end{multlined}
\end{equation*}  
Since the divisor function $\tau(n)$ satisfies $\tau(n) \ll n^{\varepsilon'}$, the number of ways to write an integer $1 \le n \le \prod\limits_{i=1}^{d-1} N_i$ as a product $h_1...h_{d-1}$ is $\ll N_1^{\varepsilon'}$.    Using this fact and \eqref{geometricseriesestimate} and separating out the terms where some $h_i = 0$, we have 
\begin{equation*}
\begin{multlined}
\Big|\sum_{\mathbf{n} \in \mathbf{N}} e(P(\mathbf{n}))\Big|^{2^{d-1}} \ll \Big(\prod_{l=1}^d N_l \Big)^{2^{d-1}} (N_{d-1})^{-1} + \Big(\prod_{l=1}^d N_l \Big)^{2^{d-1}-1+\varepsilon'} \sum_{\substack{1 \le i \le \prod\limits_{l=1}^{d-1} N_l}} \min(N_d, ||i m \alpha||^{-1})
\end{multlined}
\end{equation*}  
Finally, by Lemma \ref{L:FractionalPart} this is 
\begin{equation}
\ll  \Big(\prod_{l=1}^d N_l \Big)^{2^{d-1}+\varepsilon'} \bigg(\frac{1}{N_{d-1}} + \frac{m}{q} + \frac{1}{N_d} + \frac{m}{\prod_{l=1}^{d-1} N_l} + \frac{q}{\prod_{l=1}^d N_l} \bigg)
\end{equation}
Raising to the power $2^{1-d}$ completes the proof.
\end{proof}

We are now ready to obtain an estimate for $S_{\mathbf{D}}$ from \eqref{genericweyl} and apply the Erd\H{o}s-Turan inequality to prove the discrepancy bound.  
\begin{proof}[Proof of Theorem \ref{T:maindisc}]

Let $d = s-A = B+C+c$ be the degree of the polynomial in $S_{\mathbf{D}}$ with leading coefficient $h z \alpha$.  First assume $d > 1$. 
The successive continued fraction convergents of $\alpha$ satisfy 
\begin{equation}\frac{1}{2q_i q_{i+1}} \le \Big|\alpha - \frac{p_i}{q_i} \Big| \le \frac{1}{q_i q_{i+1}}
\end{equation}  By the definition of irrationality measure, $\displaystyle{\Big|\alpha - \frac{p}{q} \Big| > \frac{1}{q^{\mu+\epsilon'}}}$ for all integers $p$ and $q$ sufficiently large.  Therefore $q_{i+1} \le q_i^{\mu-1+\epsilon'}$ for $i$ sufficiently large.   

Assume that $1 \le h \le n^{\varepsilon}$.  Then $h z = o(n^{\varepsilon (A+2)})$.  Recall that $C + c \ge 1$, and  therefore $\prod\limits_{v=c+A+1}^s  |T_{l_v}| \prod\limits_{u=1}^c |I_{k_{i_u}, b}| \ge n^\delta$.  Choose $q$ to be a convergent such that $n^{\varepsilon (A+3)} \le q \le n^{\mu \varepsilon (A+3)}$.  By Lemma \ref{L:WeylBound},
\begin{equation}
S_{\mathbf{D}} \ll \prod\limits_{v=c+A+1}^s  |T_{l_v}| \prod\limits_{u=1}^c |I_{k_{i_u}, b}| n^{\varepsilon'} \bigg(\frac{n^{\varepsilon (A+2)}}{n^{\varepsilon(A+3)}} + \frac{n^{\varepsilon'}}{n^{\varepsilon} } + \frac{n^{\varepsilon (A+2) + \varepsilon'}}{n^\delta} + \frac{n^{\mu \varepsilon (A+3) + \varepsilon'}}{n^\delta} \bigg)^{2^{1-d}}
\end{equation}

Then using \eqref{cardinalityD}, we can bound \eqref{weyltruncated} by
\begin{equation}
n^{\varepsilon'} \bigg(\frac{n^{\varepsilon (A+2)}}{n^{\varepsilon(A+3)}} + \frac{n^{\varepsilon'}}{n^{\varepsilon} } + \frac{n^{\varepsilon (A+2) + \varepsilon'}}{n^\delta} + \frac{n^{\mu \varepsilon (A+3) + \varepsilon'}}{n^\delta} \bigg)^{2^{1-d}} + \Big( \frac{20 s \log \log \log n}{\log \log n} \Big)^{\log \log n}
\end{equation}
 
Choosing $m = n^{\varepsilon}$ in the Erd\H{o}s-Turan inequality (Proposition \ref{P:erdosturan}), we obtain the bound  
\begin{equation} 
D\big(\mathcal{P}^{\mathscr{P}_{\mathbf{e}_1,...,\mathbf{e}_s}}_{(k_i), (l_v), (j_w) } \big) \ll \frac{1}{n^{\varepsilon}} + \log n \Big( \frac{20 s \log \log \log n}{\log \log n} \Big)^{\log \log n}
\end{equation}
for $d > 1$.  
If the degree $d = 1$, either $c=1$ or $C=1$ since $C+c \ge 1$.  By \eqref{geometricseriesestimate}, we have $S_{\mathbf{D}} \ll ||hz \alpha||^{-1}$ where $z = o(n^{\varepsilon A + \varepsilon'})$.  Then 
\begin{equation} \label{finallinear}
\frac{n^{\varepsilon'} |S_{\mathbf{D}}|}{\big|\mathcal{P}^{\mathscr{P}_{\mathbf{e}_1,...,\mathbf{e}_s}}_{(k_i), (l_v), (j_w) }\big|} \ll \frac{n^{\varepsilon'}}{n^\delta} \frac{1}{||hz \alpha||}
\end{equation}

By \cite[Lemma 2.3.3]{kuipers}, for any integer $m$ 
\begin{equation} \label{expbound1} 
\sum_{h=1}^{m} \frac{1}{h z ||h z \alpha||} \le \sum_{h=1}^{mz} \frac{1}{h ||h \alpha||} \ll (mz)^{\mu - 2 + \varepsilon'}  
\end{equation} 

Then choosing $m = n^{\varepsilon}$ in the Erd\H{o}s-Turan inequality as before completes the proof. 

\end{proof}

\section{Eigenvalue process at rational angles} \label{S:generalkrational}

To obtain the limiting eigenvalue point process $\displaystyle{\lim_{n \to \infty} \mathfrak{E}_{n,k}^\alpha}$ at a rational angle $\alpha = s/t$, we need the following distributional convergence analogous to Lemma \ref{L:unigeneral} for irrational $\alpha$.  

\begin{lemma} \label{unigeneralrational}
Let $\alpha = s/t$ be a rational number in reduced form where $t = \prod\limits_{m=1}^M p_{m}^{e_m(t)}$.  Fix positive integers $r$ and $l$ such that $r \ge k$ and $l \ge M$.  For $\displaystyle{1 \le i_1 < ... < i_k \le r}$, define $\mathbb{Z}/t\mathbb{Z}$-valued exchangeable random variables 
\begin{equation}
V_{i_1,...,i_k}^l =  \frac{\prod_{m=1}^{M} p_m^{\mathbf{X}_{m i_1}+...+ \mathbf{X}_{m i_k}}}{g_k(\mathbf{X}_l^{i_1,...,i_k})} U_{i_1}...U_{i_k}
\end{equation} 
where $U_1,U_2,...$ are i.i.d. random variables distributed uniformly on the group of units $(\mathbb{Z}/t\mathbb{Z})^\times$.  Let $U_{i_1,...,i_k}^l = V_{i_1,...,i_k}^l/t \in \mathbb{Q}$ where we identify 
$\mathbb{Z}/t\mathbb{Z}$ with $[t] \subset \mathbb{Z}$.

Then we have the distributional convergence 
\begin{multline*} \Bigg(\textbf{E}_l(L_{1}^{(n)},...,L_{r}^{(n)}), L_1^{(n)}/n,...,L_r^{(n)}/n, \bigg(\psi_\alpha \Big(\frac{L_{i_1}^{(n)}...L_{i_k}^{(n)}}{ g_k(\textbf{E}_l(L_{i_1}^{(n)}, ...,L_{i_k}^{(n)}))} \Big) \bigg) \Bigg) \\ \buildrel d \over \to \left(\textbf{X}_l^{1,...,r}, L_1,...,L_r, (U_{i_1,...,i_k}^l) \right) 
\end{multline*}
\end{lemma}

\begin{proof}

Following the proof of Lemma \ref{L:unigeneral}, fix some matrix $\mathbf{A} \in \mathbb{N}^{l \times r}$ and corresponding factorization class $\mathscr{P}_{\mathbf{A}}$.  The analog of computing the discrepancy $\displaystyle{ D(\mathcal{P}^{\mathscr{P}_{\mathbf{A}}, \boldsymbol \alpha}_{(k_i), b} ) }$ in Lemma \ref{L:welltorus} is finding the distribution of $\displaystyle{\bigg(\psi_\alpha \Big(\frac{m_{i_1}...m_{i_k}}{ g_k(\textbf{E}_l(m_{i_1}, ...,m_{i_k}))} \Big) \bigg) }$ such that $(m_1,...,m_r)$ are chosen uniformly from the set 
\begin{equation}
\{(m_1,...,m_r): (\mathbf{E}_l(m_1,...,m_r) = \mathbf{A}) \land \bigwedge_{i=1}^r (m_i \in I_{k_i, b} ) \} 
\end{equation}  
It is sufficient to find the distribution of $\displaystyle{\bigg(\frac{m_{i_1}...m_{i_k}}{ g_k(\textbf{E}_l(m_{i_1}, ...,m_{i_k}))} \bigg)}$ mod $t$ and then apply the function $\psi_\alpha$ .  

For $1 \le i \le r$, let $N_i = \prod\limits_{m=1}^l p_m^{\mathbf{A}_{mi}+1 + e_m(t)}$ and let $s_i$ be an element of $\mathscr{P}_{\mathbf{a}_i}\cap [N_i]$.   Define the sequences $s^i(q) = s_i + q N_i$ where $q$ runs over the non-negative integers.  Note that $s^i(q) \in \mathscr{P}_{\mathbf{a}_i}$ for all $q$.  Thus, we can partition each factorization class $\mathscr{P}_{\mathbf{a}_i}$ into congruence classes modulo $N_i$.  Also, for $\displaystyle{ 1 \le i_1 < ... < i_k \le r}$, all elements in $\displaystyle{ \bigg\{ \frac{s^{i_1}(q_{i_1})...s^{i_k}(q_{i_k})}{g_k(\textbf{a}_{i_1},...,\textbf{a}_{i_k}) }: q_{i_j} \in \mathbb{N} \bigg\} }$ lie in the same congruence class mod $t$.  Thus, it is sufficient to restrict our attention to $(m_1,...,m_r) \in [N_1] \times... \times [N_r]$ and compute the distribution of the multiset 
\begin{equation} V^{\mathbf{A}} := \bigg\{ \bigg(\frac{m_{i_1}...m_{i_k}}{g_k(\textbf{a}_{i_1},...,\textbf{a}_{i_k}) } \bigg) \in (\mathbb{Z}/t\mathbb{Z})^{ \binom{r}{ k}}: m_{i_j} \in \mathbb{Z}/N_{i_j}\mathbb{Z} \cap \mathscr{P}_{\mathbf{a}_{i_j}} \bigg\}
\end{equation}
Here, $V^{\mathbf{A}}$ is well-defined since there is a natural projection map $\mathbb{Z}/N_{i_j}\mathbb{Z} \to \mathbb{Z}/t\mathbb{Z}$.  By the Chinese remainder theorem, \[\mathbb{Z}/N_i\mathbb{Z} \cong \prod_{m=1}^l \mathbb{Z}/p_m^{\mathbf{A}_{mi}+1+e_m(t)} \mathbb{Z} \]  Then a little modular arithmetic shows that the $V^{\mathbf{A}}$ can be written as 
\begin{equation}
V^{\mathbf{A}}_{i_1,...,i_k} = \frac{\prod_{m=1}^M p_m^{\mathbf{A}_{m i_1}+...+ \mathbf{A}_{m i_k}}}{g_k(\textbf{a}_{i_1},...,\textbf{a}_{i_k}) }  U_{i_1}...U_{i_k}
\end{equation}
where $U_1,...,U_r$ are i.i.d. random variables distributed uniformly on $(\mathbb{Z}/t\mathbb{Z})^\times$.  Here, we use the fact that $p_l^{\mathbf{A}_{l i_j}}$ are invertible in $\mathbb{Z}/t\mathbb{Z}$ for $l>M$ and hence can be absorbed into the random variable $U_{i_j}$ by translation invariance.  The rest of the proof follows as for Lemma \ref{L:unigeneral}.
\end{proof}  

\begin{corollary} \label{firstrrational}

Recall the notation from Lemma \ref{unigeneralrational} and as usual, allow the multi-dimen\-sion\-al indices $(i_1,...,i_k)$ to run over the range $\displaystyle{1 \le i_1 < ...< i_k \le r}$.  Define $\mathbb{Z}/t\mathbb{Z}$-valued exchangeable random variables  \[V_{i_1,...,i_k} =  \frac{\prod_{m=1}^M p_m^{\mathbf{X}_{m i_1}+...+ \mathbf{X}_{m i_k}}}{g_k(\mathbf{X}_\infty^{i_1,...,i_k})} U_{i_1}...U_{i_k} \] 
where $U_1,U_2,...$ are i.i.d. random variables distributed uniformly on $(\mathbb{Z}/t\mathbb{Z})^\times$.  Let $U_{i_1,...,i_k} = V_{i_1,...,i_k}/t \in \mathbb{Q}$ where we identify $\mathbb{Z}/t\mathbb{Z}$ with $[t] \subset \mathbb{Z}$.  Then we have the distributional convergence 
\begin{multline*}
\bigg( \bigg(\frac{L_{i_1}^{(n)}...L_{i_k}^{(n)}}{\lcm(L_{i_1}^{(n)},...,L_{i_k}^{(n)})} \bigg), (L_{i_1}^{(n)}...L_{i_k}^{(n)}/n^k),(\psi_\alpha(\lcm(L_{i_1}^{(n)},...,L_{i_k}^{(n)}) ) ) \bigg) \\
\buildrel d \over \to \left( (g_{k}(\textbf{X}_\infty^{i_1,...,i_k})),  (L_{i_1}...L_{i_k}), (U_{i_1,...,i_k})\right)
\end{multline*}

\end{corollary}

\begin{proof}
This follows from the continuous mapping theorem and Lemma \ref{L:firstl}.
\end{proof}
  
\begin{theorem}\label{T:rationalgeneral}

Let $\alpha = s/t$ be a rational number in reduced form where $t = \prod\limits_{m=1}^M p_{m}^{e_m(t)}$.  For $\displaystyle{i_1 < ... < i_k}$, let $U_{i_1,...,i_k}$ be as defined in Corollary \ref{firstrrational}.  Define the process \[\mathfrak{E}_{\infty, k}^t = k!
\sum_{i_1 <...< i_k} g_{k}(\textbf{X}_\infty^{i_1,...,i_k}) \sum_{q \in \mathbb{Z}} \delta_{(q + U_{i_1,...,i_k}) g_{k}(\textbf{X}_\infty^{i_1,...,i_k})  /(L_{i_1}...L_{i_k}) } \]  Then $\mathfrak{E}_{n, k}^\alpha \to \mathfrak{E}_{\infty,k}^t$ weakly in the sense that we have the weak convergence of $\overline{\mathbb{R}}$-valued random variables $\displaystyle{\int f d \mathfrak{E}_{n,k}^\alpha \buildrel d \over \to \int f d \mathfrak{E}_{\infty,k}^t}$ for all continuous, compactly supported functions $f$.
\end{theorem}

\begin{proof}

Let $f$ have support contained in the interval $(-T, T)$ and define $J_k(n)$ as in \eqref{J_k}.  Just as in equation \eqref{J set}, we have the set equality \begin{equation} J_k \cap [n^k/(t T)] = t\mathbb{Z} \cap [n^k/(tT)]. \end{equation}  If $j \in t\mathbb{Z} \cap [n^k/(tT)]$, the interval $(\alpha - T/n^k, \alpha+ T/n^k)$ contains exactly one eigenangle corresponding to the $\sigma_k$-cycle of length $j$ (namely $\alpha$).

Write $\mathfrak{E}_{n,k}^\alpha = \xi_{n,k}^{tT} + \eta_{n,k}^{tT}$ where $\xi_{n,k}^{tT}$ and $\eta_{n,k}^{tT}$ are as defined in \eqref{e1}.  Then 
\begin{equation}
\int f d \xi_{n,k}^{tT} = k! \sum_{1 \le i_1 <...< i_k \le tT} \frac{L_{i_1}^{(n)}...L_{i_k}^{(n)}}{\lcm(L_{i_1}^{(n)},...,L_{i_k}^{(n)})} \sum_{q \in \mathbb{Z}} f(n^k (q/\lcm(L_{i_1}^{(n)},...,L_{i_k}^{(n)}) - \alpha ))
\end{equation} 
If $r > tT$, then $L_r^{(n)} < n/(tT)$ and therefore $\lcm(L_r^{(n)}, L_{i_1}^{(n)},...,L_{i_{k-1}}^{(n)}) < n^k/(tT)$.  Then if $f(0) \ge 0$,  
\begin{equation}
f(0) \bigg(\sum_{\substack{1 \le j \le n^k/(tT) \\ t|j}} C_{j,k}^{(n)} - \binom{\lceil tT \rceil }{k} \bigg) \le \int f d \eta_{n,k}^{tT} \le f(0) \sum_{\substack{1 \le j \le n^k/(tT) \\ t|j}} C_{j,k}^{(n)} 
\end{equation}
Let $\mathfrak{E}_{\infty,k}^t = \xi_{\infty, k}^{tT} + \eta_{\infty, k}^{tT}$ where $\xi_{\infty, k}^{tT}$ and $\eta_{\infty, k}^{tT}$ are defined as in \eqref{e2} with $\mathfrak{E}_{\infty,k}^t$ in place of $\mathfrak{E}_{\infty,k}^*$ and with the new definition of the variables $U_{i_1,...,i_k}$.  Then 
\begin{equation}
\int f d \mathfrak{\xi}_{\infty,k}^{tT} = k! \sum_{1 \le i_1 <...< i_k \le tT} g_{k}(\textbf{X}_\infty^{i_1,...,i_k}) \sum_{q \in \mathbb{Z}} f( (U_{i_1,...,i_k} + q)g_{k}(\textbf{X}_\infty^{i_1,...,i_k}) /(L_{i_1}...L_{i_k}) )
\end{equation}
and since $L_r < 1/(tT)$ for $r > tT$, we have 
\begin{equation}
\int f d \mathfrak{\eta}_{\infty,k}^{tT} = k! f(0) \sum_{\substack{1 \le i_1 <...< i_k \\ i_k > tT}} g_{k}(\textbf{X}_\infty^{i_1,...,i_k}) \mathbbm{1}(U_{i_1,...,i_k} = 1) 
\end{equation}
If $f(0) > 0$, then $\int f d \mathfrak{E}_{n,k}^\alpha$ converges weakly to the $\overline{\mathbb{R}}$-valued random variable $\int f d \mathfrak{E}_{\infty,k}^t$, which is infinite almost surely by the Borel-Cantelli lemma.     

Now we can assume $f(0) = 0$.  By Corollary \ref{firstrrational} and the continuous mapping theorem, \[ \int f d \xi_{n,k}^{tT} \buildrel d \over \to \int f d \xi_{\infty, k}^{tT} \] which proves the theorem.
\end{proof}

\begin{remark} 

At rational $\alpha = s/t$, the corresponding eigenvalue point processes for the $k$-subset and $S^{(n-k,1^k)}$ representations both converge weakly to $\frac{1}{k!} \mathfrak{E}_{\infty, k}^t$ by the reasoning in Remark \ref{otherreps}.  
\end{remark}

\section{Power series representation for gap probability}\label{S:power}

Corollary \ref{C:eigengapformula} gives a formula for the limiting eigenvalue gap probability $P_k^\theta(y_1,y_2)$ for irrational $\alpha$ with finite irrationality measure.  In this section, we give a procedure to compute the expectation explicitly when $y_2 - y_1 \le k^k$ and $\theta=1$.  To simplify the resulting combinatorics, it is convenient to borrow some terminology from graph theory.  The following definitions introduce basic notions about hypergraphs.  

\begin{definition}
A hypergraph is a generalization of a graph in which an edge can connect any number of vertices.  Formally, a hypergraph $G$ is a pair $G = (V, E)$ where $V$ is a set of vertices and $E$ is a multiset of non-empty subsets of $V$.  A subset of $V$ of size $j$ is called a $j$-edge.  Since $E$ is a multiset, multi-edges or multiple edges on the same set of vertices in hypergraphs are allowed.  A hypergraph isomorphism between $(V_1, E_1)$ and $(V_2, E_2)$ is a bijection between vertex sets $V_1$ and $V_2$ that respects the edge multisets $E_1$ and $E_2$.  Given a vertex subset $W = \{v_1,...,v_{|W|}\} \subseteq V$, we define the induced subhypergraph $\displaystyle{G^W = (W, \{ e \cap W: e \in E(G), e \cap W \neq \emptyset \}) }$.  Given two hypergraphs $G_1 = (V_1,E_1)$ and $G_2 = (V_2,E_2)$, the union is given by $G_1 \cup G_2 = (V_1 \cup V_2, E_1 \cup E_2)$.  (Note that $V_1 \cup V_2$ is a union of sets while $E_1 \cup E_2$ is a union of multisets).   

\end{definition}

\begin{definition}
A labeled hypergraph is a hypergraph whose vertices have been assigned distinct labels in $\mathbb{N}$.  We define (nonstandard terminology) a label isomorphism, $\phi: G_1 \to G_2$, between labeled hypergraphs $G_1$ and $G_2$ to be a hypergraph isomorphism that respects ordering of the vertices.  In other words, if we have two vertices $a,b \in V(G_1)$ such that $a < b$, then $\phi(a) < \phi(b)$.  This is denoted by $G_1 \cong^l G_2$.  
\end{definition}

\begin{definition}
The $j$-degree of a vertex of a hypergraph is the total number of $j$-edges incident to the vertex.  The unlabeled $j$-degree sequence is the non-decreasing sequence of its vertex $j$-degrees.  The labeled $j$-degree sequence is the list of vertex $j$-degrees by labelling.   
\end{definition}

\begin{definition}
Let $\mathscr{B}_j^{m}$ be the set of labelled hypergraphs whose edge set consists of $m$ distinct $j$-edges and which do not have isolated vertices, i.e. every vertex lies in some edge $e$. 
Let $\mathscr{A}_j^{m} \subset \mathscr{B}_j^m$ be the finite subset of labelled hypergraphs whose vertex label set is of the form $\{1, 2,...,r\}$ for some $r$.  
\end{definition}

We can think of the random variables $L_1, L_2,...$ as vertices of a labeled hypergraph $G$ and each product $\prod\limits_{u=1}^k L_{i_u}$ as a $k$-edge containing the vertices $L_{i_1},...,L_{i_k}$.  Then each term in the product expansion of \eqref{gapformula} corresponds to a unique labelled hypergraph $G \in \mathscr{B}_k^m$ for some $m$. 

For $G \in \mathscr{B}_k^{m}$ , let 
\begin{equation}\label{psygib} \psi(G) = \mathbb{E} \bigg[ \prod_{(i_1,...,i_k) \in E(G)} \frac{1}{g_k(\mathbf{X}_\infty^{i_1,...,i_k})} \bigg] \end{equation}

Note that if $\displaystyle{H \in \mathscr{B}_k^{m}}$ is isomorphic to $G$, then $\psi(H) = \psi(G)$ since $\psi(H)$ only depends on the isomorphism class of $H$ (doesn't care about the labelling of the vertices). 

For a hypergraph $\displaystyle{H = (V,E) \in \mathscr{B}_k^{m}}$ with labeled $k$-degree sequence $(d_1,...,d_{|V|})$, define 
\begin{equation} \label{LH} L^H := L_{1}^{d_1} L_{2}^{d_2}...L_{|V|}^{d_{|V|}}
\end{equation}   
With this notation, it is easy to see the following:
\begin{lemma} \label{L:powerseriesformula}
Let $y_2 - y_1 \le k^k$.  Write $P^\theta_k(y_1,y_2) = f(y_1 - y_2)$ where $\displaystyle{f(x) = \sum_{m=0}^\infty c_m x^m}$.  The coefficient for the $m^{\textrm{th}}$ term of the power series representation is given by \begin{equation} c_m = \sum_{H \in \mathscr{B}^m_k} \psi(H) \mathbb{E}[L^H] = \sum_{G \in \mathscr{A}_k^{m}} \psi(G) \sum_{H \cong^l G} \mathbb{E}[L^H] \end{equation}
\end{lemma}
Let us first compute $\displaystyle{\sum_{H \cong^l G} \mathbb{E}[L^H]}$ for each $G \in \mathscr{A}_k^m$.  We've essentially already seen how to calculate this sum in Section \ref{S:k=1irrational}.  Let $i_1,...,i_v$ be the vertices of $H$ with $k$-degrees $d_1,...,d_v$.  Set $d = d_1 +...+d_v$.  If $\theta = 1$, then from the moments formula \eqref{griffiths} we have 
\begin{multline}
\nonumber\mathbb{E}[L_{i_1}^{d_1} L_{i_2}^{d_2}...L_{i_v}^{d_v}]  = \frac{1}{d!} \int_0^\infty \int_{x_v}^\infty...\int_{x_2}^\infty \frac{E(x_1)^{i_1 - 1}}{(i_1 - 1)!} \frac{(E(x_2) - E(x_1))^{i_2 - i_1 - 1}}{(i_2 - i_1 - 1)!}... \\
...\frac{(E(x_v) - E(x_{v-1}))^{i_v - i_{v-1} - 1}}{(i_v - i_{v-1} - 1)!}  x_1^{d_1-1} e^{-x_1}...x_v^{d_v-1} e^{-x_v} e^{-E_1(x_v)} dx_1 dx_2... dx_v
\end{multline}

By \eqref{telescope},
\begin{equation}\label{LHformula} 
\begin{multlined}
\sum_{H \cong^l G} \mathbb{E}[L^H] = \sum_{i_1 <...<i_v } \mathbb{E}[L_{i_1}^{d_1} L_{i_2}^{d_2}...L_{i_v}^{d_v}] \\=\frac{1}{d!} \int_0^\infty \int_{x_v}^\infty...\int_{x_2}^\infty x_1^{d_1-1} e^{-x_1}...x_v^{d_v-1} e^{-x_v} dx_1...dx_v
\end{multlined}
\end{equation}
This is an elementary integral and evaluates to a rational number for any choice of $d_1,...,d_v$.  One can compute the integrals successively by making use of the identity 
\begin{equation}
\Gamma(s,x) = (s-1)! e^{-x} \sum_{j=0}^{s-1} \frac{x^j}{j!}
\end{equation}
for positive integers $s$ where $\Gamma(s,x) = \int_x^\infty t^{s-1}e^{-t}dt$ is the incomplete Gamma function.   

We now give a procedure to compute $\psi(G)$ for each $G \in \mathscr{A}_k^m$.  First, for an arbitrary hypergraph $G$ with vertex set $V =  \{v_{1},...,v_{|V|} \}$, define
\begin{equation} \label{SG} 
S_{G}^p = \sum_{e_{v_1},...,e_{v_{|V|}} \in \mathbb{N}} \prod_{i=1}^{|V|} \frac{1}{p^{e_{v_i}}} \prod_{(v_{i_1},...,v_{i_j}) \in E(G)} \frac{p^{\max(e_{v_{i_1}},...,e_{v_{i_j}})}}{ p^{e_{v_{i_1}}+...+ e_{v_{i_j}}}}   
\end{equation} 
Then for $G \in \mathscr{A}_k^m$, summing over $\mathbf{A} \in \mathbb{N}^{l \times |V|}$, we have 
\begin{align}
\psi(G) \nonumber=& \mathbb{E} \bigg[ \prod_{(i_1,...,i_k) \in E(G)} \frac{1}{g_k(\mathbf{X}_\infty^{i_1,...,i_k})} \bigg]\\
\nonumber=& \lim_{l \to \infty } \sum_{\mathbf{A}} \bigg( \prod_{m=1}^l \Big(1 - \frac{1}{p_m} \Big)^{|V|} \prod_{i=1}^{|V|} \frac{1}{p_m^{\mathbf{A}_{m i}}} \bigg) \prod_{(i_1,...,i_k) \in E(G)} \frac{1}{g_k(\mathbf{a}_{i_1},...,\mathbf{a}_{i_k})} \\
\nonumber=& \prod_{p} \Big(1 - \frac{1}{p}\Big)^{|V|} \sum_{e_1,...,e_{|V|} \in \mathbb{N}} \prod_{i=1}^{|V|} \frac{1}{p^{e_i}} \prod_{(i_1,...,i_k) \in E(G)} \frac{p^{\max(e_{i_1},...,e_{i_k})}}{ p^{e_{i_1} +...+e_{i_k}}} \\
=& \prod_{p} \Big(1 - \frac{1}{p}\Big)^{|V|} S_{G}^p
\end{align}

For any hypergraph $G$, we have the following recursive formula for computing $S_G^p$:

\begin{lemma} \label{L:recursive}
Let $G = (V, E)$ be a hypergraph and $U \subset V$ be a vertex subset.  Then
\large
\begin{equation}
S_{G}^p = \bigg(1 - \dfrac{1}{p^{|V| + \sum (j-1)|E_j^V|}}\bigg)^{-1} \sum_{U \subsetneq V} \dfrac{1}{p^{|U| + \sum (j-1) |E_j^U|}} S_{G^U}^p
\end{equation}
\normalsize
where $E_j^U$ is the $j$-edge set of the subhypergraph $G^U$ induced by the vertex subset $U \subset V$ (counting multiedges).  The base case is $S_{\varnothing} = 1$.  
\end{lemma}

\begin{proof}

Partition the domain of summation $(e_{v_1},...,e_{v_{|V|}}) \in \mathbb{N}^{|V|}$ of $S_{G}^p$ into blocks $B_c$ such that $\min(e_{v_1},...,e_{v_{|V|}}) = c$.  In each block $B_c$, we ``factor out'' $p^c$ from each vertex, meaning we take out a total factor of 
\begin{equation} 
\frac{1}{p^{|V|c}} \prod_{(v_{i_1},...,v_{i_j}) \in E(G)} \frac{p^c}{ p^{jc}} 
\end{equation} 
from 
\begin{equation}
\sum_{(e_{v_1},...,e_{v_{|V|}}) \in B_c} \prod_{i=1}^{|V|} \frac{1}{p^{e_{v_i}}} \prod_{(v_{i_1},...,v_{i_j}) \in E(G)} \frac{p^{\max(e_{v_{i_1}},...,e_{v_{i_j}})}}{ p^{e_{v_{i_1}}+...+ e_{v_{i_j}}}}
\end{equation} 
Summing over all $c \in \mathbb{N}$ yields the $\bigg(1 - \dfrac{1}{p^{|V| + \sum (j-1)|E_j^V|}}\bigg)^{-1}$ term.  After this factoring, at least one vertex will be left with an exponent $e_{v_{i}} = 0$.  Thus, we are reduced to summing over all subhypergraphs $G^U$ with $U = (u_1,...,u_{|U|}) \subsetneq V$ and exponents $(e_{u_1},...,e_{u_{|U|}})$ such that $e_{u_i} > 0$.  In each subhypergraph $G^U$, we have room to factor out at least one factor of $p$ from each vertex.  Putting this all together yields the formula.
\end{proof}
With Lemma \ref{L:recursive}, we now have a procedure to compute $\psi(G) = \prod\limits_p \Big(1 - \frac{1}{p} \Big)^{|V|} S_G^p$ for all $G \in \mathscr{A}_k^m$.

The following examples tabulate $\displaystyle{S_{G}^p}$ for a few small hypergraphs $G$.  For ease of notation, we leave off the superscript $p$ and set $q = 1/p$.  Note that if the hypergraph $G$ is not connected, then this recursive formula should be applied on each component since if $G = G_1 \cup G_2$ where $G_1$ and $G_2$ are disjoint, then $S_G = S_{G_1} S_{G_2}$ and $\psi(G) = \psi(G_1)\psi(G_2)$.  

\begin{example}[Complete (nonhyper) graphs] \[S_{K_n} = \bigg(1 - \frac{1}{p^{n + \binom{n}{ 2} }}\bigg)^{-1} \sum_{W \subsetneq V} \frac{1}{p^{|W| + |E^W|}} S_{G^W} = \bigg(1 - \frac{1}{p^{n + \binom{n}{2} }}\bigg)^{-1} \bigg(1+ \sum_{i=1}^{n-1} \binom{n}{i} p^{-\frac{i(i+1)}{2} } S_{K_i} \bigg) \] 
\begin{align*}
S_{K_1} &= \frac{1}{1-q} \\
S_{K_2} &= \frac{1+q}{(1-q)(1-q^3)} \\
S_{K_3} &= \frac{1+2q+2q^3+q^4}{(1-q)(1-q^3)(1-q^6)} \\
S_{K_4} &= \frac{1+3q+5q^3+3q^4+3q^6+5q^7+3q^9 + q^{10}}{(1-q)(1-q^3)(1-q^6)(1-q^{10})} 
\end{align*}

\end{example}

\begin{example}[Hypergraphs $G_k \in \mathscr{B}_k^{1}$ ] 
\[ S_{G_k} = \bigg(1 - \frac{1}{p^{2k-1}}\bigg)^{-1} \bigg(1+ \sum_{i=1}^{k-1} \binom{k}{i} \frac{1}{p^{2i -1}} S_{G_i} \bigg) \] \begin{align*}
S_{G_1} &= \frac{1}{1-q} \\
S_{G_2} &= \frac{1+q}{(1-q)(1-q^3)} \\
S_{G_3} &= \frac{1+2q+2q^3+q^4}{(1-q)(1-q^3)(1-q^5)} \\
S_{G_4} &= \frac{1+3q+5q^3+3q^4+3q^5+5q^6+3q^8+q^9}{(1-q)(1-q^3)(1-q^5)(1-q^7)} 
\end{align*}
\end{example} 

\begin{example}[$H_1 \cup H_2$ where $H_1 \in \mathscr{B}_{h_1}^{1}$ and $H_2 \in \mathscr{B}_{h_2}^{1}$]  \label{2together}
Any such hypergraph $G = H_1 \cup H_2$ is characterized by three numbers: $h_1, h_2,$ and $h_3 := |V(H_1) \cap V(H_2)|$.  We denote the corresponding sum $S_G$ by $\displaystyle{S_{h_1, h_2, h_3}}$.  Note that if the vertex sets $V(H_1)$ and $V(H_2)$ are disjoint, $\displaystyle{S_{h_1,h_2,0} = S_{H_1} S_{H_2}}$.  
\begin{align*}
S_{1, 1, 0} &= (S_{G_1})^2 = \frac{1}{(1-q)^2} \\
S_{2, 1, 0} &= (S_{G_2})(S_{G_1}) = \frac{1+q}{(1-q)^2(1-q^3)} \\
S_{2, 1, 1} &= S_{G_2} = \frac{1+q}{(1-q)(1-q^3)} \\
S_{2, 2, 0} &= (S_{G_2})^2 = \frac{(1+q)^2}{(1-q)^2(1-q^3)^2} \\
S_{2, 2, 1} &= \frac{1}{1-q^5}(1+3q S_{G_1}+2 q^3 S_{G_2} + q^2 (S_{G_1})^2) \\
&= \frac{1}{1-q^5} \bigg(1+3\frac{q}{1-q} + 2 q^3 \frac{1+q}{(1-q)(1-q^3)} + \frac{q^2}{(1-q)^2} \bigg) \\
S_{2,2,2} &= \frac{1}{1-q^4}(1+ 2q S_{G_1}) = \frac{1-q^2}{(1-q)^2(1-q^4)} \\
\end{align*}
This process can be continued as far as desired.
\end{example}

Out of curiosity, we record the first few terms of the power series $P^1_2(y_1,y_2)$ when $k=2$.


\begin{corollary}
For $y_2 - y_1 \le 2^2 = 4$, we have 
\begin{equation}
P^1_2(y_1,y_2) \approx 1 - 0.18269(y_2-y_1) + 0.01448(y_2-y_1)^2 + O((y_2-y_1)^3) 
\end{equation}
\end{corollary}
\begin{proof}

Applying Lemma \ref{L:powerseriesformula}, we first compute $\psi(G)$ for $G \in \mathscr{A}_2^{1}$ and $G \in \mathscr{A}_2^{2}$.  For the linear term of the power series, there is only one graph $G_2 \in \mathscr{A}_2^{1}$ and \[\psi(G_2) = \prod_{p} (1-q)^2 \frac{1+q}{(1-q)(1-q)^3} = \prod_{p} \frac{1-q^2}{1-q^3} = \frac{\zeta(3)}{\zeta(2)} \]

For the quadratic term, there are 2 nonisomorphic graphs $H_1$ and $H_2$ (ignoring labelling) in $\mathscr{A}_2^{2}$.  These give 
\begin{align*}
\psi(H_1) = \prod_p (1-q)^3 S_{2,2,1} & = \prod_p (1-q)^3  \frac{1}{1-q^5} \bigg(1+3\frac{q}{1-q} +2 q^3 \frac{1+q}{(1-q)(1-q^3)} + \frac{q^2}{(1-q)^2} \bigg) \\
& = \prod_p \frac{(1-q)^2(1+2q+q^2+2q^3+q^4)}{(1-q^3)(1-q^5)} \approx 0.561356  
\end{align*}
\hspace{2 mm} $\displaystyle{ \psi(H_2) = (\psi(G_2))^2 = \left(\frac{\zeta(3)}{\zeta(2)}\right)^2 \approx 0.534015 }$ \vspace{12pt}

Now let us compute $\displaystyle{ \sum_{H \cong^l G } \mathbb{E}[L^H]}$ for $G$ in $\mathscr{A}_2^1$ or $\mathscr{A}_2^2$ using \eqref{LHformula}. 

For the linear term, \[\sum_{H \cong^l G_2 } \mathbb{E}[L^H] = \sum_{i_1 < i_2} \mathbb{E}[L_{i_1} L_{i_2}] = \left(\frac{1}{2!}\right)^2 \]

For the quadratic term, 
\begin{align*}
\sum_{i_1 < i_2 < i_3} \mathbb{E}[L_{i_1}^2 L_{i_2} L_{i_3}] &= \frac{11}{864} \\ 
\sum_{i_1 < i_2 < i_3} \mathbb{E}[L_{i_1} L_{i_2}^2 L_{i_3}] &= \frac{5}{864} \\
\sum_{i_1 < i_2 < i_3} \mathbb{E}[L_{i_1} L_{i_2} L_{i_3}^2] &= \frac{1}{432} \\
\sum_{i_1 < i_2 < i_3 < i_4} \mathbb{E}[L_{i_1} L_{i_2} L_{i_3} L_{i_4}] &= \left(\frac{1}{4!}\right)^2 
\end{align*}

Putting this together, the coefficient in front of the $(y_1-y_2)$ term is $\displaystyle{\frac{\zeta(3)}{4 \zeta(2)} \approx 0.18269}$. 

The coefficient in front of the $(y_1-y_2)^2$ term is
\begin{align*}
\sum_{G \in \mathscr{A}_2^{2}} \psi(G) \sum_{H \cong^l G} \mathbb{E}[L^H] &= \psi(H_1)\left(\frac{11}{864} + \frac{5}{864} + \frac{1}{432} \right) + 3\psi(H_2)\left(\frac{1}{4!}\right)^2 \\
&\approx 0.561356\left(\frac{11}{864} + \frac{5}{864} + \frac{1}{432}\right) + 3 \cdot 0.534015 \frac{1}{576} \approx 0.01448 
\end{align*}

\end{proof}

\section*{Acknowledgements}
The author wishes to thank his PhD advisor Steven Evans for helpful discussions and comments on this work. 

\bibliography{references}
\bibliographystyle{plain}


\end{document}